\documentclass[11pt,leqno]{amsart}
\usepackage{epsfig}
\usepackage{amssymb}
\usepackage{amscd}
\usepackage{amsaddr}
\usepackage[all,cmtip,matrix,arrow]{xy}
\usepackage{graphicx}
\usepackage{mathabx}
\usepackage{skak}

\newtheorem{theo}{Theorem}[section]
\newtheorem{lemma}[theo]{Lemma}
\newtheorem{defi}[theo]{Definition}
\newtheorem{prop}[theo]{Proposition}
\newtheorem{conj}[theo]{Conjecture}
\newtheorem{cor}[theo]{Corollary}
\newtheorem{remark}[theo]{Remark}

\numberwithin{equation}{section}

\mathchardef\mhyphen="2D

\def\A{{\mathbb A}}

\def\bL{\mathbb{L}}

\def\Z{\mathbb{Z}}

\def\coh{\operatorname{coh}}

\def\bD{{\mathbf D}}
\def\bR{{\mathbf R}}
\def\bL{{\mathbf L}}

\def\pre-tr{\operatorname{pre-tr}}

\def\Hom{\operatorname{Hom}}

\newcommand{\tens}[1]{%
  \mathbin{\mathop{\otimes}\displaylimits_{#1}}%
}
\newcommand{\xto}{\xrightarrow}
\newcommand{\xlto}{\xleftarrow}
\newcommand{\hto}{\hookrightarrow}
\newcommand{\leto}{\leftarrow}

\newcommand{\QCoh}{\operatorname{QCoh}}
\newcommand{\Coh}{\operatorname{Coh}}

\newcommand{\hproj}{\operatorname{h-proj}}
\newcommand{\Acycl}{\operatorname{Acycl}}
\newcommand{\Inj}{\operatorname{Inj}}
\newcommand{\Flasque}{\operatorname{Flasque}}

\newcommand{\veps}{\varepsilon}

\newcommand{\mk}{\mathrm k}

\newcommand{\cJ}{{\mathcal J}}

\newcommand{\cF}{{\mathcal F}}
\newcommand{\cG}{{\mathcal G}}
\newcommand{\cO}{{\mathcal O}}

\newcommand{\cM}{{\mathcal M}}

\newcommand{\cD}{{\mathcal D}}

\newcommand{\cA}{{\mathcal A}}
\newcommand{\cB}{{\mathcal B}}
\newcommand{\cI}{{\mathcal I}}
\newcommand{\cC}{{\mathcal C}}
\newcommand{\cE}{{\mathcal E}}

\newcommand{\cS}{{\mathcal S}}
\newcommand{\cT}{{\mathcal T}}
\newcommand{\cK}{{\mathcal K}}
\newcommand{\cH}{{\mathcal H}}

\newcommand{\cHom}{\mathcal Hom}

\newcommand{\mD}{{\mathfrak D}}

\newcommand{\un}{\underline}

\newcommand{\Tot}{\operatorname{Tot}}

\newcommand{\Perf}{\operatorname{Perf}}
\newcommand{\perf}{\operatorname{perf}}

\newcommand{\supp}{\operatorname{Supp}}

\newcommand{\coker}{\operatorname{Coker}}

\newcommand{\im}{\operatorname{Im}}

\newcommand{\colim}{\operatorname{colim}}

\newcommand{\Mod}{\operatorname{Mod}}
\newcommand{\Dif}{\operatorname{Dif}}

\newcommand{\even}{\operatorname{even}}
\newcommand{\odd}{\operatorname{odd}}

\newcommand{\Spec}{\operatorname{Spec}\,}

\newcommand{\Ho}{\operatorname{Ho}}

\newcommand{\id}{\operatorname{id}}

\newcommand{\dgalg}{\operatorname{dgalg}}
\newcommand{\dgcat}{\operatorname{dgcat}}

\newcommand{\pr}{\operatorname{pr}}

\newcommand{\com}{\operatorname{Com}}

\usepackage{epsf}
\usepackage{amscd}

\newcommand{\one}{\mathbf{1}}

\title[Homotopy finiteness of some DG categories from algebraic geometry]
{Homotopy finiteness of some DG categories from algebraic geometry}

\author{Alexander I. Efimov}
\address{Steklov Mathematical Institute of RAS, Gubkin str. 8, GSP-1, Moscow 119991, Russia\\
National Research University Higher School of Economics, Russian Federation}
\email{efimov@mccme.ru}
\thanks{The author is partially supported by Laboratory of Mirror Symmetry NRU HSE, RF government  grant, ag. N 14.641.31.0001}

\begin{document}

\begin{abstract}In this paper, we prove that the bounded derived category $D^b_{coh}(Y)$ of coherent sheaves on a separated scheme $Y$ of finite type over a field $\mk$
of characteristic zero is homotopically finitely presented. This confirms a conjecture of Kontsevich. We actually prove a stronger statement: $D^b_{coh}(Y)$ is equivalent to a DG quotient $D^b_{coh}(\tilde{Y})/T,$ where $\tilde{Y}$ is some smooth and proper variety, and the subcategory $T$ is generated by a single object.

The proof uses categorical resolution of singularities of Kuznetsov and Lunts \cite{KL}, and a theorem of Orlov \cite{Or} stating that the class of geometric smooth and proper DG categories is stable under gluing. 


We also prove the analogous result for $\Z/2$-graded DG categories of coherent matrix factorizations on such schemes. In this case instead of $D^b_{coh}(\tilde{Y})$ we have a semi-orthogonal gluing of a finite number of DG categories of matrix factorizations on smooth varieties, proper over $\A_{\mk}^1$.
\end{abstract}


\maketitle

\tableofcontents

\section{Introduction}

According to one of the approaches to noncommutative algebraic geometry, a non-commutative space is treated as a triangulated DG (or $A_{\infty}$-) category which admits a single generator \cite{KS, Or}.

By a theorem of Bondal and Van den Bergh \cite[Theorem 3.1.1]{BvdB} and by the results of Keller \cite{Ke}, for any separated scheme $X$ of finite type over a field $\mk,$
there is a DG $\mk$-algebra $A,$ defined up to a Morita equivalence, such that $D(\QCoh X)\simeq D(A).$  Here 
$D(A)$ is the derived category of (right) DG $A$-modules. This equivalence identifies the subcategories of perfect complexes (which are exactly the compact objects): $D_{perf}(X)\simeq D_{perf}(A).$ Let us denote by $\Perf(X)$ the DG enhancement (see \cite{BK}) of the triangulated category $D_{perf}(X).$ The DG category $\Perf(X)$ is treated as a non-commutative space associated with $X.$

We recall the notions of smoothness and properness for DG categories.

\begin{defi}\cite{KS} 1) A small DG category $\cC$ over $\mk$ is smooth if the diagonal $\cC\mhyphen\cC$-bimodule is perfect.

2) $\cC$ is called proper if for $X,Y\in\cC$ the complex $\cC(X,Y)$ is perfect over $\mk.$\end{defi}

It is known (\cite[Proposition 3.30]{Or}, \cite[Proposition 3.13]{L}) that $X$ is smooth (resp. proper) if and only if the DG category $\Perf(X)$ is smooth (resp. proper). Hence, these basic geometric properties of $X$ are reflected by the DG category $\Perf(X).$

In this paper we will be interested in the larger DG category $D^b_{coh}(X)\supset\Perf(X)$ --- an enhancement of the derived category of coherent sheaves. For this category the situation is quite different. Namely, the following theorem has been proven by V.~Lunts.

\begin{theo}\cite[Theorem 6.3]{L} Let $X$ be a separated scheme of finite type over a perfect field $\mk.$
Then the DG category $D^b_{coh}(X)$ is smooth.\end{theo}

It is quite surprising: the scheme $X$ can have arbitrary singularities and even be non-reduced, but the DG category $D^b_{coh}(X)$ is always smooth.

The class of smooth DG categories contains some "large" examples. For example, the field of rational functions $\mk(x_1,\dots,x_n)$ is a smooth DG algebra. It is natural to try to impose some conditions on a smooth DG category to be of finite type in an appropriate sense. 

To\"en and Vaqui\'e \cite{TV} introduced the class of so-called homotopically finitely presented (hfp) DG categories.

\begin{defi}\cite{TV} 1) A DG algebra $A$ is hfp if in the homotopy category of DG algebras $A$ is a retract of a free graded algebra $\mk\langle x_1,\dots,x_n\rangle$ with differential satisfying
$$d x_i\in \mk\langle x_1,\dots,x_{i-1}\rangle,\quad 1\leq i\leq n.$$

2) A small DG category is hfp if it is Morita equivalent to a DG algebra which is hfp.\end{defi}

See Section \ref{sec:Ho_DG_cat} for a more detailed discussion, in particular for the notion of a retract.

The hfp DG categories play the same role in the Morita homotopy category of DG categories as perfect $\cA$-modules in the derived category of all $\cA$-modules (where $\cA$ is some small DG category). Also, their analogue in the homotopy category of CW complexes is the category of the so-called finitely dominated spaces, which are homotopy retracts of finite CW complexes, see \cite{Wh1, Wh2, W}.

The basic facts about hfp DG categories are recalled in Section \ref{sec:Ho_DG_cat}. Here we mention that if a DG category is hfp, then it is smooth. On the other hand, if a DG category is smooth and proper, then it is hfp. 





Our main result is the following.

\begin{theo}\label{intro:hfp_coh}Let $Y$ be a separated scheme of finite type over a field $\mk$ of characteristic zero. Then the DG category $D^b_{coh}(Y)$ is hfp.\end{theo}

The statement of this theorem has been previously conjectured by Kontsevich \cite{Ko}.

We also prove a similar result for coherent matrix factorizations \cite{EP}. For any regular function $W$ on $Y$ we have a $\Z/2$-graded DG category $D^{abs}_{coh}(X,W)$ - an enhancement of the absolute derived category of coherent matrix factorizations of $W.$

\begin{theo}\label{intro:hfp_MF} Let $Y$ be a separated scheme of finite type over a field $\mk$ of characteristic zero, and $W$ a regular function on $Y.$ Then the $\Z/2$-graded DG category $D^{abs}_{coh}(Y,W)$ is hfp.\end{theo}


\begin{remark}\label{remark:enhancements_intro} We should  stress that although the triangulated categories $D^b_{coh}(X)$ and $D^{abs}_{coh}(X,W)$ are not known to have a unique enhancement, they both have a natural choice of enhancement. 

Namely, for any small $\mk$-linear abelian category $\cC,$ the category $D^b(\cC)$ can be enhanced by the DG quotient $\com^b(\cC)/\Acycl(\cC),$ where $\com^b(\cC)$ is the DG category of bounded complexes of objects of $\cC$ (with standard $\Hom$-complexes), and $\Acycl(\cC)$ is the full DG subcategory of acyclic complexes. This in particular gives a natural enhancement for $D^b_{coh}(X).$ 

Similarly, there is a natural enhancement of the category $D^{abs}(\cC,W)$ for any small abelian category $\cC$ and an element $W\in Z(\cC)$ of the center of $\cC,$ see Remark \ref{remark:enhancement_for_D^abs}. This gives a natural enhancement for $D^{abs}_{coh}(X,W).$

Remarkable general results on the uniqueness of enhancements have been obtained in the papers \cite{LO, CS}.\end{remark}

There is a particularly nice class of homotopically finite DG categories: those which admit a so-called smooth categorical compactification.

\begin{defi}\label{def:intro_smooth_cat_comp} A smooth categorical compactification of a DG category $\cA$ is a DG quasi-functor $F:\cC\to\cA,$ where the DG category $\cC$ is smooth and proper, the extensions of scalars functor $F^*:\Perf(\cC)\to\Perf(\cA)$ is a localization (up to direct summands), and its kernel is generated by a single object.\end{defi}

The motivation for the term "smooth categorical compactification" is the following. Suppose that $Y$ is smooth, and $\overline{Y}\supset Y$ a usual (algebro-geometric) smooth compactification. Then the restriction functor $D^b_{coh}(\overline{Y})\to D^b_{coh}(Y)$ is a smooth categorical compactification.

One can show that the existence of a smooth categorical compactification implies the homotopy finiteness (see Corollary \ref{cor:sm_comp->hfp}).

We prove a result which is stronger than Theorems \ref{intro:hfp_coh}, \ref{intro:hfp_MF}.

\begin{theo}\label{intro:smooth_compactification}Let $Y$ be a separated scheme of finite type over a field $\mk$ of characteristic zero. Then

1) the DG category $D^b_{coh}(Y)$ has a smooth categorical compactification of the form $D^b_{coh}(\tilde{Y})\to D^b_{coh}(Y),$ where $\tilde{Y}$ is a smooth and proper variety. 

2) for any regular function $W\in\cO(Y)$ the D($\Z/2$-)G category $D^{abs}(Y,W)$ has a $\Z/2$-graded smooth categorical compactification $C_W\to D^{abs}(X,W),$
with a semi-orthogonal decomposition $C_W=\langle D^{abs}(V_1,W_1),\dots,D^{abs}(V_m,W_m)\rangle,$ where each $V_i$ is a $\mk$-smooth variety and the morphisms $W_i:V_i\to\A_{\mk}^1$ are proper.\end{theo}

The general idea of the proof of Theorem \ref{intro:smooth_compactification} is motivated by the following conjecture of Bondal and Orlov.

\begin{conj}\label{conj:BO}\cite{BO} Let $Y$ be a variety with rational singularities, and $f:X\to Y$ a resolution of singularities (recall that rationality of singularities means $\bR f_*\cO_X\cong\cO_Y$). Then the functor
$\bR f_*:D^b_{coh}(X)\to D^b_{coh}(Y)$ is a localization.\end{conj}

If we are able to prove Conjecture \ref{conj:BO} and also to show that the kernel is generated by a single object, then choosing any smooth compactification $\overline{X}$ of $X$ we get a smooth categorical compactification $D^b_{coh}(\overline{X})$ of $D^b_{coh}(Y).$ Unfortunately, we are not able to prove Conjecture \ref{conj:BO} in general, although we are able to prove it in some class of cases, which we mention here.

\begin{theo}\label{th:intro_BO_special} Suppose that $Y$ has rational singularities, $Z\subset Y$ is a closed smooth subscheme, and $X=Bl_ZY$ is smooth, so that $f:X\to Y$ is a resolution of singularities. Let us denote by $T=f^{-1}(Z)$ the exceptional divisor, and by $p:T\to Z$ the induced morphism, and by $j:T\to X$ the embedding. Suppose that $\bR f_* I_T^n=I_Z^n$ for $n\geq 1.$ Then the functor $\bR f_*:D^b_{coh}(X)\to D^b_{coh}(Y)$ is a localization, and its kernel is generated by $j_*((p^*D^b_{coh}(Z))^{\perp}).$\end{theo}

A more general version of this result is Theorem \ref{th:BO_special_case}. A special case of the setting of Theorem \ref{th:BO_special_case} is when $Y\subset\A^m$ is a cone over some projective embedding of a smooth Fano variety, and $Z=\{0\}\subset Y$ is the origin.


For an arbitrary scheme $Y$ (separated, of finite type) the idea is to use the so-called categorical resolution, constructed by Kuznetsov and Lunts \cite{KL}. It plays the same role as the derived category of the resolution of rational singularities, and it exists for any separated scheme of finite type. Surprisingly, in this framework we are able to prove the analogue of Conjecture \ref{conj:BO}, which allows us to prove Theorem \ref{intro:smooth_compactification}. We note that even if $Y$ has rational singularities, we still use the categorical resolution to obtain the smooth categorical compactification.

We now describe the structure of the proof, concentrating on the part 1) (the case of matrix factorizations is analogous, although it is more technically involved).

We assume that $Y$ is proper (if not then we can replace $Y$ by its compactification). Following the same steps as \cite{KL}, we construct a smooth and proper DG category $C,$ with a DG quasi-functor $C\to D^b_{coh}(Y),$ so that $C=\langle D^b_{coh}(X_1),\dots,D^b_{coh}(X_m)\rangle,$ where each $X_i$ is a smooth and proper variety over $\mk.$ It suffices to prove that $C\to D^b_{coh}(Y)$ is a localization, and its kernel is generated by a single object. Then we can apply the result of Orlov \cite[Theorem 4.15]{Or} which states that the class of admissible subcategories of derived categories of smooth and proper varieties is closed under gluing via perfect bimodules. This allows us to replace $C$ by a derived category of some smooth and proper variety.  

The construction of the DG category $C$ is inductive: we choose a sequence of blow-ups along smooth centers
$$Y_n\to Y_{n-1}\to\dots\to Y_1\to Y,$$ such that $(Y_n)_{red}$ is smooth, and proceed by induction on $n.$

For $n=0,$ the scheme $Y_{red}$ is smooth, and the categorical compactification is given by a ringed space
$(Y,\cA_Y)$ (obtained by "Auslander-type construction"), together with a morphism $\rho_Y:(Y,\cA_Y)\to Y,$ see subsection \ref{ssec:A-construction_coh}. Namely, in this case
the direct image functor $\rho_{Y*}:D^b_{coh}(\cA_Y)\to D^b_{coh}(Y)$ is a smooth categorical compactification, and we have a semi-orthogonal decomposition
$$D^b_{coh}(\cA_{Y})=\langle D^b_{coh}(Y_{red}),\dots,D^b_{coh}(Y_{red}))\rangle,$$
where the number of copies is the smallest number $m$ such that $I^m=0,$ where $I\subset \cO_Y$ is the nilpotent radical.

Another important ingredient is the "categorical blow-up", see Subsection \ref{ssec:cat_blowup_sheaves}. Namely, if $X$ is a blow-up of $Y$ along $S\subset Y,$ then under some assumptions on $S$ (which are always achieved by replacing it by a sufficiently large infinitesimal neighborhood), we can glue (semi-orthogonally) the DG categories $D^b_{coh}(X)$ and $D^b_{coh}(S)$ to obtain a DG category $\cD_{coh}(X,S)$ with a localization DG functor \begin{equation}\label{eq:pi_*_cat_blow_up}\pi_*:\cD_{coh}(X,S)\to D^b_{coh}(Y).\end{equation} Proving that $\pi_*$ is a localization and controlling the kernel is the most difficult part of the proof, see Theorem \ref{th:main_localization}.

After this is done, we present the construction of a smooth categorical compactification (Subsection \ref{ssec:constr_of_sm_comp}). 
It is a combination of Auslander-type constructions and modified categorical blow-ups. The latter were also studied in \cite{KK}, with a more conceptual description. Compared to \cite{KL}, we need some restriction on the choice of integer parameters, see Remark \ref{rem:choice_of_parameters}.

It is a challenging question if one can drop the assumption of characteristic zero in Theorems \ref{intro:hfp_coh}, \ref{intro:hfp_MF}, \ref{intro:smooth_compactification} (e.g. using some categorical modifications of de Jong's alterations).


Another interesting subject is to consider coherent D-modules instead of coherent sheaves. The following conjecture was suggested to me by V.~Drinfeld.

\begin{conj}If $X$ is a separated scheme of finite type over a field $\mk$ of characteristic zero, then the bounded derived category $D^b(\Coh\mhyphen\cD_X)$ of coherent $\cD$-modules on $X$ (more precisely, its DG enhancement) has a smooth categorical compactification.\end{conj}

In the forthcoming paper \cite{E} we will prove that at least the DG category $D^b(\Coh\mhyphen\cD_X)$ is homotopically finite, by completely different methods.

The paper is organized as follows.

In Section \ref{sec:Ho_DG_cat} we discuss the notions of homotopy finiteness of DG categories and smooth categorical compactifications, and also
formulate some basic results related to this notions.

In Section \ref{sec:localizations} we recall the Neeman's criterion for a functor to be a localization, and introduce homological epimorphisms of DG categories (in the terminology of \cite{GL}, \cite{Pau}) which generalize localizations.

In Section \ref{sec:gluing_conceptual} we recall the notion of gluing of DG categories via a bimodule. We also recall the relation of gluing with the so-called upper-triangular DG categories \cite{Tab3}.  

In Section \ref{sec:gluing} we prove various technical results, which we need in the sequel.

In Section \ref{sec:exotic_derived} we recall the notion of coderived category and absolute derived category. We formulate basic results for locally noetherian abelian categories, in particular, about compact generation of coderived categories.

Section \ref{sec:coh_and_MF} is devoted to specific convenient enhancements for derived categories of coherent sheaves and absolute derived categories of coherent matrix factorizations. For these enhancements, we have natural DG functors (not just quasi-functors) of direct image for a proper morphism. 

In Subsection \ref{ssec:coh_MF} we also prove the analogues of theorems of Rouquier and Lunts for matrix factorizations on separated schemes of finite type over a field. Namely, we prove the existence of strong generators in the absolute derived categories of coherent matrix factorizations (Theorem \ref{th:strong_generator_MF}), and smoothness of their enhancements under some necessary assumptions (Theorem \ref{th:smoothness_MF}).

In Subsection \ref{ssec:nice_ringed_spaces} we introduce the category of nice ringed spaces. Its objects are pairs $(X,\cA_X),$ where $X$ is a separated noetherian scheme and $\cA_X$ is a coherent sheaf of $\cO_X$-algebras,
satisfying some additional property. We discuss the (co)derived categories of (quasi-)coherent sheaves on nice ringed spaces, and functors between them. In particular, we show that the category $\QCoh(\cA_X)$ of quasi-coherent $\cA_X$-modules and the category $\Mod\mhyphen\cA_X$ of all $\cA_X$-modules are locally noetherian, and an object of  $\QCoh(\cA_X)$ is injective iff it is injective in $\Mod\mhyphen\cA_X.$

Section \ref{sec:sm_cat_comp} is devoted to the proof of Theorem \ref{intro:smooth_compactification} (see Theorem \ref{th:smooth_compactification}).

In Subsection \ref{ssec:A-construction_coh} we study Auslander-type construction from \cite{KL}. This is a nice ringed space $(S,\cA_S)$ constructed from a triple $(S,\tau,n)$ where $S$ is a scheme,
$\tau$ a coherent ideal sheaf and $n$ is a positive integer such that $\tau^n=0.$ We study various properties of the category $D^b_{coh}(\cA_S)$ which we use in the sequel: the localization functor to $D^b_{coh}(S),$ the semi-orthogonal decomposition and its compatibility with direct images.

Subsection \ref{ssec:A-construction_MF} is devoted to studying categories of matrix factorizations on the same nice ringed spaces, and proving similar results for them.

Subsection \ref{ssec:cat_blowup_sheaves} is devoted to proving that the functor \eqref{eq:pi_*_cat_blow_up} mentioned above is a localization. The main result of this section is Theorem \ref{th:main_localization}, and the localization is proved in a more general context, without even mentioning blow-ups. The proof is quite involved. It is worth mentioning that the proof uses the Auslander-type construction, although Theorem \ref{th:main_localization} is formulated without mentioning any nice ringed spaces which are not schemes.

In Subsection \ref{ssec:cat_blowup_MF} we prove a similar result for matrix factorizations (Theorem \ref{th:main_localizationMF}). The proof follows the same lines.

Subsection \ref{ssec:constr_of_sm_comp} is devoted to the proof of the main result (Theorem \ref{th:smooth_compactification}) on the existence of a smooth compactification for $D^b_{coh}(Y)$ and $D^{abs}(Y,W).$

Appendix \ref{app:MF} is devoted to studying totalizations of complexes of matrix factorizations from the functorial point of view. Namely, we show
that on suitable kinds of derived categories the totalization is a localization. This is quite useful and natural fact, which essentially allows to reduce various statements on absolute derived (and coderived) categories of matrix factorizations in abelian categories to similar statements about the usual (co)derived categories.

In Appendix \ref{app:duality} we prove Grothendieck duality for quasi-coherent sheaves and quasi-coherent matrix factorizations on
nice ringed spaces, using the technique from Appendix \ref{app:MF}.


{\noindent{\bf Acknowledgements.}} I am grateful to Alexander Beilinson, Alexei Bondal, Vladimir Drinfeld, Dmitry Kaledin, Maxim Kontsevich, Alexander Kuznetsov, Valery Lunts, Dmitri Orlov and Leonid Positselski for useful discussions. I am also grateful to an anonymous referee for a number of suggestions, corrections and useful remarks.

\section{Homotopy theory of DG algebras and DG categories}
\label{sec:Ho_DG_cat}

For the introduction on DG categories, we refer the reader to \cite{Ke}. Our basic reference for model categories is \cite{Ho}. The references for model structures on DG algebras and DG categories are \cite{Tab1, Tab2}.
The notion of homotopy finiteness is taken from \cite{TV}. The references for DG quotients are \cite{D, Ke2}.

We fix some base field $\mk.$ We will consider either $\Z$-graded or $\Z/2$-graded DG categories. The latter can be treated as DG categories over $\mk[u^{\pm 1}],$
where $u$ has degree $2.$ These two cases are parallel for our discussion. If we do not specify the grading, we mean that everything holds in both frameworks. We write $-\otimes-$ for $-\otimes_{\mk}-.$ Also, for a homogeneous element $v$ of a graded vector space $V,$ we denote by $|v|$ its grading.

All DG modules are assumed to be right unless otherwise stated. Given a small DG category $\cA,$ we denote by $\Mod\mhyphen\cA$ the DG category of right DG modules (it is denoted by $\Dif \cA$ in \cite[Section 1.2]{Ke}). We also denote by $\cA\mhyphen \Mod=\Mod\mhyphen\cA^{op}$ the DG category of left $\cA$-modules. We have a fully faithful Yoneda embedding functor $\cA\to \Mod\mhyphen\cA.$ For any DG category $T$ (not necessarily small) the $\mk$-linear category $H^0(T)$ has the same objects as $T,$ and the morphisms are given by $H^0(T)(X,Y)=H^0(T(X,Y)).$ It is shown in \cite[Lemma 2.2]{Ke} that the category $H^0(Mod\mhyphen\cA)$ is naturally triangulated. The derived category $D(\cA)$ is defined to be the Verdier quotient of $H^0(Mod\mhyphen\cA)$ by the full triangulated subcategory of acyclic DG modules.

It is also convenient to define the category $Z^0(T)$ for any DG category $T,$ similarly to $T.$ Here for a complex $\cK^{\bullet}$ of vector spaces we denote by $Z^0(\cK^{\bullet})$ the vector space of closed elements of degree zero.

By the results of \cite[Section 3]{Ke}, the full subcategory $H^0(\Acycl(\cA))$ 
of acyclic DG modules in $H^0(\Mod\mhyphen\cA)$ is both left and right admissible. Recall that an $\cA$-module $M$ is called {\it h-projective} (resp. {\it h-injective}) if it is in the left (resp. right) orthogonal to $H^0(\Acycl\mhyphen\cA).$ We denote by $\hproj(\cA)\subset \Mod\mhyphen\cA$ the full DG subcategory of acyclic $\cA$-modules. In particular, we have an equivalence $D(\cA)\simeq H^0(\hproj(\cA)).$ This allows to define the left derived functor $\bL F:D(\cA)\to D(\cB)$ of any exact functor $F:H^0(\Mod\mhyphen\cA)\to H^0(\Mod\mhyphen\cB)$ to be the composition $$D(\cA)\xto{\sim}H^0(\hproj(\cA))\xto{F} H^0(\Mod\mhyphen\cB)\to D(\cB).$$

The tensor product bifunctor
$$-\otimes_{\cA}-:\Mod\mhyphen\cA\otimes\cA\mhyphen\Mod$$ is given by
$$M\otimes_{\cA}N=\coker(\bigoplus\limits_{X,Y\in\cA}M(Y)\otimes \cA(X,Y)\otimes N(Y)\xto{\nu} \bigoplus\limits_{X\in\cA}M(X)\otimes N(X)),$$
where $\nu(m\otimes f\otimes n)=mf\otimes n-m\otimes fn.$

Given small DG categories $\cA,$ $\cB,$ we denote by $\cA\mhyphen\Mod\mhyphen\cB$ the DG category $\Mod\mhyphen(\cA^{op}\otimes\cB)$ of $\cA\mhyphen\cB$-bimodules. Then an $\cA$-$\cB$-bimodule $N$ defines a DG functor
$$-\tens{\cA}N:\Mod\mhyphen\cA\to\Mod\mhyphen\cB,$$ given by
$$(M\otimes_{\cA}N)(X)=M\otimes_{\cA}N(-,X).$$

This DG functor induces an exact functor $-\otimes_{\cA}N:H^0(\Mod\mhyphen\cA)\to H^0(\Mod\mhyphen\cA).$ We denote by $-\stackrel{\bL}{\otimes}_{\cA}N:D(\cA)\to D(\cB)$ the left derived functor. 

Similarly, given an $\cA\mhyphen\cB$-bimodule $M$ and a $\cB\mhyphen\cC$-bimodule $N,$ their tensor product $M\otimes_{\cB}N\in\cA\mhyphen\Mod\mhyphen\cC$ is given by
$$(M\otimes_{\cB}N)(X,Y)=M(X,-)\otimes_{\cB}N(-,Y).$$ Deriving the resulting bifunctor on either side gives the same biexact bifunctor
$$D(\cA^{op}\otimes\cB)\times D(\cB^{op}\otimes\cC)\to D(\cA^{op}\otimes\cC).$$


We denote by $\cS\cF_{fg}(\cA)\subset\Mod\mhyphen\cA$ the full DG subcategory of {\it semi-free finitely generated} modules. That is, a module $M$ is in $\cS\cF_{fg}(\cA)$ if it has a finite filtration by DG submodules, such that all the subquotients are isomorphic to the shifts of representable DG modules. In particular, all representable $\cA$-modules are in $\cS\cF_{f.g}(\cA).$ In fact, $\cS\cF_{fg}(\cA)\subset \hproj(\cA),$ and the category $H^0(\cS\cF_{fg}(\cA))$ is identified with the full triangulated subcategory of $D(\cA)$ generated by representable modules via shifts, cones and finite direct sums. We recall that a DG category $\cA$ is called weakly (strongly) pre-triangulated if the Yoneda functor $\cA\to\cF_{fg}(\cA)$ is a quasi-equivalence (resp. a DG equivalence). In particular, for a weakly pre-triangulated category, the category $H^0(\cA)$ is triangulated. By definition \cite{BK}, an enhancement of a triangulated category $\cT$ is a weakly pre-triangulated DG category $\cA,$ together with an exact equivalence $H^0(\cA)\simeq T.$

We recall that the triangulated subcategory $D_{\perf}(\cA)\subset D(\cA)$ is defined to be the Karoubi closure of $H^0(\cS\cF_{fg}(\cA))$ inside $D(\cA).$ In fact, the triangulated category $D(\cA)$ is compactly generated (see \cite[Section 4.2]{Ke}) and the subcategory $D(\cA)^c$ of compact objects coincides with $D_{\perf}(\cA)$ (see \cite{Ra}, \cite[Lemma 2.2]{Ne1} and \cite[Theorem 5.3]{Ke}).

We denote by $\dgalg_{\mk}$ the category of DG algebras over $\mk.$
By \cite{J}, it has a model structure, with weak equivalences being quasi-isomorphisms and fibrations being
surjections. This model category is finitely generated in the terminology of \cite{Ho}. Its finite cell objects are the following.

\begin{defi}{\cite{TV}} A finite cell DG algebra $B$ is a DG algebra which is isomorphic as a graded algebra to
the free algebra of finite type:
$$B^{gr}\cong \mk\langle x_1,\dots,x_n\rangle,$$
and moreover we have $$dx_i\in \mk\langle x_1,\dots,x_{i-1}\rangle,\quad 1\leq i\leq n.$$\end{defi}

We recall that for a category $\cC,$ an object $X\in\cC$ is called a {\it retract} of an object $Y\in\cC,$ if there exist morphisms $f:X\to Y$ and $g:Y\to X,$ such that $gf=\id_X.$

The following definition is due to To\"en and Vaqui\'e \cite{TV}. It makes sense for all finitely generated model categories.

\begin{defi}{\cite{TV}} A DG algebra $A$ is homotopically finitely presented (hfp), if in the homotopy category $\Ho(dgalg_{\mk})$ the object $A$ is
a retract of some finite cell DG algebra $B.$
\end{defi}

We recall the notions of smoothness and properness.

\begin{defi}{\cite{KS}} 1) A DG algebra $A$ is smooth over $\mk$ if the diagonal $A\mhyphen A$-bimodule is perfect:
$$A\in D_{\perf}(A\otimes A^{op}).$$

2) A DG algebra $A$ is proper over $\mk$ if $A\in D_{\perf}(\mk),$ or, in other words, the total cohomology of $A$ is finite-dimensional.\end{defi}

We have the following implications, due to To\"en and Vaqui\'e:

\begin{theo}\label{th:TV_implications}{\cite{TV}} 1) If a DG algebra is hfp over $\mk,$ then it is smooth.

2) If a DG algebra is smooth and proper over $\mk,$ then it is hfp.

3) If DG algebras $A$ and $A^\prime$ are Morita equivalent and $A$ is hfp, then so is $A^\prime.$\end{theo}

The part 3) of the above theorem implies that we can talk about the homotopy finiteness of small DG categories,
which are Morita equivalent to a DG algebra, that is, generated by a single object.

We may as well take the category of small DG categories $\dgcat_{\mk},$ and define weak equivalences as Morita equivalence.
Tabuada \cite{Tab1} has constructed the corresponding model structure, which is again finitely generated. We denote by $\Ho_M(dgcat_{\mk})$
the corresponding homotopy category.

Also, by \cite{Tab2} there is a model structure on $\dgcat_{\mk}$ with weak equivalences being quasi-equivalences. We denote by $\Ho(\dgcat_{\mk})$ the corresponding homotopy category.

\begin{defi}A DG category $\cB$ is called finite cell if 

i) $\cB$ has a finite number of objects;

ii) the graded category $\cB^{gr}$ is freely generated by a finite number of morphisms $f_1,\dots,f_n;$

iii) we have
$$df_i\in \mk\langle f_1,\dots f_{i-1}\rangle,\quad 1\leq i\leq n.$$
\end{defi}

The homotopically finite DG categories are defined in the same way.

\begin{defi}A small DG category $\cA$ is homotopically finitely presented (hfp) if in the homotopy category
$\Ho_M(dgcat_{\mk})$ $\cA$ is a retract of a finite cell DG category.\end{defi}

By \cite[Corollary 2.12]{TV}, a DG category is hfp if and only if it is Morita equivalent to a hfp DG algebra. 



\begin{remark}The notions of smoothness and properness make sense for all small DG categories, and the statements 1), 2) of Theorem \ref{th:TV_implications} also hold for small DG categories.\end{remark}

The following is known but we prove it here since we do not know a reference.

\begin{prop}\label{prop:quotient_by_1hfp}Let $\cC$ be a small DG category, which is hfp, and $E\in Ob(\cC)$ an object. Then the DG quotient $\cC/E$ is also hfp.\end{prop}

\begin{proof}We may and will assume that $\cC$ is a DG category with two objects, which we denote by $X_1,$ $X_2,$ and $E=X_1.$ Indeed, as we already mentioned, $\cC$ is Morita equivalent to a (hfp) DG algebra, hence (after possibly replacing $\cC$ by a larger subcategory of $\cS\cF_{fg}(\cC)$) we can find an object $X\in\cC$ which classically generates $D_{\perf}(\cC).$ Then the full subcategory $\cC'=\{E,X\}\subset \cC$ is Morita equivalent to $\cC,$ and $\cC'/E$ is Morita equivalent to $\cC/E.$ Hence, we can replace $\cC$ by $\cC'$ (we assume $E\ne X,$ since otherwise there is nothing to prove).   

If $\cC$ is itself finite cell, then by definition the Drinfeld DG quotient $\cC/E$ is also finite cell.

In general, consider the category $\dgcat_{\mk}(1,2)$ of DG categories with the fixed set of objects $(X_1,X_2).$ Morphisms in $\dgcat_{\mk}(1,2)$
are DG functors which induce the identity map on objects. Then $\dgcat_{\mk}(1,2)$ is again a finitely generated model category. Indeed, it is equivalent to the undercategory $(\mk\oplus\mk\downarrow\dgalg_{\mk}),$ that is, the category of DG algebras $A$ equipped with a morphism $\mk\oplus\mk\to A.$ It follows from \cite[Theorem 2.7]{Hir} that an undercategory of any finitely generated model category is finitely generated. 

It is easy to show that the
DG category $\cC$ is also homotopically finite in $\dgcat_{\mk}(1,2).$ Further, the Drinfeld DG quotient functor
$$-/X_1:\dgcat_{\mk}(1,2)\to \dgcat_{\mk}(1,2)$$ preserves weak equivalences, hence induces a functor
$$-/X_1:\Ho(\dgcat_{\mk}(1,2))\to \Ho(\dgcat_{\mk}(1,2)).$$
We already know that it preserves finite cell objects. Hence, it also preserves homotopically finite objects, since they are retracts of finite cell objects. This proves the proposition.
\end{proof}

In the introduction we have introduced the notion of a smooth categorical compactification (Definition \ref{def:intro_smooth_cat_comp}). We have the following corollary.


\begin{cor}\label{cor:sm_comp->hfp}Assume that a small DG category $\cA$ has a smooth categorical compactification (Definition \ref{def:intro_smooth_cat_comp}). Then $\cA$ is hfp.\end{cor}

\begin{proof}Indeed, this follows directly from Proposition \ref{prop:quotient_by_1hfp} and Theorem \ref{th:TV_implications}.\end{proof}

\section{Homological epimorphisms and localizations}
\label{sec:localizations}

We recall the following result of Neeman on the localizations of compactly generated triangulated categories. If $\cT$ is a compactly generated triangulated category, then $\cT^c\subset\cT$ denotes the full triangulated subcategory of compact objects.

\begin{theo}\label{th:Neeman}\cite{N} Let $\cT$ and $\cS$ be compactly generated triangulated categories,
and $F:\cT\to\cS$ an exact functor commuting with small direct sums and preserving compact objects. The following are equivalent:

(i) The induced functor $F^c:\cT^c\to\cS^c$ is a localization up to direct summands (i.e. it is a localization onto its image, and the Karoubi completion of the image coincides with $\cS^c$);

(ii) The functor $f:\cT\to\cS$ is a localization, and its kernel is generated (as a localizing subcategory)
by its intersection with $\cT^c.$\end{theo}

We will restrict ourselves to triangulated categories with a DG enhancement. We need to make a remark on our notation.

{\noindent{\bf Notational convention.}} {\it For a DG functor $\Phi:\cA\to\cB$ between small DG categories we denote by $\Phi_*:D(\cB)\to D(\cA)$ the restriction of scalars functor. It's left adjoint, the extension of scalars, is denoted by $\Phi^*:D(\cA)\to D(\cB).$ We denote by $I_{\cA}\in\cA\mhyphen\Mod\mhyphen\cA$ the diagonal bimodule given by $I_{\cA}(X,Y)=\cA(Y,X).$ When it does not lead to confusion, we also denote this bimodule by $\cA,$ as well it's various restrictions of scalars. For example, the extension of scalars functor above can be written as
$$\Phi^*(-)=-\stackrel{\bL}{\tens{\cA}}\cB.$$} 

 The following notion of a homological epimorphism is a straightforward generalization of the corresponding notions from \cite{GL} (the case of associative rings) and \cite{Pau} (the case of DG algebras).

\begin{defi}A DG functor $\Phi:\cA\to\cB$ between small DG categories is a homological epimorphism if the extension of scalars functor
$$\Phi^*:D(\cA)\to D(\cB)$$ is a localization.\end{defi}

\begin{remark}\label{remark:about_localizations}An exact functor $F:\cT\to \cS$ between (not necessarily small) triangulated categories is a localization if the infuced functor $\bar{F}:\cT/\ker(F)\to\cS$ is an equivalence, which is in general hard to verify (for example, Conjecture \ref{conj:BO} is a statement of this kind). However, if we moreover assume that the functor $F$ has a left (resp. right) adjoint $G,$ then the condition on $F$ to be a localization is equivalent to the condition on $G$ to be fully faithful. 

Indeed, the composition $\bar{G}:\cS\xto{G}\cT\to\cT/\ker(F)$ is left (resp. right) adjoint to $\bar{F}.$ Thus, if $F$ is a localization, then $\bar{G}$ is an equivalence, and in fact $G$ identifies $\cS$ with $^{\perp}(\ker F)$ (resp. with $(\ker F)^{\perp}$). Conversely, if $G$ is fully faithful, that $G(\cS)\subset \cT$ is right (resp. left) admissible and $\ker(F)=G(\cS)^{\perp}$ (resp. $^{\perp}G(\cS)$), hence $F$ is a localization.\end{remark}

The property of being a homological epimorphism has a number of reformulations. 

\begin{prop}\label{prop:hom_epi}Let $\Phi:\cA\to\cB$ be a DG functor between small DG categories. The following are equivalent:

(i) $\Phi$ is a homological epimorphism;

(ii) the restriction of scalars functor $\Phi_*:D(\cB)\to D(\cA)$ is fully faithful;

(iii) For any $X,Y\in Ob(\cB)$ the natural (composition) morphism
$$\cB(\Phi(-),Y)\stackrel{\bL}{\tens{\cA}}\cB(X,\Phi(-))\to\cB(X,Y)$$
is an isomorphism in $D(k);$


(iv) The natural morphism
\begin{equation}\label{eq:counit_on_bimodules}\cB\stackrel{\bL}{\tens{\cA}}\cB=(\Phi\otimes\Phi^{op})^*I_{\cA}\to I_{\cB}\end{equation} is an isomorphism in $D(\cB\otimes\cB^{op}).$\end{prop}

\begin{proof}(i)$\Leftrightarrow$(ii). This follows from the adjunction between $\Phi^*$ and $\Phi_*$ and Remark \ref{remark:about_localizations}.

(ii)$\Leftrightarrow$(iv). Indeed, fully-faithfulness of $\Phi_*$ is equivalent to the condition on the adjunction counit $\Phi^*\Phi_*\to\id$ to be an isomorphism. But the composition $\Phi^*\Phi_*$ is given by the tensor multiplication by the $\cB\mhyphen\cB$-bimodule $\cB\stackrel{\bL}{\otimes}_{\cA}\cB.$ The counit morphism corresponds to the morphism \eqref{eq:counit_on_bimodules}. Thus, the counit is an isomorphism iff the morphism \eqref{eq:counit_on_bimodules} is an isomorphism in $D(\cB\otimes\cB^{op}).$

(iii) is essentially a reformulation of (iv).
\end{proof}

\begin{cor}If $\Phi:\cA\to\cB$ is a homological epimorphism and $\cA$ is smooth, then $\cB$ is also smooth.\end{cor}

\begin{proof} By definition of homological smoothness, the bimodule $I_{\cA}\in D(\cA\otimes\cA^{op})$ is perfect. The extension of scalars functor always preserves perfect complexes. Hence, condition (iv) from Proposition \ref{prop:hom_epi} implies that $I_{\cB}\in D(\cB\otimes\cB^{op})$ is also perfect, thus $\cB$ is smooth.
\end{proof}

For completeness, we list the following equivalent reformulations of the (simpler) property of being quasi-fully-faithful.

\begin{prop}Let $\Phi:\cA\to\cB$ be a DG functor between small DG categories. The following are equivalent:

(i) $\Phi$ is quasi-fully-faithful;

(ii) the functor $\Phi^*:D(\cA)\to D(\cB)$ is fully faithful;

(iii) the functor $\Phi_*:D(\cB)\to D(\cA)$ is a localization;

(iv) the natural morphism 
\begin{equation}\label{eq:unit_on_bimodules} I_{\cA}\to(\Phi\otimes\Phi^{op})_*I_{\cB}\end{equation} is an isomorphism in $D(\cA\otimes\cA^{op}).$\end{prop}

\begin{proof}(i)$\Leftrightarrow$(iv). Indeed, we have that $(\Phi\otimes\Phi^{op})_*I_{\cB}(X,Y)=\cB(\Phi(X),\Phi(Y)).$ Further, the morphism \eqref{eq:unit_on_bimodules} is given exactly by the morphisms $\cA(X,Y)\to\cB(\Phi(X),\Phi(Y)).$ Hence,
\eqref{eq:unit_on_bimodules} is a quasi-isomorphism of bimodules iff $\Phi$ is quasi-fully-faithful.

(ii)$\Leftrightarrow$(iii) follows from the adjunction and Remark \ref{remark:about_localizations}.

(ii)$\Leftrightarrow$(iv) is proved exactly in the same way as in Proposition \ref{prop:hom_epi}.
\end{proof}

So the properties of being homological epimorphism and quasi-fully-faithful are dual to each other.

\begin{defi}We call a DG functor $\Phi:\cA\to\cB$ between small DG categories a localization if the functor $\Phi^*:D_{perf}(\cA)\to D_{perf}(\cB)$ is a localization up to direct summands.\end{defi}

The above Theorem \ref{th:Neeman} directly implies the following.

\begin{cor}\label{cor:criterion_of_loc} Let $\Phi:\cA\to\cB$ be a functor between small DG categories. The following are equivalent:

(i) $\Phi$ is a localization;

(ii) $\Phi$ is a homological epimorphism and the kernel of $\Phi^*:D(\cA)\to D(\cB)$ is generated (as a localizing subcategory) by its intersection with $D_{perf}(\cA).$\end{cor}

We finish this section by mentioning a situation when a homological epimorphism is automatically a localization.

\begin{lemma}\label{lem:automatic_localization}Suppose that we have a commutative square of DG functors
$$
\begin{CD}
\cA @>F_1>> \cB\\
@VG_1VV @VG_2VV\\
\cC @>F_2>> \cD
\end{CD}
$$
Suppose that $F_1$ and $F_2$ are quasi-fully-faithful, $G_1$ is a localization, $G_2$ is a homological epimorphism, and the induced functor $\bar{G_2}:\cB/F_1(\cA)\to\cD/F_2(\cC)$ is a Morita equivalence. Then $G_2$ is a localization. Moreover, we have an identification of subcategories: $$\ker(G_2^*:D_{perf}(\cB)\to D_{perf}(\cD))=F_1^*(\ker(G_1^*:D_{perf}(\cA)\to D_{perf}(\cC))).$$ 
\end{lemma}

\begin{proof}By Corollary \ref{cor:criterion_of_loc}, we only need to show that $\ker(G_2^*:D(\cB)\to D(\cD))$ is identified with $F_1^*(\ker(G_1^*:D(\cA)\to D(\cC))).$ Let us denote by $\pr_1:\cB\to\cB/F_1(\cA),$ $\pr_2:\cD\to\cD/F_2(\cC)$ the projection DG functors. We have semi-orthogonal decompositions
\begin{equation}\label{eq:standard_SOD's}D(\cB)=\langle\pr_{1*}D(\cB/F_1(\cA)),F_1^*D(\cA)\rangle,\quad D(\cD)=\langle\pr_{2*}D(\cD/F_2(\cC)),F_2^*D(\cC)\rangle.\end{equation}
The functor $G_2^*:D(\cB)\to D(\cD)$ is compatible with the semi-orthogonal decompositions \eqref{eq:standard_SOD's}, and it induces the functors $\bar{G_2}^*$ and $G_1^*$ on the components. By our assumptions, the functor $\bar{G_2}^*$ is an equivalence. It follows that $\ker(G_2^*:D(\cB)\to D(\cD))$ is contained in $F_1^*(D(\cA)).$ This yields the assertion.\end{proof}

\section{Gluing of DG categories}
\label{sec:gluing_conceptual}

First we recall the notion of gluing, following the notation of \cite{Or}.

\begin{defi}Let $\cA$ and $\cB$ be small DG categories, and $M\in D(\cA\otimes\cB^{op})$ a bimodule. 

1) Define the DG category
$\cC=\cA\with_M\cB$ as follows. First, $Ob(\cA\with_M\cB)=Ob(\cA)\sqcup Ob(\cB).$ The complexes of morphisms are defined by
$$\cC(X,Y)=\begin{cases}\cA(X,Y) & \text{ for }X,Y\in\cA;\\
\cB(X,Y) & \text{ for }X,Y\in\cB;\\
M(X,Y) & \text{ for }X\in\cA,\, Y\in\cB;\\
0 & \text{ for }X\in\cB,\,Y\in\cA.\end{cases}$$
The composition in $\cA\with_M\cB$ is given by the compositions in $\cA,$ $\cB$ and by the bimodule structure on $M$.

2) The DG category $\cA\oright_M \cB$ is defined as follows. Its objects are triples $(X,Y,\mu),$ where $X\in Ob(\cA),$ $Y\in Ob(\cB),$ and $\mu\in M^0(X,Y)$ a closed element of degree zero. The graded $\mk$-modules of morphisms are defined by $\Hom((X,Y,\mu),(X',Y',\mu'))=\cA(X,X')\oplus\cB(Y,Y')\oplus M(X,Y')[-1].$ The differential is given by the formula $$d(f_1,f_2,f_{12})=(d(f_1),d(f_2),-d(f_{12})-f_2\mu+\mu'f_1).$$ The composition is given by
$$(f_1,f_2,f_{12})\circ (g_1,g_2,g_{12})=(f_1\circ g_1,f_2\circ g_2,f_{12}g_{1}+(-1)^{|f_2|}f_2g_{12}).$$\end{defi}

These two versions of gluing are related to each other as follows.

\begin{prop}\label{prop:two_gluings}1) There is a natural fully faithful DG functor $\Phi:\cA\oright_M\cB\to \cS\cF_{fg}(\cA\with_M\cB),$ given on objects by $(X,Y,\mu)\mapsto Cone(h_X\xto{\mu} h_Y).$ Moreover, $\Phi$ is a Morita equivalence.

2) If both $\cA$ and $\cB$ are weakly (resp. strongly) pre-triangulated, then $\Phi$ is a quasi-equivalence (resp. a DG equivalence).\end{prop}

\begin{proof}Straightforward checking.\end{proof}

\begin{remark}It follows from Proposition \ref{prop:two_gluings} 1) that for any DG functor $F:\cA\with_M\cB\to\cC,$ where $\cC$ is strongly pre-triangulated, we have a natural DG functor $F':\cA\oright_M\cB\to\cC,$ where $F'(X,Y,\mu)=Cone(F(\mu):F(X)\to F(Y))$ (the DG functor $F'$ is well defined defined up to a canonical DG isomorphism). We use this observation implicitly in the sequel.\end{remark}

We now describe the operations of gluing in terms of adjoint functors, following Tabuada \cite{Tab3}.

\begin{defi}{\cite[Definition 12.1]{Tab3}} An upper triangular DG category is a triple $(\cA,\cB,M),$ where $\cA$ and $\cB$ are small DG categories, and $M\in\cB\mhyphen\Mod\mhyphen\cA$ is a DG bimodule. A morphism $(\cA_1,\cB_1,M_1)\to (\cA_2,\cB_2,M_2)$ is given by a triple $(F_{\cA},F_{\cB},F_M),$ where $F_{\cA}:\cA_1\to\cA_2$ and $F_{\cB}:\cB_1\to\cB_2$ are DG functors, and $F_M:M_1\to (F_{\cA}\otimes F_{\cB}^{op})_*M_2$ is a morphism of DG bimodules. The composition is defined in the natural way. 
\end{defi}

The category of upper triangulated DG categories is denoted by $\dgcat_{\mk}^{tr}.$ Note that we have a natural functor $I:\dgcat_{\mk}\to\dgcat_{\mk}^{tr},$ $I(\cA)=(\cA,\cA,I_{\cA}).$

\begin{prop}\label{prop:gluings_as_adjoints_on_the_nose}1) The functor $I:\dgcat_{\mk}\to\dgcat_{\mk}^{tr}$ has a left adjoint $\with:\dgcat_{\mk}^{tr}\to\dgcat_{\mk},$ given by $(\cA,\cB,M)\mapsto \cA\with_M\cB.$

2) The functor $I:\dgcat_{\mk}\to\dgcat_{\mk}^{tr}$ has a right adjoint $\oright:\dgcat_{\mk}^{tr}\to\dgcat_{\mk},$ given by $(\cA,\cB,M)\to \cA\oright_M\cB.$\end{prop}

\begin{proof}Both assertions essentially follow directly from the definitions. To prove 1), it suffices to note that a DG functor $F:\cA\with_M\cB\to\cC$ is determined by the restricted functors $F_{\cA}:\cA\to\cC,$ $F_{\cB}:\cB\to\cC,$ and by the morphisms $M(X,Y)\to \cC(F_{\cA}(X),F_{\cB}(Y))$ for $X\in\cA,$ $Y\in\cB.$ The latter define a morphism of $\cA\mhyphen\cB$-bimodules $F_M:M\to (F_{\cA}\otimes F_{\cB}^{op})_*I_{\cC}.$ This gives a morphism $(F_{\cA},F_{\cB},F_M):(\cA,\cB,M)\to I(\cC)$ in $\dgcat_{\mk}^{tr}.$ Conversely, a morphism $(\cA,\cB,M)\to I(\cC)$ defines a DG functor $\cA\with_M\cB\to\cC.$ This proves 1).

To prove 2), let us note that a morphism $(F_1,F_2,f):I(\cC)\to (\cA,\cB,M)$ in $\dgcat_{\mk}^{tr}$ gives us a DG functor $\Phi:\cC\to \cA\oright_M\cB,$ by the formula 
\begin{equation}\label{eq:morphism_from_diagonal}\Phi(X)=(F_1(X),F_2(X),f(\one_X)).\end{equation}
Conversely, given $\Phi:\cC\to \cA\oright_M\cB,$ we define the functors $F_1,$ $F_2$ to be the compositions $$F_1:\cC\to \cA\oright_M\cB\to\cA,\quad F_2:\cC\to \cA\oright_M\cB\to\cB.$$ The morphism $f:I_{\cC}\to (F_1\otimes F_2^{op})_*M$ is uniquely determined by \eqref{eq:morphism_from_diagonal}. This proves 2).\end{proof}

According to \cite[Proposition 12.6 and Theorem 12.9]{Tab3}, the category $\dgcat_{\mk}^{tr}$ admits two natural structures of a cofibrantly generated model category. In the first one, the weak equivalences are {\it total quasi-equivalences,} i.e. the morphisms $F:(\cA,\cB,M)\to (\cA',\cB',M')$ such that $F_{\cA}$ and $F_{\cB}$ are quasi-equivalences, and $F_M$ is a quasi-isomorphism. Equivalently, the DG functor $\with(F)$ is required to be a quasi-equivalence. We denote the corresponding homotopy category by $\Ho(\dgcat_{\mk}^{tr}).$ 

In the second model structure, the weak equivalences are {\it total Morita equivalences}, which are defined similarly, but now $F_{\cA}$ and $F_{\cB}$ are required to be Morita equivalences. Equivalently, $\with(F)$ is required to be a Morita equivalence. We denote the corresponding homotopy category by $\Ho_M(\dgcat_{\mk}^{tr}).$

\begin{cor}\label{cor:gluings_as_adjoints_homotopic}1) The functor $I$ sends quasi-equivalences to total quasi-equivalences, and the functors $\with,\oright$ send total quasi-equivalences to quasi-equivalences. In particular, these functors induce the functors $\Ho(I):\Ho(\dgcat_{\mk})\to \Ho(\dgcat_{\mk}^{tr})$ and $\Ho(\with),\Ho(\oright):\Ho(\dgcat_{\mk}^{tr})\to \Ho(\dgcat_{\mk}).$ The functor $\Ho(\with)$ (resp. the functor $\Ho(\oright)$) is left (resp. right) adjoint to $\Ho(I).$

2) The functor $I$ sends Morita equivalences to total Morita equivalences. The functors $\with,\oright$ send total Morita equivalences to Morita equivalences. In particular, these functors induce the functors $\Ho_M(I):\Ho_M(\dgcat_{\mk})\to\Ho_M(\dgcat_{\mk}^{tr})$ and $\Ho_M(\with),\Ho_M(\oright):\Ho_M(\dgcat_{\mk}^{tr})\to \Ho_M(\dgcat_{\mk}).$ The functor $\Ho_M(\with)$ (resp. the functor $\Ho_M(\oright)$) is left (resp. right) adjoint to $\Ho_M(I).$\end{cor}

We need the following statement on morphisms in $\Ho(\dgcat_{\mk}^{tr}).$ 

\begin{prop}\label{prop:Homs_of_bimodules_give_functors}1) Let $\cA,$ $\cB$ be small DG categories. We have a natural functor $\tilde{U}_{\cA,\cB}:Z^0(\cB\mhyphen\Mod\mhyphen \cA)\to\dgcat_{\mk}^{tr},$ given on objects by $\tilde{U}_{\cA,\cB}(M)=(\cA,\cB,M).$
The functor $\tilde{U}_{\cA,\cB}$ sends quasi-isomorphisms to total quasi-equivalences. In particular, it induces a functor
$U_{\cA,\cB}:D(\cA\otimes\cB^{op})\to \Ho(\dgcat_{\mk}^{tr}).$

2) Let $F_{\cA}:\cA\to\cA',$ $F_{\cB}:\cB\to \cB'$ be DG functors between small DG categories, and take $M\in\cB\mhyphen\Mod\mhyphen \cA,$ $M'\in\cB'\mhyphen\Mod\mhyphen \cA'.$ Then we have natural maps
\begin{equation}\label{eq:Homs_in_dgcat_k^tr}\tilde{u}:Z^0(\Hom_{\cA\otimes\cB^{op}}(M,(F_{\cA}\otimes F_{\cB}^{op})_*M'))\to \Hom_{\dgcat_{\mk}^{tr}}((\cA,\cB,M),(\cA',\cB',M'))\end{equation}
\begin{equation}\label{eq:Homs_in_Ho_dgcat_k^tr}u:\Hom_{D(\cA\otimes\cB^{op})}(M,(F_{\cA}\otimes F_{\cB}^{op})_*M')\to \Hom_{\Ho(\dgcat_{\mk}^{tr})}((\cA,\cB,M),(\cA',\cB',M')).\end{equation}\end{prop}

\begin{proof}1) The functor $\tilde{U}_{\cA,\cB}$ is defined on morphisms by $\tilde{U}_{\cA,\cB}(f)=(\id_{\cA},\id_{\cB},f).$ The remaining assertion follows immediately from the definition of a total quasi-equivalence.

2) The map $\tilde{u}$ is given by $\tilde{u}(f)=(F_{\cA},F_{\cB},f).$ The map $u$ is given by $u(f)=g\circ U_{\cA,\cB}(f),$ where $g$ is the natural morphism
$$g=(F_{\cA},F_{\cB},\id):(\cA,\cB,(F_{\cA}\otimes F_{\cB}^{op})_*M')\to (\cA',\cB',M')\quad\text{in }\Ho(\dgcat_{\mk}^{tr}).$$\end{proof}

Finally, we need the following observation on filtered colimits in $\dgcat_{\mk}^{tr}.$

\begin{prop}\label{prop:filtered_colimits_dgcat^tr} Let $\{(\cA_j,\cB_j,M_j)\}_{j\in J}$ be a filtered diagram in $\dgcat_{\mk}.$ Then we have $\colim_J(\cA_j,\cB_j,M_j)=(\cA,\cB,M),$ where
$$\cA=\colim_J\cA_j,\quad \cB=\colim_J\cB_j,\quad M=\colim_J(\cA\tens{\cA_j}M_j\tens{\cB_j}\cB)$$
(note that we take non-derived tensor product).\end{prop}

\begin{proof}This follows directly from the definition of $\dgcat_{\mk}^{tr},$ and from the standard adjunction between extension of scalars and restriction of scalars.\end{proof}

\section{More on gluing}
\label{sec:gluing}

We now prove some technical results about gluing of DG categories. For a DG functor $F:\cD\to\cE$ let us denote by $N_F\in\cD\mhyphen\Mod\mhyphen \cE$ the corresponding DG bimodule given by the formula
$$N_F(X,Y)=\cE(X,F(Y)).$$
Clearly, we have
$$F^*\cong -\stackrel{\bL}{\tens{\cD}}N_F:D(\cD)\to D(\cE),\quad F_*\cong -\stackrel{\bL}{\tens{\cE}}N_{F^{op}}:D(\cE)\to D(\cF).$$

We recall the notion of (derived) adjoint bimodules. Let $M\in\cD\mhyphen\Mod\mhyphen \cE,$ and $N\in\cE\mhyphen\Mod\mhyphen \cD.$ We say that $M$ is left adjoint to $N$ if we have morphisms
$$\tau:I_{\cD}\to M\stackrel{\bL}{\tens{\cE}}N\text{ in }D(\cD\otimes\cD^{op}),\quad 
\varepsilon:N\stackrel{\bL}{\tens{\cD}}M\to I_{\cE}\text{ in }D(\cE\otimes\cE^{op})$$
such that the following compositions are equal to the identity:
$$M\xto{\tau\otimes 1} M\stackrel{\bL}{\tens{\cE}}N\stackrel{\bL}{\tens{\cD}}M\xto{1\otimes \varepsilon}M, 
\quad N\xto{1\otimes \tau} N\stackrel{\bL}{\tens{\cE}}M\stackrel{\bL}{\tens{\cD}}N\xto{\varepsilon\otimes 1}N.$$

In the case of such adjunction, $M$ and $N$ uniquely determine each other up to a natural isomorphism in the derived category. Moreover, $M$ and $N$ are $\cE$-perfect, and we have $N\cong\bR\Hom_{\cE}(M,I_{\cE})$ in $D(\cD\otimes\cE^{op}).$ We write $M^R$ for $N,$ and $N^{L}$ for $M.$ In this section it will be convenient for us to write $Fiber(-)$ instead of $Cone(-)[-1].$

The following lemma is simply a DG interpretation of mutation of semi-orthogonal decompositions of triangulated categories.

\begin{lemma}\label{lem:adjoint_gluing}Let $\cA$ and $\cB$ be small DG categories, and $M\in \Perf_{\cA}(\cA\otimes\cB^{op}).$ Then for the right adjoint bimodule $M^R\in \Perf_{\cA^{op}}(\cB\otimes\cA^{op})$
the DG categories $\cC:=\cA\oright_M\cB$ and $\cB\oright_{M^R}\cA$ are naturally Morita equivalent.\end{lemma}

\begin{proof}Replacing $\cA$ and $\cB$ by Morita equivalent DG categories, we may assume that the bimodule $M$ is of the form $N_F,$ where $F:\cB\to\cA$ is a DG functor. The right adjoint bimodule $M^R$ is then quasi-isomorphic to $N_{F^{op}}.$

Define a DG functor $\Psi_{\cB}:\cB\to \cC$ by $\Psi_{\cB}(X)=(F(X),X,\id_{F(X)})\in\cC$ on objects, and by $\Psi_{\cB}(f)=(F(f),f,0)$ on morphisms. Since
$$\cC(\Psi_{\cB}(X_1),\Psi_{\cB}(X_2))=Fiber(\cA(F(X_1),F(X_2))\oplus\cB(X_1,X_2)\xto{(\id,F(X_1,X_2))} \cA(F(X_1),F(X_2))),$$
we see that $\Psi_{\cB}$ is quasi-fully-faithful.  

Further, we define the DG fully faithful functor $\Psi_{\cA}:\cA\to \cC$ by $\Psi_{\cA}(Y)=(Y,0,0).$ We have that \begin{equation}\label{eq:semiorth_mutation}\cC(\Psi_{\cA}(Y),\Psi_{\cB}(X))=Fiber(\cA(Y,F(X))\xto{\id}\cA(Y,F(X)))\end{equation}
is an acyclic complex. Further, we have an isomorphism
\begin{equation}\label{eq:non-zero_Homs_mutation}\cC(\Psi_{\cB}(X),\Psi_{\cA}(Y))\cong\cA(F(X),Y)=N_{F^{op}}(X,Y).\end{equation}

Altogether, we get a quasi-fully-faithful DG functor
$$\Psi:\cB\with_{N_{F^{op}}}\cA\to\cC,$$ given by the functors $\Psi_{\cB},\Psi_{\cA}$ and the isomorphisms \eqref{eq:non-zero_Homs_mutation}. 
Moreover, its image generates $D_{\perf}(\cB\with_{N_{F^{op}}}\cA).$ Therefore, $\Psi$ is a Morita equivalence. By Proposition \ref{prop:two_gluings}, the DG categories $\cB\with_{N_{F^{op}}}\cA$ and $\cB\oright_{N_{F^{op}}}\cA$ are Morita equivalent. This proves the proposition.\end{proof}

\begin{lemma}\label{lem:functor_from_gluing}Let $\cA$ and $\cB$ be small DG categories and $M\in \cB\mhyphen\Mod\mhyphen \cA$ a bimodule.
 Suppose that we are given a small DG category $\cC.$

1)A pair of DG functors $(\Phi_{\cA}:\cA\to\cC,\,\Phi_{\cB}:\cB\to\cC)$ induces the maps
$$\tilde{\phi}:Z^0(\Hom_{\cA\otimes\cB^{op}}(M,(\Phi_{\cA}\otimes\Phi_{\cB}^{op})_*I_{\cC}))\to\Hom_{\dgcat_{\mk}}(\cA\with_M\cB,\cC)$$
$$\phi:\Hom_{D(\cA\otimes\cB^{op})}(M,(\Phi_{\cA}\otimes\Phi_{\cB}^{op})_*I_{\cC})\to\Hom_{\Ho(\dgcat_{\mk})}(\cA\with_M\cB,\cC).$$

2) A pair of DG functors $(\Phi_{\cA}:\cC\to\cA,\,\Phi_{\cB}:\cC\to\cB)$ induces the maps
$$\tilde{\psi}:Z^0(\Hom_{\cC\otimes\cC^{op}}(I_{\cC},(\Phi_{\cA}\otimes\Phi_{\cB}^{op})_*M))\to\Hom_{\dgcat_{\mk}}(\cC,\cA\oright_M\cB).$$
$$\psi:\Hom_{D(\cC\otimes\cC^{op})}(I_{\cC},(\Phi_{\cA}\otimes\Phi_{\cB}^{op})_*M)\to\Hom_{\Ho(\dgcat_{\mk})}(\cC,\cA\oright_M\cB).$$\end{lemma}

\begin{proof}1) The maps $\tilde{\phi},$ $\tilde{\psi}$ come from the map $\tilde{u}$ in Proposition \ref{prop:Homs_of_bimodules_give_functors} 2), and the adjunctions from Proposition \ref{prop:gluings_as_adjoints_on_the_nose}. Similarly, the maps $\phi,$ $\psi$ come from the map $u$ in Proposition \ref{prop:Homs_of_bimodules_give_functors} 2), and the adjunctions from Corollary \ref{cor:gluings_as_adjoints_homotopic} 1).\end{proof}

It will be convenient to use a specific way to compute derived tensor products, namely via {\it bar resolutions}. For the rest of this section we denote by $M\stackrel{\bL}{\otimes}_{\cA}N$ the complex of $\mk$-modules which is defined as follows to be the following complex of $\mk$-modules. As a graded $\mk$-module, it is given by
$$(M\stackrel{\bL}{\otimes}_{\cA}N)^{gr}=\bigoplus\limits_{n\geq 0}\bigoplus\limits_{X_0,\dots,X_n\in\cA}M(X_n)\otimes\cA(X_{n-1},X_n)[1]\otimes\dots\otimes\cA(X_0,X_1)[1]\otimes N(X_0).$$
Let us write $x[f_n\mid\dots\mid f_1]y$ instead of $x\otimes f_n\otimes\dots\otimes y.$ The differential is given by
\begin{multline*}d(x[f_n\mid\dots\mid f_1]y)=d_M(x)[f_n\mid\dots f_1]y+\sum\limits_{i=1}^n (-1)^{\veps_i}x[f_n\mid\dots\mid d_{\cA}(f_i)\mid \dots\mid f_1]y+\\
(-1)^{\veps_0}x[f_n\mid\dots\mid f_1]d_N(y)+(-1)^{|x|}xf_n[f_{n-1}\mid\dots\mid f_1]y+\\
\sum\limits_{i=1}^{n-1}(-1)^{\veps_i}x[f_n\mid\dots\mid f_{i+1}f_i\mid\dots\mid f_1]y-(-1)^{\veps_1}x[f_n\mid\dots\mid f_2]f_1y,\end{multline*}
where $\veps_i=|x|+\sum\limits_{j=i+1}^n|f_j|-n+i,$ for $0\leq i\leq n.$

In fact, the complex $M\stackrel{\bL}{\otimes}_{\cA}N$ is naturally isomorphic to $M\otimes_{\cA}\tilde{Y}\otimes_{\cA}N,$ $\tilde{Y}$ is the bar resolution of the diagonal $\cA\mhyphen\cA$ bimodule, as defined in \cite[Example 6.6]{Ke}.

Note that we have a natural morphism of complexes $$f:M\stackrel{\bL}{\otimes}_{\cA}N\to M\otimes_{\cA}N,\quad f(x[f_n\mid\dots\mid f_1]y)=\begin{cases}x\otimes y & \text{ for }n=0;\\
0 & \text{ for }n>0.\end{cases}$$

More generally, given an $\cA\mhyphen\cB$-bimodule $M$ and a $\cB\mhyphen\cC$-bimodule $N,$ we define the $\cA\mhyphen\cC$-bimodule $M\stackrel{\bL}{\otimes}N$ by the formula
$$(M\stackrel{\bL}{\otimes}N)(X,Y)=M(X,-)\stackrel{\bL}{\tens{\cB}}N(-,Y).$$
This operation on bimodules is naturally associative on the chain level, which makes it convenient for us.

\begin{lemma}\label{lem:hom_epi_from_gluing}Suppose that we are given with small DG categories $\cA,\cB,$ bimodule $M\in\Mod\mhyphen (\cA\otimes\cB^{op}),$ a small DG category $\cC,$ DG functors $\Phi_{\cA}:\cA\to\cC,$ $\Phi_{\cB}:\cB\to\cC,$
and a morphism $f:M\to (\Phi_{\cA}\otimes \Phi_{\cB}^{op})_*I_{\cC}$ in $\cB\mhyphen\Mod\mhyphen\cA.$ Then the following are equivalent

(i) the DG functor $\tilde{\phi}(f):\cA\with_M\cB\to\cC$ is a homological epimorphism;

(ii) the DG $\cC\otimes\cC^{op}$-module
\begin{equation}\label{eq:convolution}Tot(\cC\stackrel{\bL}{\tens{\cA}}M\stackrel{\bL}{\tens{\cB}}\cC\to \cC\stackrel{\bL}{\tens{\cA}}\cC\oplus \cC\stackrel{\bL}{\tens{\cB}}\cC\to \cC)\end{equation}
is acyclic, where the derived tensor product is computed using bar resolutions (so that we take the totalization of an actual complex of DG bimodules).\end{lemma}

\begin{proof}Indeed, by Proposition \ref{prop:hom_epi} we have that (i) is equivalent to the acyclicity of the DG $\cC\otimes\cC^{op}$-module
$$Cone(\cC\stackrel{\bL}{\tens{\cA\with_M\cB}}\cC\to \cC).$$
But if we compute the derived tensor product via bar resolution, we see that this DG $\cC\otimes\cC^{op}$-module is DG isomorphic to \eqref{eq:convolution}. This proves the lemma.\end{proof}

\begin{lemma}\label{lem:localizations_SOD}Assume that we have a morphism $\Phi=(\Phi_{\cA},\Phi_{\cB},\Phi_{\cM}):(\cA,\cB,M)\to (\cA^\prime,\cB^\prime,M^\prime)$ in $\dgcat_{\mk}^{tr}.$

Then the morphism $\Psi=\with(\Phi):\cA\with_M\cB\to \cA'\with_{M'}\cB'$ is a localization (resp. homological epimorphism) iff both DG functors $\Phi_{\cA}$ and $\Phi_{\cB}$ are localizations (resp. homological epimorphisms)
and we have an isomorphism \begin{equation}\label{eq:iso_M}(\Phi_{\cA}\otimes\Phi_{\cB}^{op})^*M\stackrel{\sim}{\to}M^\prime\end{equation} in $D(\cA\otimes\cB^{op}),$ which corresponds by adjunction to $\Phi_M.$
\end{lemma}

\begin{proof}



First we prove the statement for homological epimorphisms. We have the following isomorphisms: $$(\Psi\otimes \Psi^{op})^*(I_{\cA\with_M\cB})_{\mid \cA^\prime\otimes\cA^{\prime op}}\cong \cA'\stackrel{\bL}{\tens{\cA}}\cA'\quad\text{in }D(\cA'\otimes\cA^{\prime op}),$$
$$(\Psi\otimes\Psi^{op})^*(I_{\cA\with_M\cB})_{\mid \cB^\prime\otimes\cB^{\prime op}}\cong \cB'\stackrel{\bL}{\tens{\cB}}\cB'\quad\text{in }D(\cB'\otimes\cB^{\prime op}),$$
$$(\Psi\otimes\Psi^{op})^*(I_{\cA\with_M\cB})_{\mid \cA^\prime\otimes\cB^{\prime op}}\cong \cB'\stackrel{\bL}{\tens{\cB}}M\stackrel{\bL}{\tens{\cA}}\cA'\quad\text{in }D(\cA'\otimes\cB^{\prime op}),$$ 
$$(\Psi\otimes\Psi^{op})^*(I_{\cA\with_M\cB})_{\mid \cB^\prime\otimes\cA^{\prime op}}=0\quad\text{in }D(\cB'\otimes\cA^{\prime op}).$$
It follows that $\Psi$ is a homological epimorphism iff both $\Phi_{\cA}$ and $\Phi_{\cB}$ are homological epimorphisms and the morphism \eqref{eq:iso_M}
is an isomorphism.

Now, assume that $\Psi$ is a homological epimorphism. By Corollary \ref{cor:criterion_of_loc}, $\Psi$ is a localization iff the kernel of
$$\Psi^*:D(\cA\with_{M}\cB)\to D(\cA^\prime\with_{M^\prime}\cB^\prime)$$
is generated by compact objects as a localizing subcategory. But the functor $\Psi^*$ is compatible with semi-orthogonal decompositions
$$D(\cA\with_{M}\cB)=\langle D(\cA),D(\cB)\rangle,\quad D(\cA^\prime\with_{M^\prime}\cB^\prime)=\langle D(\cA^\prime),D(\cB^\prime)\rangle,$$
and it induces the functors $\Phi_{\cA}^*:D(\cA)\to D(\cA^\prime),$ $\Phi_{\cB}^*:D(\cB)\to D(\cB^\prime).$ Since both semi-orthogonal decompositions induce decompositions on compact objects, it follows that $\ker\Phi^*$ is generated by compact objects iff both $\ker\Phi_{\cA}^*$ and $\ker\Phi_{\cB}^*$ are generated by compact objects. This proves the assertion about localizations.
\end{proof}

We will use the following special case of the gluing construction.

\begin{defi}\label{defi:gluing_from_a_pair}Suppose that we are given a pair of DG functors between small DG categories
$$\cA\xlto{F_1}\cC\xto{F_2}\cB.$$ Take the $\cB\mhyphen\cA$-bimodule
$$M:=\cB\stackrel{\bL}{\tens{\cC}}\cA,$$ where the derived tensor product is computed via bar resolution. We put
$$\cA\with_{(\cC)}\cB:=\cA\with_M\cB,\quad \cA\oright_{(\cC)}\cB:=\cA\oright_M\cB.$$\end{defi}

\begin{lemma}\label{lem:functor_from_comm_diag}1) Assume that we have a commutative diagram of DG categories
$$\begin{CD}
\cA @< F_1 << \cC\\
@V G_1 VV                              @V F_2 VV\\
\cD @< G_2 << \cB.
\end{CD}$$
Then it induces a natural DG functor
\begin{equation}\label{eq:functor_from_comm_diag}\cA\with_{(\cC)}\cB\to\cD.\end{equation}

2) The functor \eqref{eq:functor_from_comm_diag} is a homological epimorphism iff the DG $\cD\otimes\cD^{op}$-module
$$Tot(\cD\stackrel{\bL}{\tens{\cC}}\cD\to \cD\stackrel{\bL}{\tens{\cA}}\cD\oplus \cD\stackrel{\bL}{\tens{\cB}}\cD\to \cD)$$ is acyclic. 
\end{lemma}

\begin{proof}1) By Lemma \ref{lem:functor_from_gluing} it suffices to define a morphism of bimodules
$$f:\cB\stackrel{\bL}{\tens{\cC}}\cA\to (G_1\otimes G_2^{op})_*I_{\cD}.$$ 
We define $f$ to be the composition
$$\cB\stackrel{\bL}{\tens{\cC}}\cA\to \cB\tens{\cC}\cA\xto{f'} (G_1\otimes G_2^{op})_*I_{\cD},$$
$$f'(g_2\otimes g_1)=G_2(g_2)G_1(g_1)\text{ for }g_1\in\cA(X,F_1(Y)),\, g_2\in\cB(F_2(Y),Z).$$

2) The second assertion is a direct application of Lemma \ref{lem:hom_epi_from_gluing}.\end{proof}

\begin{lemma}\label{lem:functor_to_comm_diag}1) For a pair of DG functors $\cA\xlto{F_1}\cC\xto{F_2}\cB,$ there is a natural DG quasi-functor $\cC\to\cA\oright_{(\cC)}\cB.$ 

2)Assume that we have a commutative diagram of DG categories
$$\begin{CD}
\cA @< F_1 << \cC\\
@V G_1 VV                              @V F_2 VV\\
\cD @< G_2 << \cB,
\end{CD}$$
such that $\cD$ is strongly pre-triangulated. Then the composition $\cC\to\cA\oright_{(\cC)}\cB\to\cD$ is zero in $\Ho(\dgcat_{\mk}).$\end{lemma}

\begin{proof}1) We have the natural morphism in $D(\cA\otimes\cB^{op}):$ $\cC\cong \cC\stackrel{\bL}{\otimes}_{\cC}\cC\to \cB\stackrel{\bL}{\otimes}_{\cC}\cA.$ By Lemma \ref{lem:functor_from_gluing} 2), this morphism induces a quasi-functor $\cC\to \cA\oright_{(\cC)}\cB.$

2) Denoting by $\Phi:\cC\to\cD$ the composite quasi-functor, for any object $X\in\cC$ we see that $\Phi(X)=Cone(G_1F_1(X)\xto{\id}G_2F_2(X))\simeq 0$ in $H^0(\cD).$\end{proof}

\begin{lemma}\label{lem:functor_between_gluings_from_comm_diags}1) Assume that there is a morphism of diagrams
$$\Phi:(\cA\stackrel{F_1}{\leftarrow}\cC\stackrel{F_2}{\to}\cB)\to (\cA^\prime\stackrel{F_1^\prime}{\leftarrow}\cC^\prime\stackrel{F_2^\prime}{\to}\cB^\prime),$$
i.e. a triple of DG functors $(\Phi_{\cA}:\cA\to\cA^\prime,\Phi_{\cA}:\cB\to\cB^\prime,\Phi_{\cC}:\cC\to\cC^\prime),$
such that both squares commute. Then $\Phi$ induces a DG functor
$$\Psi:\cA\with_{(\cC)}\cB\to \cA^\prime\with_{(\cC^\prime)}\cB^\prime.$$

2) In the notation of 1), assume that the DG functors $\Phi_{\cA}$ and $\Phi_{\cB}$ are localizations, and $\Phi_{\cC}$ is a homological epimorphism. Then the DG functor $\Psi$ is a localization.\end{lemma}

\begin{proof}
1) By Lemma \ref{lem:localizations_SOD}, it suffices to define a morphism of bimodules
$$g:\cB\stackrel{\bL}{\tens{\cC}}\cA\to (\Phi_{\cA}\otimes\Phi_{\cB}^{op})_*(\cB'\stackrel{\bL}{\tens{\cC^\prime}}\cA').$$ It is defined in the natural way since the derived tensor products are computed via bar resolution. Also, by adjunction $g$ corresponds to the following composition in $D(\cA\otimes\cB^{op}):$
\begin{equation}\label{eq:composition}(\Phi_{\cA}\otimes\Phi_{\cB}^{op})^*(\cB\stackrel{\bL}{\tens{\cC}}\cA)=
\cB'\stackrel{\bL}{\tens{\cB}}\cB\stackrel{\bL}{\tens{\cC}}\cA\stackrel{\bL}{\tens{\cA}}\cA'\cong\cB'\stackrel{\bL}{\tens{\cC^\prime}}\cC'\stackrel{\bL}{\tens{\cC}}\cC'\stackrel{\bL}{\tens{\cC^\prime}}\cA'\to \cB'\stackrel{\bL}{\tens{\cC^\prime}}\cA'.\end{equation}

2) By Lemma \ref{lem:localizations_SOD}, we only need to check that the composition \eqref{eq:composition} is an isomorphism. But since the DG functor $\Phi_{\cC}$ is a homological epimorphism, the morphism $\cC'\stackrel{\bL}{\otimes}{\cC}\cC'\to \cC^\prime$ is an isomorphism in $D(\cC'\otimes\cC^{\prime op}),$ hence the last arrow of \eqref{eq:composition} is also an isomorphism.
\end{proof}

\begin{lemma}\label{lem:filtered_daigrams_gluing}Suppose that we have a filtered diagram of pairs of DG functors $(\cA_j\xlto{F_1^j}\cC_j\xto{F_2^j}\cB_j),$ $j\in J.$ Put $\cA:=\colim\limits_J\cA_j,$ $\cB:=\colim\limits_J\cB_j,$ $\cC:=\colim\limits_J\cC_j.$ Then we have an isomorphism \begin{equation}\label{eq:oright_filtered_colimits}\colim\limits_{J}(\cA_j\oright_{(\cC_j)}\cB_j)\xto{\sim}\cA\oright_{(\cC)}\cB\end{equation} in $\dgcat_{\mk}.$\end{lemma}

\begin{proof}Applying Proposition \ref{prop:filtered_colimits_dgcat^tr}, we obtain an isomorphism $\colim\limits_{J}(\cA_j,\cB_j,\cB_j\stackrel{\bL}{\otimes}_{\cC_j}\cA_j)\cong(\cA,\cB,\cB\stackrel{\bL}{\otimes}_{\cC}\cA)$ in $\dgcat_{\mk}^{tr}.$ It is straightforward to check that the functor $\oright:\dgcat_{\mk}^{tr}\to\dgcat_{\mk}$ commutes with filtered colimits. This implies the isomorphism \eqref{eq:oright_filtered_colimits}.\end{proof}

\begin{lemma}\label{lem:gluing_and_mutation}Within the notation of Lemma \ref{lem:functor_from_comm_diag}, suppose that $\cC=\cA,$ $\cD=\cB,$ $F_1=\id_{\cA},$ $G_2=\id_{\cB}.$ We assume that $\cA$ and $\cB$ are strongly pre-triangulated. We denote by $\Phi:\cA\oright_{(\cA)}\cB\to \cB$ the induced DG functor, and by $\Psi:\cA\to\cA\oright_{(\cA)}\cB$ the induced DG quasi-functor. 

Then $\Psi$ is quasi-fully-faithful, we have a semi-orthogonal decomposition $$[\cA\oright_{(\cA)}\cB]=\langle [\cB],\Psi([\cB])\rangle,$$
and the functor $[\Phi]:[\cA\oright_{(\cA)}\cB]\to [\cB]$ is a left semi-orthogonal projection. In particular, the functor $\Phi$ is a localization, and its kernel is generated by $\Psi(\cB).$\end{lemma}

\begin{proof}Let us note that we have a quasi-equivalence $\cA\oright_{(\cA)}\cB\xto{\sim}\cA\oright_{N_F}\cB.$ The assertion then follows from the proof of Lemma \ref{lem:adjoint_gluing}.\end{proof}

\begin{defi}\label{defi:SOD_for_DG}A semi-orthogonal decomposition of a DG category $\cB$ is pair of DG functors $F_1:\cA_1\to\cB,$ $F_2:\cA_2\to\cB,$ 
such that $F_1,$ $F_2$ are quasi-fully-faithful, and the triangulated category $D_{perf}(\cB)$ has a semi-orthogonal decomposition $D_{perf}(\cB)=\langle F_1^*D_{perf}(\cA_1),F_2^*D_{perf}(\cA_2)\rangle.$ If these conditions are satisfied, we write $\cB=\langle F_1(\cA_1),F_2(\cA_2)\rangle.$ We also write $\cB=\langle\cA_1,\cA_2\rangle$ if the DG functors $F_1$ and $F_2$ are either clear from the context, or irrelevant.\end{defi}

\begin{remark}\label{rem:SOD_and_gluing}Let us note that, within notation of Definition \ref{defi:SOD_for_DG}, the natural DG functor $\cA_1\with_{(F_1\otimes F_2^{op})_*I_{\cB}}\cA_2\to\cB$ is a Morita equivalence.\end{remark}

\begin{lemma}\label{lem:diagram_of_twelve_functors} Suppose that we have morphisms of diagrams $(\cA_1\leto\cC_1\to\cB_1)\xto{\Phi^1}(\cA\leto\cC\to\cB)\xlto{\Phi^2}(\cA_2\leto\cC_2\to\cB_2),$ with components $\Phi_{\cA}^i:\cA_i\to\cA,$ $\Phi_{\cB}^{i}:\cB_i\to\cB,$ $\Phi_{\cC}^i:\cC_i\to \cC.$ Suppose that we have a semi-orthogonal decomposition $\cA=\langle\Phi_1(\cA_1),\Phi_2(\cA_2)\rangle,$ and similarly for $\cB$ and $\cC.$ Let us denote by $\Psi^i:\cA_i\oright_{(\cC_i)}\cB_i\to\cA\oright_{(\cC)}\cB$ the induced functors (by Lemma \ref{lem:functor_between_gluings_from_comm_diags}, 1)). Then we have a semi-orthogonal decomposition
\begin{equation}\label{eq:SOD_images_Psi^1_Psi^2}\cA\oright_{(\cC)}\cB=\langle\Psi^1(\cA_1\oright_{(\cC_1)}\cB_1),\Psi^2(\cA_2\oright_{(\cC_2)}\cB_2)\rangle.\end{equation}\end{lemma}

\begin{proof}By Remark \ref{rem:SOD_and_gluing}, we may and will assume that $\cA=\cA_1\with_{(\Phi^1_{\cA}\otimes(\Phi^2_{\cA})^{op})_*}\cA_2,$ and the functors $\Phi^i_{\cA}$ are natural inclusions (and similarly for $\cB$ and $\cC$). Under these assumptions, the morphism of $\cB_i\mhyphen\cA_i$-bimodules $\cB_i\stackrel{\bL}{\otimes}_{\cC_i}\cA_i\to (\Phi^i_{\cA}\otimes(\Phi^i_{\cB})^{op})_*(\cB\stackrel{\bL}{\otimes}_{\cC}\cA)$ is an isomorphism, $i=1,2.$ Moreover, the $\cB_1\mhyphen\cA_2$-bimodule $(\Phi^2_{\cA}\otimes(\Phi^1_{\cB})^{op})_*(\cB\stackrel{\bL}{\otimes}_{\cC}\cA)$ is zero. It follows that both $\Psi^1$ and $\Psi^2$ are quasi-fully-faithful, the images of $\Psi^1$ and $\Psi^2$ are semi-orthogonal in $H^0(\cA\oright_{(\cC)}\cB),$ and they generate $D_{perf}(\cA\oright_{(\cC)}\cB).$ This proves the semi-orthogonal decomposition \eqref{eq:SOD_images_Psi^1_Psi^2}.\end{proof}

We also prove here one result which is unnecessary for the rest of the paper.

\begin{prop}Let $\cC=\cA\with_M\cB.$ Then the following are equivalent:

(i) $\cC$ is hfp;

(ii) both $\cA$ and $\cB$ are hfp and $M$ is perfect.\end{prop}

\begin{proof}
(i)$\Rightarrow$(ii) First, $\cA$ and $\cB$ are hfp by Proposition \ref{prop:quotient_by_1hfp}. Further, Since $\cC$ is smooth, it follows from \cite[Theorem 3.24]{LS} that $M$ is perfect.

(ii)$\Rightarrow$(i) First suppose that $\cA$ and $\cB$ are finite cell DG categories. If $M$ is also finite cell (i.e. semi-free finitely generated) then $\cA \with_M \cB$ is also finite cell, hence homotopically finite.

If $M$ is any perfect DG module, then in $D(\cA\otimes\cB^{op})$ it is a retract of some finite cell DG module
$N,$ and $\cA \with_M \cB$ is a homotopy retract of $\cA \with_N \cB,$ hence $\cA \with_M \cB$ is homotopically finite.

Now suppose that more generally $\cA$ is a homotopy retract of some finite cell $\cA^\prime,$ and $\cB$ is a homotopy retract of some finite cell $\cB^\prime,$ and $\Phi_{\cA}:\cA\to\cA^\prime,$ $\Phi_{\cB}:\cB\to\cB^\prime$ are the corresponding DG quasi-functors. Then $\cA \with_M \cB$ is a homotopy retract of $\cA^\prime \with_{(\Phi_{\cA}\otimes\Phi_{\cB}^{op})^*M} \cB^\prime,$ which is already known to be homotopically finite. Proposition is proved.
\end{proof}

\section{Exotic derived categories}
\label{sec:exotic_derived}

In this section we recall the notion of absolute and coderived categories. The exotic derived categories were introduced by Positselski \cite{Pos1, Pos2}. The case of a locally noetherian abelian category was also studied earlier by Krause \cite{Kr}.

\begin{defi}Let $\cC$ be an exact category with exact small coproducts. Denote by $K(\cC)$ the homotopy category of complexes of objects in $\cC.$ Define the subcategory of coacyclic complexes
$\text{co-Acycl}(\cC)\subset K(\cC)$ to be the localizing subcategory generated by the totalizations of short exact sequences of complexes. Then the coderived category is defined as the quotient:
$$D^{co}(\cC):=K(\cC)/\text{co-Acycl}(\cC).$$\end{defi}

Here by a "totalization" of a short exact sequence of complexes we mean the sum-total complex of the corresponding bicomplex (in this bicomplex the non-zero columns are given by the terms of the exact sequence, and they are placed in degrees $-1,$ $0$ and $1$). The localizing subcategory of a triangulated category with small coproducts (i.e. cocomplete) is by definition a full triangulated subcategory which is closed under small coproducts.

Note that a bounded below acyclic complex is always co-acyclic, hence we have a natural functor
$$D^+(\cC)\to D^{co}(\cC).$$

We will be mostly interested in locally noetherian abelian categories \cite{Gab1, Gab2}. Recall that an abelian category $\cC$ is locally noetherian if it has small coproducts and exact direct limits (AB5), and $\cC$ has a small generating set of noetherian objects. For such a category $\cC,$ we denote by $\cC_f\subset\cC$ the full (essentially small) subcategory of noetherian objects. Also, we denote by $D^b_f(\cC)\subset D^b(\cC)$ the subcategory of complexes with bounded noetherian cohomology.

The basic example of a locally noetherian category is the category $\cC=\QCoh(X)$ of quasi-coherent sheaves on a noetherian scheme $X.$ In this case $\cC_f=\Coh(X).$

\begin{theo}\label{th:D^co-comp-gen-ab} Let $\cC$ be a locally noetherian abelian category. Then the following holds.

1) There is a semi-orthogonal decomposition $K(\cC)=\langle K(\Inj \cC),\text{co-Acycl}(\cC)\rangle,$
where $K(\Inj \cC)$ is the homotopy category of complexes with injective components.

2) The functor $D^+(\cC)\to D^{co}(\cC)$ is fully faithful.

3) The category $D^{co}(\cC)$ is compactly generated, and the subcategory of compact objects coincides with the essential image of
the composition
$$D^b_f(\cC)\to D^+(\cC)\to D^{co}(\cC).$$\end{theo}

\begin{proof} The statement 1) is proved in \cite[Section 3.7]{Pos2} in a different but analogous context. It also follows from \cite[Corollary A.6.2]{Pos3} after inversion of arrows. 

The statement 2) is proved in \cite[Lemma A.1.2 (a)]{Pos3}.

The statement 3) follows from 1) and \cite[Proposition 2.3]{Kr}\end{proof}

We will use the following terminology borrowed from \cite[Section 2]{Dr2}.

\begin{defi}\label{defi:Z_+-categories}1) A $\Z_+$-category is a pair $(\cC,W),$ where $\cC$ is a category, and $W:\id_{\cC}\Rightarrow\id_{\cC}$ is a natural transformation.

2) A $\Z_+$-functor between  $\Z_+$-categories $(\cC_1,W_1)$ and $(\cC_2,W_2)$ is a functor $F:\cC_1\to\cC_2$ such that $(W_2)_{F(X)}=F((W_1)_X)$ for any object $X\in\cC_1.$\end{defi}

In \cite[Definition 4.1]{BDFIK} a $\Z_+$-functor is called "factored functor".

We will say that a $\Z_+$-category $(\cC,W)$ is additive (resp. abelian, exact,...) if the underlying category $\cC$ is additive (resp. abelian, exact,...), and similarly for $\Z_+$-functors.

\begin{defi}\label{defi:MF} Let $(\cC,W)$ be a $\Z_+$-category.
A matrix factorization is a pair $(F,\delta),$ where $F=F^{ev}\oplus F^{odd}$ is a $\Z/2$-graded object in $\cC,$ and $\delta:F\to F$ is an odd morphism such that $\delta^2=W_F.$\end{defi}

Matrix factorizations form a $\Z/2$-graded DG category which we denote by $MF_{dg}(\cC,W).$ This DG category is strongly pre-triangulated. We denote by $K(\cC,W)$ its homotopy category.

Suppose that $\cC$ is moreover abelian. If $\cC$ is small, then we define the triangulated subcategory $\text{Acycl}(W)\subset K(\cC,W)$ of {\it absolutely acyclic} matrix factorizations to be generated by the totalizations of short exact sequence of matrix factorizations. The absolute derived category is defined as the quotient
$$D^{abs}(\cC,W):=K(\cC,W)/\text{Acycl}(W).$$

\begin{remark}\label{remark:enhancement_for_D^abs} Note that by the definition of the category $D^{abs}(\cC,W),$ it has a natural DG enhancement, given by the DG quotient $MF_{dg}(\cC,W)/\text{Acycl}_{dg}(W),$ where $\Acycl_{dg}(W)$ is the full DG subcategory of  absolutely acyclic matrix factorizations.\end{remark}

If $\cC$ has exact small coproducts, then we define the subcategory $\text{co-Acycl}(W)\subset K(\cC,W)$ to be the localizing subcategory generated by totalizations of short exact sequences. In this case the coderived category is defined as the quotient
$$K(\cC,W)/\text{co-Acycl}(W).$$

\begin{theo}\label{th:D^co-comp-gen-MF}Let $(\cC,W)$ be a locally noetherian abelian $\Z_+$-category. Then the following holds.

1) There is a semi-orthogonal decomposition $K(\cC,W)=\langle K(\Inj \cC,W),\text{co-Acycl}(W)\rangle,$
where $K(\Inj \cC,W)$ is the homotopy category of matrix factorizations with injective components.

2) The category $D^{co}(\cC,W)$ is compactly generated, and the subcategory of compact objects is the Karoubi completion of the essential image of
the functor
$$D^{abs}(\cC_f,W)\to D^{co}(\cC,W),$$
which is fully faithful.\end{theo}

\begin{proof}1) The statement 1) is proved in \cite[Section 3.7]{Pos2} in a different but analogous context.

The statement 2) is proved in the same way as \cite[Proposition 1.5 (d)]{EP}.\end{proof}

In the special case when $\cC=\QCoh X$ for some noetherian separated scheme, and $W$ is given by a regular function on $X$ we write $D^{co}(\QCoh(X,W))$ or just $D^{co}(X,W).$
instead of $D^{co}(\QCoh X,W).$ We also write $D^{abs}(\Coh(X,W))$ or $D^{abs}_{coh}(X,W)$ instead of $D^{abs}(\Coh X,W).$

\section{Coherent sheaves and coherent matrix factorizations}
\label{sec:coh_and_MF}

\subsection{Coherent sheaves}
\label{ssec:coh_sh}

Our reference for the notion of enhancement is \cite{BK}.

Fix some base field $\mk.$
Let $X$ be separated noetherian scheme over $\mk.$ Recall that the category $\QCoh(X)$ is locally noetherian, and $\Coh(X)\subset \QCoh(X)$ is exactly its subcategory of noetherian objects. Hence, by Theorem \ref{th:D^co-comp-gen-ab} 3), the triangulated category $D^b_{coh}(X)$
is exactly the subcategory of compact objects in the coderived category $D^{co}(X):=D^{co}(\QCoh(X)).$

We will need the following enhancement of $D^{co}(X).$ Denote by $\Flasque(X)\subset\QCoh(X)$ the subcategory of all
flasque quasi-coherent sheaves. It is closed under small coproducts, extensions, and under cokernels of injections. It contains the category of injective quasi-coherent sheaves: $\Inj(X)\subset \Flasque(X).$ It follows that one has a semi-orthogonal decomposition
$$K(\Flasque(X))=\langle K(Inj_X),\text{co-Acycl}(\Flasque(X))\rangle.$$
Therefore, we have
$$D^{co}(X)\cong K(\Flasque(X))/\text{co-Acycl}(\Flasque(X)).$$
Let us denote by $\com(\cA)$ the DG category of (unbounded) complexes of objects in $\cA,$ where $\cA$ is any additive $\mk$-linear category.
Then one has a natural DG enhancement of $D^{co}(X):$
$$D^{co}(X)\cong \Ho(\com(\Flasque(X))/\com^{co\mhyphen ac}(\Flasque(X))).$$
The set-theoretic issues are resolved in the same way as in \cite[Appendix A]{LO}.

Further, for a morphism $f:X\to Y$ of noetherian separated $\mk$-schemes one has a natural DG functor
$$f_*:\com(\Flasque(X))\to \com(\Flasque(Y)),$$ which takes $\com^{co\mhyphen ac}(\Flasque(X))$
to $\com^{co\mhyphen ac}(\Flasque(Y)).$ Hence, we have a natural DG functor
$$f_*:\com(\Flasque(X))/\com^{co\mhyphen ac}(\Flasque(X))\to \com(\Flasque(Y))/\com^{co\mhyphen ac}(\Flasque(Y)).$$

Now let $\mD^b(X):=\com^b_{coh}(\Flasque(X))\subset \com(\Flasque(X))$ be  a full DG subcategory consisting of complexes which are isomorphic to an object of $D^b(\Coh X))$ in $D^{co}(X).$ As in \cite[Appendix A]{LO}, we may assume that this DG categories are small, and we have well-defined pushforward DG functors
$f_*:\mD^b(X')\to \mD^b(X)$ for any proper morphism $f:X'\to X.$ We have
$$D^b_{coh}(X)\cong \Ho(\mD^b(X)).$$
Moreover, we have equalities of DG functors $(fg)_*=f_*g_*$ for composable proper morphisms $f$ and $g.$







\begin{remark}The DG category $\mD^b(X)$ is of course quasi-equivalent to the standard enhancements of $D^b_{coh}(X):$ the DG category of bounded below complexes of injective quasi-coherent sheaves with bounded coherent cohomology; the Drinfeld quotient of the DG category of bounded complexes of coherent sheaves by the DG subcategory of acyclic complexes. We choose this particular enhancement because it is nicely adjusted to the proper direct images.\end{remark}

We recall some results on the triangulated category $D^b_{coh}(X)$ and the DG category $\mD^b(X).$

\begin{defi}\cite{BvdB, R} Let $\cT$ be a small triangulated category and $E\in\cT$ an object. Take the sequence $\{\cT_n\}_{n\geq 0}$ of full subcategories of $\cT,$
where

- $\cT_0$ consists of direct summands of finite direct sums of shifts of $E;$

- $\cT_{n+1}$ consists of direct summands of objects $F$ such that there exists an exact triangle
$$F^\prime\to F\to F^{\prime\prime}\to F^\prime[1]$$ with $F^\prime\in\cT_n,$ $F^{\prime\prime}\in\cT_0.$

Then $E$ is called a strong generator if for $n>>0$ one has $\cT_n=\cT.$\end{defi}

\begin{theo}\label{th:Rouquier_generator}\cite{R} If a scheme $X$ is separated of finite type over a field $\mk,$ then the triangulated category
$D^b_{coh}(X)$ has a strong generator.\end{theo}

In the case of perfect fields there is a stronger result by Lunts.

\begin{theo}\label{th:Lunts_smoothness}\cite{L} Let $X$ be a separated scheme of finite type over a perfect field $\mk.$
Then the DG category $\mD^b(X)$ is smooth.\end{theo}

Our main result on the derived categories of coherent sheaves (Theorem \ref{intro:hfp_coh}) is stronger than Theorem \ref{th:Lunts_smoothness}, but requires to assume that the field has characteristic zero.



\subsection{Coherent matrix factorizations}
\label{ssec:coh_MF}

For matrix factorizations, the story is pretty much similar. By Theorem \ref{th:D^co-comp-gen-MF} the ($\Z/2$-graded) category of coherent matrix factorizations
$D^{abs}_{coh}(X,W)$ is (up to direct summands) the subcategory of compact objects in $D^{co}(X,W).$

For the derived functors between the (absolute derived and coderived) categories of matrix factorizations there are two approaches: the technique from Appendix \ref{app:MF}, and the approach from \cite{BDFIK}. They yield the same result (see \cite[Proof of Proposition 2.22, Remark 4.4]{BDFIK}). Below we freely use these derived functors, just as for the usual derived categories of (quasi-)coherent sheaves.

As above, we can construct a system of enhancements $\mD^{abs}(X,W)$ of triangulated categories $D^{abs}_{coh}(X,W).$

Namely, one has a $(\Z/2)$-graded DG category $\Flasque (X,W)\subset \QCoh (X,W)$ of flasque matrix factorizations, and its full DG subcategory $$\Flasque (X,W)^{co\mhyphen ac}\subset \Flasque (X,W)$$ of coacyclic matrix factorizations. We get
a natural enhancement of $D^{co}(\QCoh (X,W)):$
$$D^{co}(\QCoh (X,W))\cong \Ho(\Flasque (X,W)/\Flasque (X,W)^{co\mhyphen ac}).$$
Again, for any morphism $f:X\to Y$ and a function $W$ on $Y,$ we have a natural DG functor
$$f_*:\Flasque (X,f^*W)/\Flasque (X,f^*W)^{co\mhyphen ac}\to \Flasque (Y,W)/\Flasque (Y,W)^{co\mhyphen ac}$$

Further, we have a full DG subcategory $\Flasque (X,W)_{coh}\subset \Flasque (X,W),$ consisting of matrix factorizations, which are in the essential image of the inclusion $D^{abs}_{coh}(X,W)\subset D^{co}(\QCoh (X,W)).$ Putting $\mD^{abs}(X,W):=\Flasque (X,W)_{coh}/\Flasque (X,W)^{co\mhyphen ac},$ we have
$$D^{abs}_{coh}(X,W)\cong \Ho(\mD^{abs}(X,W)).$$

By Proposition \ref{prop:D^abs_to_D^abs} we have that for any proper
morphism $f:X\to Y$ and a function $W$ on $Y$ the functor $f_*:D^{co}(X,f^*W)\to D^{co}(Y,W)$ takes $D^{abs}_{coh}(X,f^*W))$ to $D^{abs}_{coh}(Y,W).$

As in the previous subsection, we may and will assume that all the DG categories $\mD^{abs}(X,W)$ are small, and for any proper
morphism $f:X\to Y$ and a function $W$ on $Y$ one has the DG functor
$f_*:\mD^{abs}(X,f^*W)\to \mD^{abs}(Y,W).$
Again, we have equalities of DG functors $(fg)_*=f_*g_*.$

For completeness, we formulate and prove the results analogous to Theorems \ref{th:Rouquier_generator} and \ref{th:Lunts_smoothness}.

\begin{theo}\label{th:strong_generator_MF}Let $X$ be a separated scheme of finite type over a field $\mk,$ and $W\in\cO(X)$ a regular function.
Then the triangulated category $D^{abs}_{coh}(X,W)$ has a strong generator.\end{theo}

\begin{proof} Denote by $X^0$ the fiber of $W$ over $0.$ In the special case when $W$ is nowhere zero divisor, by \cite[Theorem 2.7]{EP} we have an
equivalence of triangulated categories
$$D^{abs}_{coh}(X,W)\cong D^b_{coh}(X^0)/\langle j^*D^b_{coh}(X)\rangle,$$
where we denote by $j:X^0\to X$ the inclusion.
By Theorem \ref{th:Rouquier_generator}, we know that the category $D^b_{coh}(X^0)$ has a strong generator. Hence, the category $D^{abs}_{coh}(X,W)$ has a strong generator.

In general we proceed in the following way. We have a natural functor $\iota:\Coh(X^0)\to D^{abs}_{coh}(X,W),$
$\iota(\cG)=(\cG,0)$ - a matrix factorization concentrated in even degree, with zero "differential". The functor $\iota$ extends to the exact functor on triangulated categories:
$$\iota:D^b_{coh}(X^0)\to D^{abs}_{coh}(X,W).$$ Further, for any coherent matrix factorization $(\cF,\delta)$ we have
an exact triangle in $D^{abs}_{coh}(X,W):$
\begin{equation}\label{eq:triangle_for_MF}\iota(\ker\delta^0)\to (\cF,\delta)\to \iota(\coker\delta^0)[1]\to \iota(\ker\delta^0)[1].\end{equation} By Theorem \ref{th:Rouquier_generator}, we have a strong generator $\cG\in D^b_{coh}(X^0),$ say, $\langle\cG\rangle_n=D^b_{coh}(X^0).$ It follows from the triangle \eqref{eq:triangle_for_MF} that
$\iota(\cG)$ is a strong generator of $D^{abs}_{coh}(X,W),$ with $\langle\iota(\cG)\rangle_{2n+1}=D^{abs}_{coh}(X,W).$\end{proof}

We need an analogue of Sard Lemma for regular functions on singular schemes.

\begin{lemma}\label{lem:Sard_lemma}Let $X$ be a separated scheme of finite type over a perfect field $\mk,$ and $W:X\to\A^1$ a regular function. Denote by $t$ the coordinate on $\A^1.$ Let us assume that the scheme $(X\times_{k[t]}k(t))_{red}$ has a smooth stratification over $\mk(t)$ (this holds automatically if $k$ has characteristic zero). Then the category $D^{abs}_{coh}(X,W-c)$ vanishes for all but finitely many values of $c\in k.$

We call $c\in k$ a critical value of $W$ if $D^{abs}_{coh}(X,W-c)\ne 0.$\end{lemma}

\begin{proof}It follows from our assumption that there is a zero-dimensional subscheme $S\subset\A^1$ such that $X_{red}\setminus W^{-1}(S)$ has a stratification $$X_{red}\setminus W^{-1}(S)=\bigsqcup\limits_{i=1}^n Z_i,$$ such that each the restrictions $W_{\mid Z_i}:Z_i\to\A^1$ are smooth morphisms. We may and will assume that $W^{-1}(S)=\emptyset,$ hence $X_{red}=\bigsqcup\limits_{i=1}^n Z_i.$

We claim that for all $c\in k$ we have $D^{abs}_{coh}(X,W-c)=0.$ It suffices to prove this for $c=0.$ Let us note that the image of the pushforward functor $D^{abs}_{coh}(X_{red},W)\to D^{abs}_{coh}(X,W)$ generates $D^{abs}_{coh}(X,W)$ as a triangulated category (each matrix factorizations in $\Coh(X,W)$ has a finite filtration with subquotients being direct images of objects of $\Coh(X_{red},W)$). Thus, we may and will assume that $X=X_{red}.$

Note that under our assumptions $W$ is nowhere zero divisor.
By \cite[Theorem 2.7]{EP} we have an equivalence 
\begin{equation}\label{eq:equiv_MF_with_sing}D^{abs}_{coh}(X,W)\cong D^b_{coh}(X^0)/\langle j^*D^b_{coh}(X)\rangle,\end{equation}
where again we denote by $j:X^0\to X$ the inclusion. We have a stratification of the fiber:
$X^0_{red}=\bigsqcup\limits_{i=1}^n Z_i^0,$ and each $Z_i^0$ is smooth over $\mk.$ It follows from our assumptions that the functions $W_i=W_{\mid \bar{Z_i}}$ are nowhere zero divisors on $\bar{Z_i}.$ Let us put $\bar{Z_i}^0:=W_i^{-1}(0),$ and denote by $\iota_i:\bar{Z_i}\to X,$ $\iota_i^0:\bar{Z_i}^0\to X^0,$ and $j_i:\bar{Z_i}^0\to\bar{Z_i}$ the natural inclusions. We have the isomorphisms of functors $(\iota_i^0)_*\bL j_i^*\cong \bL j^*(\iota_i)_*:D^b_{coh}(X)\to D^b_{coh}(X^0).$ 

Let us choose a coherent sheaf $\cG_i$ on $\overline{Z_i}$ such that $(\cG_i)_{\mid Z_i}$ is a generator of $D^b_{coh}(Z_i)=D_{\perf}(Z_i).$ Then $\cF_i:=j_i^*\cG_i$ is an object of $D^b_{coh}(\overline{Z_i}^0),$ such that $(\cF_i)_{\mid Z_i^0}$ is a generator of $D^b_{coh}(Z_i^0)=\Perf(Z_i^0).$ It follows by induction \cite[Proposition 6.8]{L} that the direct sum $\bigoplus\limits_{i=1}^n(\iota_i^0)_*\cF_i$ generates $D^b_{coh}(X^0).$ This sum is in turn isomorphic to $j^*(\bigoplus\limits_{i=1}^n(\iota_i)_*\cG_i).$ In particular, $j^* D^b_{coh}(X)$ generates $D^b_{coh}(X^0).$ By \eqref{eq:equiv_MF_with_sing}, we get $D^{abs}_{coh}(X,W)=0.$ The lemma is proved.
\end{proof}

We recall the notion of Thom-Sebastiani sum: given regular functions $W_i\in\cO(X_i),$ $i=1,2,$ the function $W_1\boxplus W_2\in\cO(X_1\times X_2)$ is given by $$W_1\boxplus W_2:=p_1^*(W_1)+p_2^*(W_2),$$
where $p_i:X_1\times X_2\to X_i$ are the projections. For convenience we also put $$W_1\boxminus W_2:=W_1\boxplus (-W_2).$$

\begin{theo}\label{th:smoothness_MF}Let $X$ be a separated scheme of finite type over a perfect field $\mk,$ and $W:X\to\A^1$ a regular function, satisfying the condition from Lemma \ref{lem:Sard_lemma}.

Then the DG category $\mD^{abs}(X,W)$ is smooth as a $\Z/2$-graded DG category.\end{theo}

\begin{proof}First, we have an enhanced anti-equivalence
$$D^{abs}_{coh}(X,W)\cong D^{abs}_{coh}(X,-W),\quad F\mapsto \bR\cHom(F,\bD),$$
where $\bD\in D^b_{coh}(X)$ is the canonical dualizing complex.

Combining this with Proposition \ref{prop:box_product_Morita}, we obtain a Morita equivalence 
$$\mD^{abs}(X,W)\otimes \mD^{abs}(X,W)^{op}\simeq \mD^{abs}_{X^0\times X^0}(X\times X,W\boxminus W).$$
Analogously to \cite[Proposition 6.17 c)]{L}, we see that the diagonal bimodule over $\mD^{abs}(X,W)$ corresponds to
$$\Gamma_{X^0\times X^0}\Delta_*\bD\in D^{co}_{X^0\times X^0}(\QCoh (X,W)).$$
We are reduced to showing that this object is actually in the image of $D^{abs}_{coh,X^0\times X^0}(X\times X,W\boxminus W).$

Now we use the assumption on the smooth stratification. As in the proof of Lemma \ref{lem:Sard_lemma}, we can choose a zero-dimensional subscheme
$S\subset\A^1$ such that $X_{red}\setminus W^{-1}(S)$ has a smooth stratification which is smooth over $\A^1.$ We may and will assume that $S=\{0\}.$
Then we can find a stratification of $X_{red}$ such that each stratum is either smooth over $\A^1$ or smooth over $\mk$ and contained in $X^0.$
Taking the product of this stratification with itself we get a stratification of $X\times X,$ such that each stratum is either smooth over $\A^1$ (with respect to the function $W\boxminus W$) or is contained in $X^0\times X^0.$ As in the proof of Lemma \ref{lem:Sard_lemma} we conclude that $$D^{abs}_{coh}((X\times X)\setminus (X^0\times X^0),W\boxminus W)=0.$$
Therefore, we have $\Gamma_{X^0\times X^0}=\id,$ hence $\Gamma_{X^0\times X^0}\Delta_*\bD=\Delta_*\bD$ is a compact object of $D^{co}_{X^0\times X^0}(\QCoh (X,W)).$ This proves smoothness of $\mD^{abs}(X,W).$
\end{proof}

Our main result on the absolute derived categories of coherent matrix factorizations (Theorem \ref{intro:hfp_MF}) is stronger than Theorem \ref{th:smoothness_MF}, but again requires an assumption of zero characteristic.



\subsection{Nice ringed spaces}
\label{ssec:nice_ringed_spaces}

In this section we define the category of nice ringed spaces, and also study coherent sheaves and matrix factorizations on them.

\begin{defi}1) A nice ringed space over $\mk$ is a pair $(S,\cA_S)$ consisting of a noetherian separated $\mk$-scheme $S$ and a coherent sheaf of (unital) $\cO_S$-algebras $\cA_S,$ satisfying the following property:

- there is a nilpotent coherent two-sided ideal sheaf $I\subset \cA_S,$ such that there is an isomorphism of $\cO_S$-algebras:
$$\cA_S/I\cong \bigoplus\limits_{i=1}^n\cO_S/J_i,$$ where $J_i\subset\cO_S$ are some coherent ideal sheaves.

2) A morphism of nice ringed spaces from $(S,\cA_S)$ to $(S^\prime,\cA_{S^\prime})$ is a pair $(f,\varphi),$ where $f:S\to S^\prime$ is a morphism
of schemes and $\varphi:f^*\cA_Y\to \cA_S$ is a (possibly non-unital) morphism of $\cO_S$-algebras.\end{defi}

The composition of morphisms is defined in the natural way. Each noetherian separated $\mk$-scheme $S$ can be considered as a nice ringed space $(S,\cO_S).$

For a nice ringed space $(S,\cA_S)$ denote by $\QCoh(\cA_S)$ (resp. $\Coh(\cA_S)$) the abelian category of right $\cA_S$-modules which are $\cO_S$-quasi-coherent (resp. $\cO_S$-coherent). We denote by $\Mod\mhyphen\cA_S$ the category of all sheaves of right $\cA_S$-modules.

\begin{prop}\label{prop:enough_inj}Let $(S,\cA_S)$ be a nice ringed space.

1) The abelian category $\QCoh(\cA_S)$ is locally noetherian, and noetherian objects in it form the subcategory $\Coh(\cA_S).$ In particular, it has enough injectives.

2) The abelian category $\Mod\mhyphen \cA_S$ is locally noetherian.

3) An object $\cI\in\QCoh(\cA_S)$ is injective in the category $\QCoh(\cA_S)$ iff it is
injective in the category $\Mod\mhyphen \cA_S.$ In particular, in this case the restriction $\cI_{\mid U}$ onto any open subscheme $U\subset S$ is injective
in $\QCoh((\cA_S)_{\mid U}).$
\end{prop}
\begin{proof}1) Clearly, the category $\QCoh(\cA_S)$ has exact filtered colimits. Further, since $\cA_S$ itself is a coherent $\cO_S$-module, each quasi-coherent
$\cA_S$-module $\cF$ is a union of its coherent $\cA_S$-submodules, which are noetherian objects. Indeed, for any coherent $\cO_S$-submodule $\cG\subset\cF,$
we have a bigger coherent $\cA_S$-submodule $\im (\cG\tens{\cO_S}\cA_S\to\cF).$
It follows that coherent $\cA_S$-modules form a family of generators for the category $\QCoh(\cA_S).$ Hence, $\QCoh(\cA_S)$ is locally noetherian.

2) Again, it is clear that $\Mod\mhyphen \cA_S$ has exact filtered colimits. Further, for any open subset $U\subset S$ we denote by $(\cA_S)_U\in\Mod\mhyphen \cA_S$ the extension by zero of the sheaf $(\cA_S)_{\mid U}.$ The objects $(\cA_S)_U$ form a set of generators of the category $\Mod\mhyphen \cA_S.$ We are left to show that
each $(\cA_S)_U$ is a noetherian object.

It suffices to show this for $U=S.$ Moreover, we may and will assume $S$ to be affine. The restriction of scalars functor $\Mod\mhyphen \cA_S\to\Mod\mhyphen \cO_S$ is exact and conservative. Hence, it suffices to show that any coherent $\cO_S$-module $\cG$ is noetherian in $\Mod\mhyphen\cO_S.$ Since $S$ is affine, $\cG$ is a quotient of $\cO_S^n$ for some $n>0.$ By \cite[Lemma 7.7 and the proof of Theorem 7.8]{Ha}, $\cO_S$ is noetherian in $\Mod\mhyphen\cO_S.$ This proves the assertion.

3) Fix some nilpotent two-sided ideal $I\subset\cA_S$ such that $\cA_S/I\cong \bigoplus\limits_{i=1}^n\cO_S/J_i,$ where $J_i\subset\cO_S$ are some coherent ideal sheaves. Take any point $x\in S$ and any index $1\leq i\leq n$ such that $x\in \supp(\cO_S/J_i).$ The $(\cO_{S}/J_i)_x$-module $\mk(x)$ (the residue field) then gives a {\it simple} $\cA_{S,x}$-module $\mk(x)_i.$ It has an injective hull $\cI_{x,i}\in\Mod\mhyphen \cA_{S,x}.$ Further, we have a morphism of nice ringed spaces $$\iota_x:(\Spec\cO_{S,x},\cA_{S,x})\to (S,\cA_S).$$ We put $\cJ(x,i):=\iota_{x*}(\cI_{x,i})\in\QCoh(\cA_S).$ The $\cA_S$-module $\cJ(x,i)$ is injective in the category $\Mod\mhyphen \cA_S,$ hence also in the category $\QCoh(\cA_S).$ By the part 2), the category
$\Mod\mhyphen \cA_S$ is locally noetherian, hence the direct sum of arbitrary number of injectives is again injective. Hence, it suffices to show that the following holds.

{\noindent{\bf Claim.}} {\it Any quasi-coherent $\cA_S$-module $\cF$ can be embedded into the direct sum of $\cA_S$-modules $\cJ(x,i).$}

\begin{proof}The $\cA_S$-module $\cF$ has a finite decreasing filtration by $\cF\cdot I^p,$ and all the subquotients are $\cA/I$-modules. Since it suffices to prove Claim for the subquotients $\cF I^p/\cF I^{p+1},$ we may and will assume that for some $i$ the object $\cF$ is actually a quasi-coherent $(\cO_S/J_i)$-module considered as an $\cA_S$-module.

Since $\cF$ is a union of its coherent $\cA_S$-submodules, it follows by Zorn's Lemma that the problem reduces to the case when $\cF$ is a coherent $\cO_S/J_i\mhyphen$module, which we assume from now on.
Let $x_1,\dots,x_l\in S$ be generic points of irreducible components of $\supp\cF.$ Let us identify the residue fields $\mk(x_j)$ with the quasi-coherent $\cO_S$-modules $\iota_{x_j *}(k(x_j)).$ Each $\cA_S$-module $\cF\tens{\cO_S}k(x_j)$ is a direct sum of copies of $\mk(x_j)_i,$ hence it embeds into some $\cA_S$-module $\cJ_j,$ which is a direct sum of copies of $\cJ(x_j,i).$ Take the composition
$$\cF\to\bigoplus\limits_{j=1}^l\cF\tens{\cO_S}k(x_j)\to\bigoplus\limits_{j=1}^l J_j.$$ Denote by $\cF^\prime$ its kernel. Claim
reduces to $\cF^\prime.$ But each of the irreducible components of $\supp\cF^\prime$ is strictly contained in at least one of the irreducible components of $\supp\cF.$ We then apply the same argument to $\cF^\prime$ instead of $\cF.$ Since $S$ is noetherian, after a finite number of iterations we will reduce the statement to a zero sheaf. Claim is proved.\end{proof}
The claim implies 3). Proposition is proved.
\end{proof}

For a morphism of nice ringed spaces $(f,\phi):(X,\cA_X)\to (Y,\cA_Y)$ we have a direct image functor
$(f,\phi)_*:\QCoh(\cA_X)\to \QCoh(\cA_Y),$ which is left exact and commutes with small coproducts. More precisely, for $\cF\in \QCoh(\cA_X)$ we first take
the $f_*\cA_X$-module $f_*\cF,$ and then put
$$(f,\phi)_*(\cF):=f_*(\cF)\cdot\psi(1),$$
where $\psi:\cA_Y\to f_*\cA_X$ corresponds to $\phi$ by adjunction.

\begin{prop}\label{prop:f_*_ringed} Let $(f,\phi):(X,\cA_X)\to (Y,\cA_Y)$ be a morphism of nice ringed spaces.
Suppose that the morphism $f:X\to Y$ is proper. Then the functor $(f,\phi)_*:D^{co}(\QCoh(\cA_X))\to D^{co}(\QCoh(\cA_Y))$ takes $D^b_{coh}(\cA_X)$ to $D^b_{coh}(\cA_Y).$
\end{prop}

\begin{proof}By Proposition \ref{prop:enough_inj}, 3) we have that each injective quasi-coherent $\cA_X$-module is flasque. It follows that for any
$\cF\in D^b_{coh}(\cA_X)$ we have $\bR(f,\phi)_*(\cF)=\bR f_*(\cF)\cdot\psi(1),$ where $\psi:\cA_Y\to f_*\cA_X$ is as above. But the restriction of scalars of $\bR f_*(\cF)$ (from $f_*\cA_X$ to $\cO_Y$) is in $D^b_{coh}(Y)$ (since $f$ is proper). Hence, we have $\bR(f,\phi)_*(\cF)\in D^b_{coh}(\cA_Y).$\end{proof}

As in Subsection \ref{ssec:coh_sh}, we define a small DG category $\mD^b(\cA_X),$ which is an enhancement of $D^b_{coh}(\cA_X)$ for each nice ringed space $(X,\cA_X).$ Moreover, for any
proper morphism $f:(X,\cA_X)\to (Y,\cA_Y)$ we have a DG functor $f_*:\mD^b(\cA_X)\to\mD^b(\cA_Y),$ and $(fg)_*=f_*g_*.$

Now, let $(S,\cA_S)$ be a ringed space and $W\in\cO(S)$ a regular function. We consider $W$ also as a central section of $\cA_S.$ Then, by Theorem \ref{th:D^co-comp-gen-MF} the category
$D^{abs}_{coh}(\cA_X,W)=D^{abs}(\Coh(\cA_X,W))$ is (up to direct summands) the subcategory of compact objects in $D^{co}(\QCoh (\cA_X,W)).$

\begin{prop}\label{prop:f_*_ringedMF} Let $f:(X,\cA_X)\to (Y,\cA_Y)$ be a morphism of nice ringed spaces, and $W\in\cO(Y)$ a regular function.
 Suppose that the morphism $f:X\to Y$ is proper. Then the functor $f_*:D^{co}(\QCoh (\cA_X,f^*W))\to D^{co}(\QCoh (\cA_Y,W))$ takes $D^{abs}_{coh}(\cA_X,W)$ to $D^{abs}_{coh}(\cA_Y,W).$
\end{prop}

\begin{proof}This follows from Propositions \ref{prop:D^abs_to_D^abs}, \ref{prop:enough_inj}, and \ref{prop:f_*_ringed}.\end{proof}

As in Subsection \ref{ssec:coh_MF} we define a $\Z/2$-graded small DG category $\mD^{abs}_{coh}(\cA_X,W)$ for each nice ringed space $(X,\cA_X)$ and $W\in\cO(S).$ For any proper morphism $f:(X,\cA_X)\to (Y,\cA_Y),$ and $W'\in\cO(Y),$ we have a DG functor $\mD^{abs}_{coh}(\cA_X,f^*W)\to\mD^{abs}_{coh}(\cA_Y,W),$ and $(fg)_*=f_*g_*.$

\section{Smooth categorical compactification of geometric categories}
\label{sec:sm_cat_comp}

In this section we construct smooth categorical compactifications (see Definition \ref{def:intro_smooth_cat_comp}) for derived categories of coherent sheaves and absolute derived categories of coherent matrix factorizations.

\subsection{Auslander-type construction: coherent sheaves}
\label{ssec:A-construction_coh}

The construction we present in this subsection is due to Kuznetsov and Lunts \cite{KL}.

Let $S$ be a noetherian scheme, $\tau\subset \cO_S$ a sheaf of ideals, and $n>0$ an integer such that $\tau^n=0.$ Starting with such a triple $(S,\tau,n)$
we define a coherent sheaf of $\cO_S$-algebras $$\cA=\cA_S=\cA_{S,\tau,n}:=\bigoplus_{1\leq i,j\leq n}\cA_{ij},\quad \cA_{ij}:=\tau^{max(i-j,0)}/\tau^{n+1-j}.$$
The multiplication $$\cA_{jk}\tens{\cO_S}\cA_{ij}\to\cA_{ik}$$ is induced by the multiplication in $\cO_S.$ We denote by
$$e_i=1\in\cA_{ii}=\cO_S/\tau^{n+1-i}$$
the orthogonal idempotents, so that $e_1+\dots+e_n=1_{\cA}.$ 


We denote by $S_0\subset S$ the subscheme defined by the ideal $\tau,$ so that $S_{red}\subseteq S_0\subseteq S.$

\begin{prop}The pair $(S,\cA_{S,\tau,n})$ is a nice ringed space.\end{prop}

\begin{proof}Consider the two-sided ideal $I\subset\cA,$ generated by the direct sum $\bigoplus\limits_{i\ne j}\cA_{ij}.$ The quotient
$\cA/I$ is isomorphic to $\bigoplus\limits_{i=1}^n \cO_{S_0}$ as a sheaf of $\cO_S$-algebras. 

To see that the ideal $I$ is nilpotent, let us define a decreasing filtration $F^{\bullet}\cA$ on $\cA$ by putting $F^p\cA_{ij}:=(\tau^m+\tau^{n+1-j})/\tau^{n+1-j},$ where $m\geq\max(i-j,0)$ is the smallest integer such that $2m+j-i\geq p.$  It is easy to see that $F^0\cA=\cA,$ $F^1\cA=I,$ $F^{p+q}\cA\subset F^p\cA\cdot F^q\cA,$ and $F^{2n-1}\cA=0.$ Thus, $I^{2n-1}=0.$ This shows that $(S,\cA_S)$ is a nice ringed space.\end{proof}

Below we will say that $(S,\cA_S)$ is obtained from the triple $(S,\tau_S,n)$ by the Auslander-type construction.

We have a morphism of nice ringed spaces $$\rho_S:(S,\cA_S)\to S,$$
which is an identity on $S$ and the map $\cO_S\to\cA_S$ is $f\mapsto f\cdot e_1.$

Also, for each $1\leq k\leq n$ we denote by $i_k:S_0\to (S,\cA_S)$ the natural morphism which is given by the inclusion $S_0\hookrightarrow S$ and the morphism
$\cA_S\to \cA_S/\langle e_1,\dots,e_{k-1},e_{k+1}\dots,e_n\rangle\cong\cO_{S_0}.$

\begin{prop}\label{prop:rho-loc_coh}The exact functor $\rho_{S*}:\Coh(\cA_S)\to \Coh(S)$ is a localization of abelian categories, and its kernel is generated (as a Serre subcategory) by
$i_{k*}(\Coh S_0),$ $2\leq k\leq n.$

In particular, the DG functor $\rho_{S*}:\mD^b(\cA_S)\to\mD^b(S)$ is a localization, and its kernel is generated by $i_{k*}(\mD^b(S_0)),$ $2\leq k\leq n.$\end{prop}

\begin{proof}Denote by $T\subset \Coh(\cA_S)$ the Serre subcategory, which consists of coherent $\cA_S$-modules annihilated by $e_1.$ We have $T=\ker\rho_{S*}.$ For any $\cF\in T,$ we have a finite filtration of $\cF$ by $\cF\cdot I^p,$ and the subquotients $\cF\cdot I^p/\cF\cdot I^{p+1}$ are direct sums of objects from $i_{k*}(\Coh S_0),$ $2\leq k\leq n.$ Thus, these subcategories generate $T$ as a Serre subcategory. We claim that the functor
\begin{equation}\label{eq:ro_S}\rho_{S*}:\Coh(\cA_S)/T\to \Coh(S)\end{equation} is an equivalence.

First, it is essentially surjective: $\rho_{S*}\rho_{S}^*(E)\cong E$ for any $E\in\Coh(S).$

Further, we show that \eqref{eq:ro_S} is full. Take any $F,G\in Coh(\cA_S),$ and a morphism $\phi:\rho_{S*}(F)\to \rho_{S*}(G).$
By adjunction, we have a morphism $\psi:\rho_S^*\rho_{S*}(F)\to G.$ The morphism $\rho_{S*}(\psi)$ is identified with $\phi$ under the isomorphism
$$\rho_{S*}\rho_{S}^*\rho_{S*}(F)\cong \rho_{S*}(F).$$ It remains to note that the counit $$\rho_{S}^*\rho_{S*}(F)\to F$$ is an isomorphism in $\Coh(\cA_S)/T.$

Finally, we show that \eqref{eq:ro_S} is faithful. Suppose that $\phi:F\to G$ is a morphism in $\Coh(\cA_S)/T$ and $\rho_{S*}(\phi)=0.$ We may and will assume that $\phi$ is an actual morphism in $\Coh(\cA_S).$ Then we have that $\im(\phi)\in T.$ Hence, $\phi$ is zero in $\Coh(\cA_S)/T.$

Proposition is proved.\end{proof}

Suppose that $n>1.$ Denote by $S^\prime\subset S$ the subscheme defined  by the ideal $\tau^{n-1}.$ Applying Auslander-type construction to the triple
$(S^\prime,\tau,n-1),$ we get a nice ringed space $(S^\prime,\cA_{S^\prime}).$ Note that $\cA_{S^\prime}$ is identified with the (non-unital) subalgebra $(1-e_1)\cA_S(1-e_1)\subset\cA_S.$ Hence, we have a natural morphism
$$e:(S,\cA_{S})\to (S,\cA_{S'}).$$

\begin{prop}\label{prop:A_S_SOD_coh}The functors $i_{1*}:\mD^b(S_0)\to\mD^b(\cA_S)$ and $e^*:\mD^b(\cA_{S'})\to\mD^b(\cA_S)$ are quasi-fully-faithful, and they give a semi-orthogonal decomposition
$$\mD^b(\cA_{S}):=\langle i_{1*}(\mD^b(S_0)),e^*(\mD^b(\cA_{S^\prime}))\rangle.$$
By induction, we get a semi-orthogonal decomposition
$$\mD^b(\cA_S)=\langle \mD^b(S_0),\dots,\mD^b(S_0)\rangle,$$
where the number of copies equals $n.$
\end{prop}

\begin{proof}This is proved in \cite{KL}, Proposition 5.14, Corollary 5.15.\end{proof}

\begin{prop}\label{prop:A_S_sm_prop_coh}Suppose that $S_0$ is smooth and proper over $\mk.$ Then the DG category $\mD^b(\cA_S)$ is smooth and proper.\end{prop}

\begin{proof}This is proved in \cite{KL}, Theorem 5.20, Proposition 5.17.\end{proof}

By definition, a morphism of triples $f:(T,\tau_T,n)\to (S,\tau_S,n)$ is a morphism $f:T\to S$ such that $f^{-1}(\tau_S)\subset\tau_T.$ It induces
a natural morphism of nice ringed spaces $\tilde{f}:(T,\cA_T)\to (S,\cA_S).$

\begin{prop}\label{prop:map_of_triples_coh}Let $f:(T,\tau_T,n)\to (S,\tau_S,n)$ be a morphism of triples.

1) We have $\rho_S\tilde{f}=f\rho_T.$

2) We have $\tilde{f}i_k=i_kf_0$ for $1\leq k\leq n.$

3) The functor $\tilde{f}_*:\mD^b(\cA_{T})\to \mD^b(\cA_{S})$ is compatible with the semi-orthogonal decompositions
$$\mD^b(\cA_{T})=\langle \mD^b(T_0),\dots,\mD^b(T_0)\rangle,\quad \mD^b(\cA_{S})=\langle \mD^b(S_0),\dots,\mD^b(S_0)\rangle.$$
Moreover, all the induced functors on the semi-orthogonal components are isomorphic to $f_{0*}.$\end{prop}

\begin{proof}1) and 2) are evident.

To show 3), we first apply 2) for $k=1$ and take the direct image:
$$\tilde{f}_*i_{1*}=i_{1*}f_{0*}:\mD^b(T_0)\to \mD^b(\cA_S).$$
Further, it is easy to see that we have an isomorphism of functors between abelian categories.
$$\tilde{f}_*e^*=e^*\tilde{f^\prime_*}:\QCoh(\cA_{T^\prime})\to \QCoh(\cA_S),$$
where $f^\prime:T^\prime\to S^\prime$ is a morphism induced by $f.$ Hence, we have an isomorphism of DG functors
$$\tilde{f}_*e^*=e^*\tilde{f^\prime_*}:\mD^b(\cA_{T^\prime})\to \mD^b(\cA_S).$$ The assertion follows by induction.\end{proof}

\subsection{Auslander-type construction: coherent matrix factorizations}
\label{ssec:A-construction_MF}

Now let $(S,\tau,n)$ be a triple as above, $(S,\cA_S)$ the corresponding nice ringed space, and $W\in\cO(S)$ a regular function on $S.$ Denote by $W_0$ (resp. $W^\prime$)
the pullback of $W$ on $S_0$ (resp. $S^\prime$).

\begin{prop}\label{prop:rho-loc_MF}The DG functor $\rho_{S*}:\mD^{abs}(\cA_S,W)\to\mD^{abs}(S,W)$ is a localization, and its kernel is generated by $i_{k*}(\mD^{abs}(S_0,W_{S_0})),$ $2\leq k\leq n.$\end{prop}

\begin{proof}This follows from Proposition \ref{prop:rho-loc_coh} and Proposition \ref{prop:Serre_quotient_MF}.\end{proof}

\begin{prop}The functors $i_{1*}:\mD^{abs}(S_0,W_0)\to\mD^{abs}(\cA_S,W)$ and $e^*:\mD^{abs}(S_0,W_0)\to\mD^{abs}(\cA_S,W)$ are quasi-fully-faithful, and give a semi-orthogonal decomposition
$$\mD^{abs}(\cA_S,W):=\langle i_{1*}(\mD^{abs}(S_0,W_0)),e^*(\mD^{abs}(\cA_{S^\prime},W^\prime))\rangle.$$
By induction, we get a semi-orthogonal decomposition
$$\mD^{abs}(\cA_S,W)=\langle \mD^{abs}(S_0,W_0),\dots,\mD^{abs}(S_0,W_0)\rangle,$$
where the number of copies equals $n.$
\end{prop}

\begin{proof}First, the functors $e^*$ and $e_*$ are exact on abelian categories and $e_*e^*=\id.$ Hence, $e^*:\mD^{abs}(S_0,W_0)\to\mD^{abs}(\cA_S,W)$ is quasi-fully-faithful.

To show that $i_{1*}$ is quasi-fully-faithful, we note that it has a left adjoint
$$i_1^*=-\stackrel{\bL}{\tens{\cA_S}}i_{1*}\cO_{S_0}:\mD^{abs}(\cA_S,W)\to \mD^{abs}(S_0,W_0),$$
which is well defined since $$i_{1*}\cO_{S_0}=Cone(\cA_S e_2\to \cA_S e_1)$$
as a left $\cA_S$-module. For $\cF\in \mD^{abs}(S_0,W_0),$ we have
$$i_{1*}\cF\stackrel{\bL}{\tens{\cA_S}}\cA_S e_2=0,\quad i_{1*}\cF\stackrel{\bL}{\tens{\cA_S}}\cA_S e_1=\cF.$$
Hence, $$i_1^*i_{1*}(\cF)=\cF,$$ so $i_{1*}$ is quasi-fully-faithful.

Further, Since $e_*i_{1*}=0,$ it follows by adjunction that $e^*(\mD^{abs}(\cA_{S^\prime},W^\prime))$ is left orthogonal to $i_{1*}(\mD^{abs}(S_0,W_0)).$

It remains to show that $e^*(\mD^{abs}(\cA_{S^\prime},W^\prime))$ and $i_{1*}(\mD^{abs}(S_0,W_0))$ generate $\mD^{abs}(\cA_S,W).$ This follows from an exact triangle $$\bL e^*e_*(\cF)\to\cF\to i_{1*}\bL i_1^*\cF,\quad \cF\in D^{abs}_{coh}(\cA_S,W).$$
\end{proof}

\begin{prop}\label{prop:A_S_sm_prop_MF}Suppose that $S_0$ is smooth and the morphism $W_0:S_0\to\A^1$ is proper. Then the $\Z/2$-graded DG category $\mD^{abs}(\cA_S,W)$ is smooth and proper.\end{prop}

\begin{proof}We proceed by induction on $n.$ If $n=1,$ then $\mD^{abs}(\cA_S,W)=\mD^{abs}(S_0,W_0)$ is smooth and proper by Preygel \cite{Pr}.

Further, suppose that Proposition is proved for $n-1\geq 1.$ Then the DG category $\mD^{abs}(\cA_{S^\prime},W^\prime)$ is smooth and proper,
so we have a semi-orthogonal decomposition
$$\mD^{abs}(\cA_S,W):=\langle i_{1*}(\mD^{abs}(S_0,W_0)),e^*(\mD^{abs}(\cA_{S^\prime},W^\prime))\rangle$$
with smooth and proper components. Hence, it remains to show that for any $\cF\in D^{abs}_{coh}(S_0,W_0),$ $\cG\in D^{abs}_{coh}(\cA_{S^\prime},W^\prime)$
we have that $\dim\Hom(i_{1*}\cF,e^*\cG)<\infty.$

By adjunction it suffices to show that $i_1^!:D^{co}(\cA_S,W)\to D^{co}(S_0,W_0)$ takes
$D^{abs}_{coh}(\cA_S,W)$ to $D^{abs}_{coh}(S_0,W_0).$ We have
$$i_1^!=\bR\cHom_{\cA_S}(i_{1*}\cO_{S_0},-)$$ It follows from \cite{KL}, Proposition 5.17, that $i_{1*}(\cO_{S_0})$ is a perfect $\cA_S$-module.
Hence, the functor $i_1^!$ takes $D^{abs}_{coh}(\cA_S,W)$ to $D^{abs}_{coh}(S_0,W_0),$ and we are done.
\end{proof}

\begin{prop}\label{prop:map_of_triples_MF}Let $f:(T,\tau_T,n)\to (S,\tau_S,n)$ be a morphism of triples and $\tilde{f}:(T,\cA_T)\to (S,\cA_S)$ the corresponding map of nice ringed spaces. Let $W$ be a regular function on $S.$

Then the functor $\tilde{f}_*:\mD^{abs}(\cA_{T},f^*W)\to \mD^{abs}(\cA_{S},W)$ is compatible with the semi-orthogonal decompositions
$$\mD^{abs}(\cA_{T},f^*W)=\langle \mD^{abs}(T_0,f_0^*W_0),\dots,\mD^{abs}(T_0,f_0^*W_0)\rangle,$$ $$\mD^{abs}(\cA_{S},W)=\langle \mD^{abs}(S_0,W_0),\dots,\mD^{abs}(S_0,W_0)\rangle.$$
Moreover, all the induced functors on the semi-orthogonal components are isomorphic to $f_{0*}.$\end{prop}

\begin{proof}This is completely analogous to Proposition \ref{prop:map_of_triples_coh}, 3).\end{proof}

\subsection{Categorical blow-ups: coherent sheaves}
\label{ssec:cat_blowup_sheaves}


Let $f:X\to Y$ be a proper morphism of noetherian separated schemes.
The following definition is taken from \cite{KL}.

\begin{defi}\cite[Definition 6.1]{KL} Let $X\to Y$ be a proper morphism. A closed subscheme $S\subset Y$ is called a nonrational locus of $Y$ with respect to $f$ if
the natural morphism
$$I_S\to\bR f_* I_{f^{-1}(S)}$$
is an isomorphism in $D^b_{coh}(Y).$\end{defi}

\begin{prop}\label{prop:nonrat_centers}\cite[Lemma 6.3]{KL} If $f:X\to Y$ be the blow-up of a sheaf of ideals $I$ on $Y.$ Then for $m$ sufficiently large the closed subscheme of $Y$ defined by the ideal
$I^m$ is a nonrational locus with respect to $Y.$\end{prop}


Let $S$ be a nonrational locus of $Y$ with respect to a {\it proper} morphism $f:X\to Y.$ Set $T:=f^{-1}(S)\subset X,$ and denote by $i:S\to Y,$
 $j:T\to X$ the closed embeddings, and by $p:T\to S$ the morphism induced by $f.$  
We have a commutative diagram of DG functors
\begin{equation}\label{eq:comm_square_of_pushforwards}
\begin{CD}
\mD^b(X) @< j_* << \mD^b(T)\\
@V f_* VV @V p_* VV\\
\mD^b(Y) @< i_* << \mD^b(S).
\end{CD}
\end{equation}

We put
$$\cD_{coh}(X,S):=\mD^b(X)\oright_{(\mD^b(T))} \mD^b(S)$$
(as in Definition \ref{defi:gluing_from_a_pair}).
By Lemma \ref{lem:functor_from_comm_diag}, 1) this commutative diagram induces a natural DG functor
$$\pi_*:\cD_{coh}(X,S)\to \mD^b(Y).$$

Similarly, we have the DG category
$$\cD_{coh,T}(X,S):=\mD^b_T(X)\oright_{\mD^b(T)}\mD^b(S),$$
and a DG functor
$$\pi_*:\cD_{coh,T}(X,S)\to \mD^b_S(Y),$$
which we denote by the same letter. It fits into the commutative diagram of DG functors

\begin{equation}\label{eq:comm_support}
\begin{CD}
\cD_{coh,T}(X,S) @>>> \cD_{coh}(X,S)\\
@V\pi_* VV                              @V\pi_* VV\\
\mD^b_S(Y) @>>> \mD^b(Y).
\end{CD}
\end{equation}

Clearly, the horizontal arrows in \eqref{eq:comm_support} are quasi-fully-faithful.

\begin{prop}\label{prop:hom_epi_pi*}The DG functors $\pi_*:\cD_{coh}(X,S)\to \mD^b(Y)$ and $\pi_*:\cD_{coh,T}(X,S)\to \mD^b_S(Y)$ are homological epimorphisms.\end{prop}

\begin{proof} We will prove the assertion for the first functor, and the proof for the second functor is analogous.

By \cite[Lemma 6.4]{KL} we have an exact triangle
$$\cO_Y\to \bR f_*\cO_X\oplus \bR i_*\cO_S\to \bR f_*\cO_T\to \cO_Y[1].$$ Hence, for an object $F\in \mD^b(Y),$ we have an exact triangle
$$F[-1]\to \bR\cHom(\bR f_*\cO_T,F)\to \bR\cHom(\bR f_*\cO_X,F)\oplus \bR\cHom(\bR i_*\cO_S,F)\to F.$$
By Grothendieck duality (see Corollary \ref{cor:projection_formula_coh}), this triangle is naturally isomorphic to
$$F[-1]\to \bR f_*\bR j_*p^!i^!F\to \bR f_*f^!F\oplus\bR i_*i^!F\to F.$$
This in turn implies that the $\mD^b(Y)\mhyphen\mD^b(Y)$-bimodule
\begin{equation*}\Tot(\mD^b(Y)\stackrel{\bL}{\tens{\mD^b(T)}}\mD^b(Y)\to \mD^b(Y)\stackrel{\bL}{\tens{\mD^b(X)}}\mD^b(Y)\oplus \mD^b(Y)\stackrel{\bL}{\tens{\mD^b(S)}}\mD^b(Y)\to \mD^b(Y))\end{equation*}
is acyclic.
Hence, by Lemma \ref{lem:hom_epi_from_gluing} the DG functor $\pi_*:\cD_{coh}(X,S)\to \mD^b(Y)$ is a homological epimorphism.\end{proof}

By Lemma \ref{lem:functor_to_comm_diag} 1), we have a DG quasi-functor
$$\Phi:\mD^b(T)\to\cD_{coh,T}(X,S).$$
By Lemma \ref{lem:functor_to_comm_diag} 2), The composition $\pi_*\Phi:\mD^b(T)\to \mD^b_S(Y)$ is zero in $\Ho(\dgcat_{\mk}).$

Our main result reduces to the following statement.

\begin{theo}\label{th:main_localization}1) Within the above notation, suppose that all infinitesimal neighborhoods $S_n\subset Y,$ $n\geq 1,$ are nonrational loci of $Y$ with respect to $f.$ Then the DG functor
$$\pi_*:\cD_{coh,T}(X,S)\to \mD^b(S)$$ is a localization,
and its kernel is generated by the image of $\Phi:\mD^b(T)\to\cD_{coh,T}(X,S).$

2) Suppose that moreover the morphism $f$ is an isomorphism outside of $S.$ Then the DG functor $$\pi_*:\cD_{coh}(X,S)\to \mD^b(Y)$$ is a localization, and again the kernel of $\pi_*$
is generated by the image of the composition $\Phi:\mD^b(T)\to\cD_{coh,T}(X,S)\hto\cD_{coh}(X,S).$\end{theo}

We break the proof of Theorem \ref{th:main_localization} into several lemmas.
First we reduce to part 1).

\begin{lemma}\label{lem:red_to_1}The statement 2) of Theorem \ref{th:main_localization} follows from the statement 1).\end{lemma}

\begin{proof}Let us note that we have quasi-equivalences $\cD_{coh}(X,S)/\cD_{coh,T}(X,S)\xto{\sim}\mD^b(X\setminus T),$ $\mD^b(Y)/\mD^b(S)\xto{\sim}\mD^b(Y\setminus S).$ By the assumption of Theorem \ref{th:main_localization} 2), the pushforward $\mD^b(X\setminus T)\to\mD^b(Y\setminus S)$ is a quasi-equivalence. The asertion is obtained from a direct application of Lemma \ref{lem:automatic_localization} to the commutative square \eqref{eq:comm_support}. 
\end{proof}

\begin{lemma}\label{lem:colimit_of_neighb}Let $Q$ be a noetherian separated scheme and $Z\subset Q$ a closed subscheme. Then the natural DG functor $$\colim\limits_n \mD^b(Z_n)\to \mD^b_{Z}(Q)$$ is a quasi-equivalence.\end{lemma}

\begin{proof}By \cite[Lemma 2.10]{TV}, the DG category $\colim\limits_n \mD^b(Z_n)$ is weakly pre-triangulated. Its homotopy category is identified with $\colim\limits_n D^b_{coh}(Z_n),$ which is thus triangulated.

The statement of the lemma is therefore equivalent to the equivalence of triangulated categories $\colim\limits_n D^b_{coh}(Z_n)\to D^b_{coh,Z}(Q).$ By \cite[Lemma 2.1]{Or2}, we have an equivalence $$D^b_{coh,Z}(Q)\simeq D^b(\Coh_Z Q).$$
Since $\Coh_Z Q\simeq\colim\limits_n \Coh Z_n,$ we have $D^b(\Coh_Z Q)\simeq\colim\limits_n D^b_{coh}(Z_n).$ This proves the lemma.\end{proof}

Denote by $i_{m,n}:S_m\to S_n,$ $i_n:S_n\to Y,$ $j_{m,n}:T_m\to T_n,$ $j_n:T_n\to X$ the natural inclusions. Also, denote by $p_n:T_n\to S_n$ the natural projections. For any $0<m<n,$ we have a commutative diagram
$$\begin{CD}
T_m @>j_{m,n}>> T_n @>j_n>> X\\
@V p_m VV                              @V p_n VV @V f VV\\
S_m @>i_{m,n}>> S_n @>i_n>> Y.
\end{CD}$$

We put
$$\cD_n:=\cD_{coh}(T_n,S).$$
We have natural DG functors $J_{m,n}:\cD_m\to\cD_n,$ $J_n:\cD_n\to\cD.$ Also, we have
the functors $P_n:\cD_n\to \mD^b(S_n),$ defined in the same way as $\pi_*$ above. Moreover, all these DG functors fit into commutative diagrams
$$\begin{CD}
\cD_m @> J_{m,n} >> \cD_n @> J_n >> \cD_{coh,T}(X,S)\\
@V P_m VV                              @V P_n VV @V \pi_* VV\\
\mD^b(S_m) @>i_{m,n*}>> \mD^b(S_n) @>i_{n*}>> \mD^b_S(Y).
\end{CD}$$

\begin{cor}\label{cor:colim_D_n}The natural DG functor
$$\colim\limits_n \cD_n\to \cD_{coh,T}(X,S)$$
is a quasi-equivalence.\end{cor}
\begin{proof}This follows immediately from Lemma \ref{lem:colimit_of_neighb} and Lemma \ref{lem:filtered_daigrams_gluing}. 
\end{proof}

Since by our assumption $S$ is a nonrational locus of $S_n$ with respect to $p_n:T_n\to S_n,$ we have by Proposition \ref{prop:hom_epi_pi*}
that the DG functor $P_n:\cD_n\to \mD^b(S_n)$ is a homological epimorphism.

As above, we have the DG quasi-functors $\Phi_n:\mD^b(T)\to \cD_n,$ $n\geq 1.$ We have
$$J_{m,n}\Phi_m=\Phi_n,\quad J_n\Phi_n=\Phi\quad\text{in }\Ho(\dgcat_{\mk}).$$

\begin{lemma}\label{lem:red_to_neighb}Let us suppose that all the DG functors
$$P_n:\cD_n\to \mD^b(S_n)$$ are localizations, and the kernel of $P_n$ is generated by $\Phi_n(\mD^b(T)).$ Then the
functor $$\pi_*:\cD_{coh,T}(X,S)\to \mD^b_{coh,S}(Y)$$ is also a localization, and its kernel is generated by $\Phi(\mD^b(T)).$\end{lemma}

\begin{proof}Indeed, by assumption, the DG functor $$P_n:\cD_n/\Phi_n(\mD^b(T))\to \mD^b(S_n)$$
is a quasi-equivalence. Hence, the DG functor
\begin{multline*}\cD_{coh,T}(X,S)/\Phi(\mD^b(T))=(\colim\limits_n \cD_n)/\Phi(\mD^b(T))=\colim\limits_n(\cD_n/\Phi_n(\mD^b(T)))\\
\to \colim\limits_n(\mD^b(S_n))\cong \mD^b_S(Y)\end{multline*}
is a quasi-equivalence (because the DG quotient commutes with colimits). This proves the lemma.\end{proof}

Hence, we reduced the statement of the theorem to proving that the DG functors $P_n$ are localizations with prescribed kernels. We start with the functor $P_1.$

\begin{lemma}\label{lem:P_1}The DG quasi-functor $\Phi_1:\mD^b(T)\to\cD_1$ is quasi-fully-faithful, we have a semi-orthogonal decomposition
$$[\cD_1]=\langle D^b_{coh}(S),\Phi_1(D^b_{coh}(T))\rangle,$$
and the functor $[P_1]$ is the left semi-orthogonal projection onto $D^b_{coh}(S).$ In particular, the DG functor $P_1$ is a localization, and its kernel is generated by $\Phi_1(\mD^b(T)).$\end{lemma}

\begin{proof} This is a direct application of Lemma \ref{lem:gluing_and_mutation}, with $\cA=\mD^b(T),$ $\cB=\mD^b(S),$ $F=\pi_*.$
\end{proof}

The following Lemma is the key technical point in the proof.

\begin{lemma}\label{lem:infinitesimal_extension}Let $g:U\to V$ be a proper morphism of noetherian separated schemes, and $Z\subset V$ a nonrational locus of $V$ with respect to $g.$ Suppose that $U^\prime$ (resp. $V^\prime$) is a square-zero thickening of $U$ (resp. $V$),
and we have a commutative diagram
$$\begin{CD}
U @>\iota_U>> U^\prime\\
@V g VV                              @V g^\prime VV \\
V @>\iota_V>> V^\prime.
\end{CD}$$
Assume that $Z$ is also a nonrational locus of $V^\prime$ with respect to $g^\prime$.
Then we have a commutative square of DG functors
$$\begin{CD}
\cD_{coh}(U,Z) @>J_U>> \cD_{coh}(U^\prime,Z)\\
@V G VV                              @V G^\prime VV \\
\mD^b(V) @>\iota_{V*}>> \mD^b(V^\prime).
\end{CD}$$
Suppose that the DG functor $G$ is a localization. Then the DG functor $G^\prime$ is also a localization, and its kernel is generated by $J_U(\ker G).$
\end{lemma}

\begin{proof}Denote by $H\subset U$ the scheme-theoretic preimage $g^{-1}(Z),$ and by $j_{H,U^\prime}:H\to U^\prime$ the inclusion,
and by $h:H\to Z$ the projection. Also denote by $i_{Z,V^\prime}:Z\to V^\prime$ the natural inclusion.

Let $\cA_{U^\prime}$ (resp. $\cA_{V^\prime},$ $\cA_{H},$ $\cA_{Z}$) be a coherent sheaf of algebras over $\cO_{U'}$ (resp. $\cO_{V'},$ $\cO_H,$ $\cO_Z$) obtained by the Auslander-type construction (Subsection \ref{ssec:A-construction_coh}) from the triple $(U^\prime,I_U,2)$ (resp. $(V^\prime,I_V,2),$ $(H,0,2),$ $(Z,0,2)$). Here $I_U\subset\cO_{U'}$ (resp. $I_V\subset \cO_{V'}$) is the sheaf of ideals defining $U$ (resp. $V$). As in Subsection \ref{ssec:A-construction_coh}, we denote by $\rho_{U^\prime}:(U^\prime,\cA_{U^\prime})\to U^\prime$ (resp. $\rho_{V^\prime},$ $\rho_H,$ $\rho_Z$) the natural morphisms of nice ringed spaces.

We get a commutative diagram

$$\begin{CD}
(H,\cA_{H}) @>\widetilde{j_{H,U^\prime}}>> (U^\prime,\cA_{U^\prime})\\
@V \tilde{h} VV                              @V \tilde{g^\prime} VV \\
(Z,\cA_{Z}) @>\widetilde{i_{Z,V^\prime}}>> (V^\prime,\cA_{V^\prime}).
\end{CD}$$

We define the category $\cD_{coh}(\cA_{U^{\prime}},\cA_{Z})$ as a gluing
$$\cD_{coh}(\cA_{U^{\prime}},\cA_{Z}):=\mD^b(\cA_{U^{\prime}})\oright_{(\mD^b(\cA_H))}\mD^b(\cA_{Z}).$$

By Lemma \ref{lem:functor_from_comm_diag}, we have a DG functor $\tilde{G^\prime}:\cD_{coh}(\cA_{U^{\prime}},\cA_{Z})\to \mD^b(\cA_{V^\prime}).$ The triple of DG functors $(\rho_{U' *},\rho_{H*},\rho_{Z*})$ defines a DG functor $$\pi_U:\cD_{coh}(\cA_{U^{\prime}},\cA_{Z})\to \cD_{coh}(U^\prime,Z).$$
We get the following commutative diagram
\begin{equation}\label{eq:comm_diag}\begin{array}{ccc}
\cD_{coh}(U^\prime,Z) & \stackrel{\pi_U}{\leftarrow} & \cD_{coh}(\cA_{U^{\prime}},\cA_{Z})\\
G^\prime \downarrow & & \downarrow \tilde{G^\prime}\\
\mD^b(V^\prime) & \stackrel{\rho_{V^\prime *}}{\leftarrow} &
\mD^b(\cA_{V^\prime}).
\end{array}\end{equation}

Recall the semi-orthogonal decompositions (Proposition \ref{prop:A_S_SOD_coh})
$$\mD^b(\cA_{U^\prime})=\langle \mD^b(U),\mD^b(U)\rangle,\quad \mD^b(\cA_{Z})=\langle \mD^b(Z),\mD^b(Z)\rangle.$$
By Lemma \ref{lem:diagram_of_twelve_functors}, we have a semi-orthogonal decomposition
\begin{equation}\label{eq:decomp_U^prime}\cD_{coh}(\cA_{U^{\prime}},\cA_{Z})=\langle \cD_{coh}(U,Z),\cD_{coh}(U,Z)\rangle.\end{equation}
Also, we have a semi-orthogonal decomposition
\begin{equation}\label{eq:decomp_V^prime}\mD^b(\cA_{V^\prime})=\langle \mD^b(V),\mD^b(V)\rangle.\end{equation}
The functor $\tilde{G^\prime}$ is compatible with the decompositions \eqref{eq:decomp_U^prime}, \eqref{eq:decomp_V^prime} and induces the functor $G$ on both components.

For $j=1,2,$ we denote by $\iota_{j}:\cD_{coh}(U,Z)\to \cD_{coh}(\cA_{U^\prime},\cA_{Z})$ the functor induced by $i_{j*}:\mD^b(U)\to \mD^b(\cA_{U'}),$ $i_{j*}:\mD^b(H)\to\mD^b(\cA_H),$ $i_{j*}:\mD^b(Z)\to\mD^b(\cA_Z).$ Also, we denote by $\tilde{e}:\cD_{coh}(U,Z)\to \cD_{coh}(\cA_{U'},\cA_Z)$ the functor induced by $e^*:\mD^b(U)\to \mD^b(\cA_{U'}),$ $e^*:\mD^b(H)\to\mD^b(\cA_H),$ $e^*:\mD^b(Z)\to\mD^b(\cA_Z).$ Hence, in terms of the semi-orthogonal decomposition \eqref{eq:decomp_U^prime}, $\iota_1$ is
the embedding of the first semi-orthogonal component, and $\tilde{e}$ is the embedding of the second semi-orthogonal component.

\begin{prop}\label{prop:tildeG'-loc}1) The DG functor $\tilde{G^\prime}:\cD_{coh}(\cA_{U^{\prime}},\cA_{Z})\to \mD^b(\cA_{V^\prime})$ is a localization, with kernel generated by $\iota_1(\ker G),$ $\tilde{e}(\ker G)$.

2) The DG functor $\pi_U:\cD_{coh}(\cA_{U^{\prime}},\cA_{Z})\to \cD_{coh}(U^\prime,Z)$ is a localization, with kernel generated by the image of $\iota_2.$\end{prop}
\begin{proof}1) Analogously to Proposition \ref{prop:hom_epi_pi*}, we see that the DG functor $\tilde{G^\prime}$ is a homological epimorphism. Since $\tilde{G^\prime}$ induces the localization functor $G:\cD_{coh}(U,Z)\to \mD^b(V)$ on both semi-orthogonal components, it follows from Lemma \ref{lem:localizations_SOD} that $\tilde{G^\prime}$ is actually a localization. The assertion about the kernel follows automatically.

2) By Proposition \ref{prop:rho-loc_coh}, the DG functors $\rho_{U^\prime *},$ $\rho_{Z*}$ and $\rho_{H*}$ are localizatins. Hence, by Lemma \ref{lem:functor_between_gluings_from_comm_diags} (part 2) the DG functor  $\pi_U:\cD_{coh}(\cA_{U^{\prime}},\cA_{Z})\to \cD_{coh}(U^\prime,Z)$ is a localization. The assertion about the kernel is obvious.\end{proof}

\begin{cor}\label{cor:G'-loc}The DG functor $G^\prime:\cD_{coh}(U^\prime,Z)\to \mD^b(V^\prime)$ is a localization.\end{cor}

\begin{proof}We know that in the commutative diagram \eqref{eq:comm_diag} the DG functors $\pi_U,$ $\tilde{G^\prime}$ and $\rho_{V*}$ are localizations (by Proposition \ref{prop:tildeG'-loc}). Hence, the DG functor $G^\prime$ is also a localization.\end{proof}

We are left to controlling the kernel of $G^\prime.$

From the commutative diagram \eqref{eq:comm_diag} we have that
$$\ker(G^\prime)=\pi_U(\tilde{G^\prime}^{-1}(\ker \rho_{V^\prime *})),$$ up to direct summands. From the commutative diagram
\begin{equation}\begin{array}{ccc}
\cD_{coh}(U,Z) & \stackrel{\iota_2}{\to} & \cD_{coh}(\cA_{U^\prime},\cA_{Z})\\
G \downarrow & & \downarrow \tilde{G^\prime}\\
\mD^b(V) & \stackrel{i_{2*}}{\to} &
\mD^b(\cA_{V^\prime}),
\end{array}\end{equation}
we get that $\ker \rho_{V^\prime *}$ is generated by $\tilde{G^\prime}\iota_{2}(\cD_{coh}(U,Z)).$ It follows that $\tilde{G^\prime}^{-1}(\ker \rho_{V' *})$ is generated by $\ker \tilde{G^\prime}$ and $\iota_{2}(\cD_{coh}(U,Z)).$ By Proposition \ref{prop:tildeG'-loc}, we know that $\ker\tilde{G^\prime}$ is generated by $\iota_1(\ker G),$ $\tilde{e}(\ker G).$

Therefore, $\ker(G^\prime)$ is generated by the following:

1) $\pi_U\iota_1(\ker G)=J_{U}(\ker G);$

2) $\pi_U \tilde{e}(\ker G)=J_{U}(\ker G);$

3) $\pi_U\iota_2(\cD_{coh}(U,Z))=0.$

Therefore, $\ker(G')$ is generated by $J_U(\ker G).$ This proves the lemma.
\end{proof}

\begin{proof}[Proof of Theorem \ref{th:main_localization}]
By Lemma \ref{lem:P_1} and Lemma \ref{lem:infinitesimal_extension} we get by induction that each DG functor $P_n:\cD_{coh}(T_n,S)\to \mD^b(S_n)$ is a localization, and $\ker P_n$ is generated by $J_{1,n}\Phi_1(\mD^b(T))=\Phi_n(\mD^b(T)).$ Therefore, by Lemma \ref{lem:red_to_neighb} the DG functor
$\pi_*:\cD_{coh,T}(X,S)\to \mD^b_S(Y)$ is a localization, and its kernel is generated by $\Phi(\mD^b(T)).$ This proves part 1) of Theorem.

By Lemma \ref{lem:red_to_1} part 1) implies part 2). Theorem is proved.
\end{proof}

We also formulate here a result which is analogous to Theorem \ref{th:main_localization} but technically simpler to prove.

\begin{theo}\label{th:BO_special_case}1) Let $f:X\to Y$ be a proper morphism such that $\bR f_*\cO_X\cong \cO_Y.$ Assume that there is a subscheme $S\subset Y,$ such that all its infinitesimal neighborhoods $S_n,$ $n\geq 1,$ are nonrational loci of $Y$ with respect to $f.$ Again, we have a Cartesian square
$$
\begin{CD}X @< j << T\\
@V f VV @V p VV\\
Y @< i << S.\end{CD}
$$
Assume that the functor $\bR p_*:D^b_{coh}(T)\to D^b_{coh}(S)$ is a localization. Then the functor $\bR f_*:D^b_{coh,T}(X)\to D^b_{coh,S}(Y)$
is also a localization and $\ker\bR f_*$ is generated by $j_*(\ker \bR p_*).$

2) Suppose that moreover the morphism $f$ is an isomorphism outside of $S.$ Then the functor $\bR f_*:D^b_{coh}(X)\to D^b_{coh}(Y)$
is a localization and $\ker\bR f_*$ is generated by $j_*(\ker \bR p_*).$\end{theo}

\begin{proof}Analogously to Lemma \ref{lem:red_to_1}, we reduce the part 2) to the part 1).

The proof of part 1) follows essentially the same steps as for Theorem \ref{th:main_localization} 1), and it considerably simplifies because we do not need to consider gluings. Thus, instead of $\cD_{coh}(X,S)$ we have $\mD^b(X),$ and instead of $\cD_{coh}(T_n,S)$ we have $\mD^b(T_n).$ Lemma \ref{lem:P_1} is not needed in this framework because we already assumed that $\bR p_*$ is a localization. In the key technical step, Lemma \ref{lem:infinitesimal_extension}, we assume that the morphisms $g$ and $g'$ satisfy $\bR g_*\cO_U=\cO_V,$ $\bR g'_*\cO_{U'}=\cO_{V'},$ and we take 
$\mD^b(U)$ (resp. $\mD^b(U')$) instead of $\cD_{coh}(U,Z)$ (resp. $\cD_{coh}(U',Z)$). All the other arguments are the same.\end{proof}

\subsection{Categorical blow-ups: matrix factorizations}
\label{ssec:cat_blowup_MF}

As above, let $S$ be a nonrational locus of $Y$ with respect to a proper morphism $f:X\to Y.$ Again we have $T:=f^{-1}(S)\subset X,$ and denote by $i:S\to Y,$
 $j:T\to X$ the closed embeddings, and by $p:T\to S$ the morphism induced by $f.$

We fix a regular function $W\in\cO(Y).$ Further, we will denote by $W_X\in\cO(X),$ $W_S\in\cO(S),$ $W_T\in\cO(T)$ (and so on) the pullbacks of $W$ under natural morphisms.


We have a commutative diagram of DG functors
$$
\begin{CD}
\mD^{abs}(X,W_X) @< j_* << \mD^{abs}(T,W_T)\\
@V f_* VV @V p_* VV\\
\mD^{abs}(Y,W) @< i_* << \mD^{abs}(S,W_S).
\end{CD}
$$
We put
$$\cD_{coh}(X,S,W):=\mD^{abs}(X,W_X)\oright_{(\mD^{abs}(T,W_T))} \mD^{abs}(S,W_S).$$
By Lemma \ref{lem:functor_from_comm_diag}, 1) we have a natural DG functor
$$\pi_*:\cD_{coh}(X,S,W)\to \mD^{abs}(Y,W).$$

Similarly, we have the DG category
$$\cD_{coh,T}(X,S,W):=\mD^{abs}_T(X,W_X)\oright_{(\mD^{abs}(T,W_T))}\mD^{abs}(S,W_S),$$
Also, we have a DG functor
$$\pi_*:\cD_{coh,T}(X,S,W)\to \mD^{abs}_S(Y,W),$$
which we denote by the same letter. There is a commutative diagram of DG functors

\begin{equation}\label{eq:comm_supportMF}
\begin{CD}
\cD_{coh,T}(X,S,W) @>>> \cD_{coh}(X,S,W)\\
@V\pi_* VV                              @V\pi_* VV\\
\mD^{abs}_S(Y,W) @>>> \mD^{abs}(Y,W).
\end{CD}
\end{equation}

The horizontal arrows in \eqref{eq:comm_supportMF} are quasi-fully-faithful.

\begin{prop}\label{prop:hom_epi_pi*MF}The DG functors $\pi_*:\cD_{coh}(X,S,W)\to \mD^{abs}(Y,W)$ and $\pi_*:\cD_{coh,T}(X,S,W)\to \mD^{abs}_S(Y)$ are homological epimorphisms.\end{prop}

\begin{proof} The proof is the same as for Proposition \ref{prop:hom_epi_pi*}. The only difference is that we use the Grothendieck duality for matrix factorizations (see Corollary \ref{cor:projection_formula_MF}).\end{proof}

By Lemma \ref{lem:functor_to_comm_diag}, we have a quasi-functor
$$\Phi:\mD^{abs}(T,W_T)\to\cD_{coh,T}(X,S,W),$$
and the composition $\pi_*\Phi:\mD^b(T)\to \mD^b_S(Y)$ is zero in $\Ho(\dgcat_{\mk}).$

Our main result reduces to the following statement.

\begin{theo}\label{th:main_localizationMF}1) Within the above notation, suppose that all infinitesimal neighborhoods $S_n\subset Y,$ $n\geq 1,$ are nonrational loci of $Y$ with respect to $f.$ Then the DG functor
$$\pi_*:\cD_{coh,T}(X,S,W)\to \mD^{abs}(Y,W)$$ is a localization,
and its kernel is generated by the image of $\Phi:\mD^{abs}(T,W_T)\to\cD_{coh,T}(X,S,W).$

2) Suppose that moreover the morphism $f$ is an isomorphism outside of $S.$ Then the DG functor $$\pi_*:\cD_{coh}(X,S,W)\to \mD^{abs}(Y,W)$$ is a localization, and again the kernel of $\pi_*$
is generated by the image of the composition $\Phi:\mD^{abs}(T,W_T)\to\cD_{coh,T}(X,S,W)\to\cD_{coh}(X,S,W).$\end{theo}

First we reduce to 1).

\begin{lemma}\label{lem:red_to_1MF}The statement 2) of Theorem \ref{th:main_localizationMF} follows from the statement 1).\end{lemma}

\begin{proof}The proof is the same as for Lemma \ref{lem:red_to_1}.
\end{proof}

\begin{lemma}\label{lem:colimit_of_neighbMF}Let $Q$ be a noetherian separated scheme, $Z\subset Q$ a closed subscheme and $W_Q\in\cO(Q)$ a regular function. Denote by $W_{Z_n}\in\cO(Z_n)$ the pullback of $W_Q$ onto $Z_n.$ Then the natural DG functor $$\colim\limits_n \mD^{abs}(Z_n,W_{Z_n})\to \mD^{abs}_{Z}(Q,W_Q)$$ is a Morita equivalence.\end{lemma}

\begin{proof}As in Lemma \ref{lem:colimit_of_neighb}, the statement is equivalent to the equivalence of triangulated categories $\colim\limits_n D^{abs}_{coh}(Z_n,W_{Z_n})\to D^{abs}_{coh,Z}(W_Q),$ up to direct summands. Analogously to \cite[Lemma 2.1]{Or2}, the functor $$D^{abs}(\coh_Z Q,W_Q)\to D^{abs}_{coh,Z}(Q,W_Q)$$ is an equivalence up to direct summands. Since
$\Coh_Z Q=\colim\limits_n \Coh Z_n,$ the functor
$\colim\limits_n D^{abs}_{coh}(Z_n,W_{Z_n})\to D^{abs}(\Coh_Z Q,W_Q)$ is an equivalence.
This proves the lemma.\end{proof}

As in the proof of Theorem \ref{th:main_localization}, denote by $i_{m,n}:S_m\to S_n,$ $i_n:S_n\to Y,$ $j_{m,n}:T_m\to T_n,$ $j_n:T_n\to X$ the natural inclusions, and by $p_n:T_n\to S_n$ the natural projections.
We put
$$\cD_n(W):=\cD_{coh}(T_n,S,W_{S_n}).$$

We have natural DG functors $J_{m,n}:\cD_m(W)\to\cD_n(W),$ $J_n:\cD_n(W)\to\cD_{coh}(X,S,W).$ Also, we have
the functors $P_n:\cD_n(W)\to \mD^{abs}(S_n,W_{S_n}),$ defined in the same way as $\pi_*$ above. All these DG functors fit into commutative diagrams
$$\begin{CD}
\cD_m(W) @> J_{m,n} >> \cD_n(W) @> J_n >> \cD_{coh,T}(X,S,W)\\
@V P_m VV                              @V P_n VV @V \pi_* VV\\
\mD^{abs}(S_m,W_{S_m}) @>i_{m,n*}>> \mD^{abs}(S_n,W_{S_n}) @>i_{n*}>> \mD^{abs}_S(Y,W).
\end{CD}$$

\begin{cor}The natural DG functor
$$\colim\limits_n \cD_n(W)\to \cD_{coh,T}(X,S,W)$$
is a Morita equivalence.\end{cor}
\begin{proof}This follows from Lemma \ref{lem:colimit_of_neighbMF} and Lemma \ref{lem:filtered_daigrams_gluing}.\end{proof}

Since by our assumption $S$ is a nonrational locus of $S_n$ with respect to $p_n:T_n\to S_n,$ we have by Proposition \ref{prop:hom_epi_pi*MF}
that the DG functor $P_n:\cD_n(W)\to \mD^{abs}(S_n,W_{S_n})$ is a homological epimorphism.

Define the DG quasi-functor $\Phi_n:\mD^{abs}(T,W_T)\to \cD_n(W)$ in the same way as the DG quasi-functor $\Phi$ above.  We have that
$$J_{m,n}\Phi_m=\Phi_n,\quad J_n\Phi_n=\Phi\quad\text{in }\Ho(\dgcat_{\mk}).$$

\begin{lemma}\label{lem:red_to_neighbMF}Assume that all the DG functors
$$P_n:\cD_n(W)\to \mD^{abs}(S_n,W_{S_n})$$ are localizations, and the kernel of $P_n$ is generated by $\Phi_n(\mD^{abs}(T,W_T)).$ Then the
functor $$\pi_*:\cD_{coh,T}(X,S,W)\to \mD^{abs}_S(Y,W)$$ is also a localization, and its kernel is generated by $\Phi(\mD^{abs}(T,W_T)).$\end{lemma}

\begin{proof}The proof is the same as for Lemma \ref{lem:red_to_neighb}.\end{proof}

Hence, we reduced the statement of the theorem to proving that the DG functors $P_n:\cD_n(W)\to \mD^{abs}(S_n,W_{S_n})$ are localizations and controlling their kernels. We start with the functor $P_1.$

\begin{lemma}\label{lem:P_1MF}The DG quasi-functor $\Phi_1:\mD^b(T)\to\cD_1$ is quasi-fully-faithful, we have a semi-orthogonal decomposition
$$\cD_1=\langle \mD^{abs}(S,W_S),\Phi_1(\mD^{abs}(T,W_T))\rangle,$$
and the DG functor $P_1$ is the left semi-orthogonal projection onto $\mD^{abs}(S,W_S).$ In particular, the DG functor $P_1$ is a localization, and its kernel is generated by $\Phi_1(\mD^{abs}(T,W_T)).$\end{lemma}

\begin{proof}This is a direct application of Lemma \ref{lem:gluing_and_mutation}, with $\cA=\mD^{abs}(T,W_T),$ $\cB=\mD^{abs}(S,W_S),$ $F=\pi_*.$
\end{proof}

The following Lemma is the analogue of Lemma \ref{lem:infinitesimal_extension} for matrix factorizations.

\begin{lemma}\label{lem:infinitesimal_extensionMF}Let $g:U\to V$ be a proper morphism of noetherian separated schemes, and $Z\subset V$ a nonrational locus of $V$ with respect to $g.$ Suppose that $U^\prime$ (resp. $V^\prime$) is a square-zero thickening of $U$ (resp. $V$),
and we have a commutative diagram
$$\begin{CD}
U @>\iota_U>> U^\prime\\
@V g VV                              @V g^\prime VV \\
V @>\iota_V>> V^\prime.
\end{CD}$$
Assume that $Z$ is also a nonrational locus of $V^\prime$ with respect to $g^\prime$. let $W_{V^\prime}\in\cO(V^\prime)$ be a regular function, and denote by $W_U\in\cO(U),$ $W_V\in\cO(V),$ $W_{U^\prime}\in\cO(U^\prime)$ its pullbacks.
Then we have a commutative square of DG functors
$$\begin{CD}
\cD_{coh}(U,Z,W_V) @>J_U>> \cD_{coh}(U^\prime,Z,W_{V^\prime})\\
@V G VV                              @V G^\prime VV \\
\mD^{abs}(V,W_V) @>\iota_{V*}>> \mD^{abs}(V^\prime,W_{V^\prime}).
\end{CD}$$
Suppose that the DG functor $G$ is a localization. Then the DG functor $G^\prime$ is also a localization, and its kernel is generated by $J_U(\ker G).$
\end{lemma}

\begin{proof} The proof is literally the same as for Lemma \ref{lem:infinitesimal_extension}.
\end{proof}

\begin{proof}[Proof of Theorem \ref{th:main_localizationMF}]
By Lemma \ref{lem:P_1MF} and Lemma \ref{lem:infinitesimal_extensionMF} we get by induction that each DG functor $P_n:\cD_{coh}(T_n,S,W_{S_n})\to \mD^{abs}(S_n,W_{S_n})$ is a localization, and $\ker P_n$ is generated by $J_{1,n}\Phi_1(\mD^{abs}(T,W_T))=\Phi_n(\mD^{abs}(T,W_T)).$ Therefore, by Lemma \ref{lem:red_to_neighbMF} the DG functor
$\pi_*:\cD_{coh,T}(X,S,W)\to \mD^{abs}_S(Y,W)$ is a localization, and its kernel is generated by $\Phi(\mD^{abs}(T,W_T)).$ This proves part 1) of Theorem.

By Lemma \ref{lem:red_to_1MF} part 1) implies part 2). Theorem is proved.
\end{proof}

\subsection{The construction of a smooth categorical compactification}
\label{ssec:constr_of_sm_comp}

In this Subsection we prove our main result.

\begin{theo}\label{th:smooth_compactification}Let $Y$ be a smooth separated scheme of finite type over a field $\mk$ of characteristic zero. Then

1) the DG category $\mD^b(Y)$ has a smooth categorical compactification of the form $\mD^b(\tilde{Y})\to \mD^b(Y),$ where
$\tilde{Y}$ is a smooth and proper variety.

2) for any regular function $W\in\cO(Y)$ the D($\Z/2$-)G category $\mD^{abs}(Y,W)$ has a $\Z/2$-graded smooth categorical compactification $C_W\to \mD^{abs}(X,W),$
with a semi-orthogonal decomposition $C_W=\langle \mD^{abs}(V_1,W_1),\dots,\mD^{abs}(V_m,W_m)\rangle,$ where each $V_i$ is a $\mk$-smooth variety and the morphisms $W_i:V_i\to\A_{\mk}^1$ are proper.\end{theo}

\begin{proof} 1). Our first observation is that the scheme $Y$ can be assumed to be proper. Indeed, if $Y$ is not proper, then by Nagata \cite{Nag} we can take some compactification $Y\subset \overline{Y},$ so that the restriction DG functor $\mD^b(\overline{Y})\to \mD^b(Y)$ is a localization, and the kernel $\mD^b_{\overline{Y}\setminus Y}(\overline{Y})$ is generated by a single object. Thus, $Y$ can be replaced by $\bar{Y}.$

From now on, we assume that $Y$ is proper. By \cite[Theorem 4.15]{Or}, it is sufficient to construct a smooth categorical compactification $C\to \mD^b(Y),$ such that $C$ has a semi-orthogonal decomposition
$C=\langle \mD^b(X_1),\dots,\mD^b(X_m)\rangle,$ where each $X_i$ is smooth and proper variety. We will obtain the DG category $C$ by the same construction as the Kuznetsov-Lunts categorical resolution \cite{KL}, with a slight restriction on the choice of parameters (see below). Also, our description is a bit different since we are dealing with derived categories of coherent sheaves, instead of perfect complexes.  

By \cite[Theorem 1.6]{BM}, there is a sequence of blow-ups with smooth centers
$$Y_n\to Y_{n-1}\to\dots\to Y_1\to Y,$$ such that $(Y_n)_{red}$ is smooth. We proceed by induction on $n.$

The base of induction is $n=0.$ In this case $Y_{red}$ is smooth and proper. Take the nilpotent radical $\cI\subset \cO_Y,$ and assume that $\cI^l=0.$ Applying the Auslander-type construction to the triple $(Y,\cI,l),$ we get a nice ringed space $(Y,\cA_Y)$ with a morphism $\rho_Y:(Y,\cA_Y)\to Y.$  By Proposition \ref{prop:rho-loc_coh} the DG functor $\rho_{Y*}:\mD^b(\cA_{Y})\to \mD^b(Y)$ is a localization, and the kernel is generated by a single object. By Proposition \ref{prop:A_S_sm_prop_coh} the DG category $\mD^b(\cA_Y)$ is smooth and proper. By Proposition \ref{prop:A_S_SOD_coh}, the DG category $\mD^b(\cA_{Y})$ has a semi-orthogonal decomposition
$$\mD^b(\cA_{Y})=\langle \mD^b(Y_{red}),\dots,\mD^b(Y_{red}))\rangle,$$
where the number of components equals to $l.$ This proves the base of induction.

Now assume that the statement of induction is proved for some $n.$ We prove it for $n+1.$ Assume that the first blow-up $f:X=Y_1\to Y$ has a smooth center $Z\subset Y.$ By Proposition \ref{prop:nonrat_centers}, there is $l>0$ such that for all $k\geq l$ the infinitesimal neighborhood $Z_k$ of $Z$ is a nonrational locus of $Y$ with respect to $f.$ As in Subsection \ref{ssec:cat_blowup_sheaves} we have a DG category $\cD_{coh}(X,Z_l),$ and by Theorem \ref{th:main_localization} the DG functor
$$\pi_*:\cD_{coh}(X,Z_l)\to \mD^b(Y)$$ is a localization, and $\ker(\pi_*)$ is generated by a single object.

We would like to modify the DG category $\cD_{coh}(X,Z_l).$ Let us put $D:=f^{-1}(Z).$ Then $D_l=f^{-1}(Z_l).$ Let us take the nice ringed spaces $(D_l,\cA_{D_l})$
and $(Z_l,\cA_{Z_l}),$ associated to the triples $(D_l,I_D,l)$ and $(Z_l,I_Z,l)$ respectively. We have a commutative diagram

$$\begin{CD}
X @< j << D_l @< \rho_{D_l} << (D_l,\cA_{D_l})\\
@V f VV        @V p VV          @V \tilde{p} VV\\
Y @< i << Z_l @<\rho_{Z_l}<< (Z_l,\cA_{Z_l}).
\end{CD}$$

We put $$\cD_{coh}(X,\cA_{Z_l}):=\mD^b(X)\oright_{(\mD^b(\cA_{D_l}))} \mD^b(\cA_{Z_l}),$$
By Lemma \ref{lem:functor_between_gluings_from_comm_diags} 1), we have the DG functor
$$\rho(X,Z_l):\cD_{coh}(X,\cA_{Z_l})\to\cD_{coh}(X,Z_l).$$
By Lemma \ref{lem:functor_between_gluings_from_comm_diags} 2), and Proposition \ref{prop:rho-loc_coh}, the functor $\rho(X,Z_l)$ is a localization and its kernel is generated by a single object. Hence,
 the composition $$\tilde{\pi}_*:=\pi_*\circ\rho(X,Z_l):\cD_{coh}(X,\cA_{Z_l})\to\mD^b(Y)$$ is also a localization, with kernel generated by a single object.

Note that the morphism $j\rho_{D_l}:(D_l,\cA_{D_l})\to X$ has finite Tor-dimension (more precisely, its Tor-dimension equals $1$). Hence, the functor $(j\rho_{D_l})_*:D^b_{coh}(\cA_{D_l})\to D^b_{coh}(X)$ has a left adjoint $\bL(j\rho_{D_l})^*:D^b_{coh}(X)\to D^b_{coh}(\cA_{D_l}).$ It is given by a DG bimodule $M_L\in\mD^b(X)\mhyphen\Mod\mhyphen\mD^b(\cA_{D_l})$ which is left adjoint to 
$N_{(j\rho_{Z_l})_*}\in D(\mD^b(X)\otimes \mD^b(\cA_{D_l})^{op}).$
The DG bimodule
$$\tilde{M}_L:=M_L\stackrel{\bL}{\otimes}_{\mD^b(\cA_{D_l})} N_{\tilde{p}_*}\in \mD^b(X)\mhyphen\Mod\mhyphen\mD^b(\cA_{Z_l})$$
is then left adjoint to $\mD^b(\cA_{Z_l})\stackrel{\bL}{\otimes}_{\mD^b(\cA_{D_l})}\mD^b(X).$ By Lemma \ref{lem:adjoint_gluing}, we have a natural isomorphism
$$\cD_{coh}(X,\cA_{Z_l})\cong \mD^b(\cA_{Z_l})\oright_{\tilde{M}_L}\mD^b(X)\quad\text{in }\Ho(\dgcat_{\mk}).$$
By induction hypothesis, there is a smooth categorical compactification
$\pi^\prime_*:C^\prime\to \mD^b(X),$ with a semi-orthogonal decomposition
$C^\prime=\langle \mD^b(X_1),\dots,\mD^b(X_{m^\prime})\rangle.$ Define the DG category
$$C:=\mD^b(\cA_{Z_l})\oright_{N_{\pi^\prime_*}\stackrel{\bL}{\otimes}_{\mD^b(X)} \tilde{M}_L} C^\prime.$$
The bimodule $C'\mhyphen\mD^b(\cA_{Z_l})$-bimodule $N_{\pi^\prime_*}\stackrel{\bL}{\otimes}_{\mD^b(X)} \tilde{M}_L$ is perfect, because it corresponds to an exact functor $\bR\tilde{p}_*\circ\bL(j\rho_{D_l})^*\circ\pi'_*.$ Thus, the DG category $C$ is smooth and proper. Moreover, it has a semi-orthogonal decomposition
\begin{equation}\label{eq:SOD_smooth_compactification}C=\langle \mD^b(Z),\dots,\mD^b(Z), \mD^b(X_1),\dots,\mD^b(X_{m^\prime})\rangle,\end{equation} where the number of copies of $\mD^b(Z)$ equals to $l.$

Since $\pi^\prime_*:C^\prime\to \mD^b(X)$ is a localization, we have a natural isomorphism
$$(\id\otimes\pi^{'op}_*)^*(N_{\pi^\prime_*}\stackrel{\bL}{\tens{\mD^b(X)}}\tilde{M}_L)\xto{\sim} \tilde{M}_L\quad\text{in }D(\mD^b(\cA_{Z_l})\otimes\mD^b(X)^{op}).$$
By Lemma \ref{lem:functor_from_gluing} we have a DG quasi-functor
$$\Pi_*:C\to \cD_{coh}(X,\cA_{Z_l})\simeq \mD^b(\cA_{Z_l})\oright_{\tilde{M}_L}\mD^b(X).$$
It is a localization by Lemma \ref{lem:localizations_SOD}, and its kernel is generated by a single object.

Therefore, the composition DG quasi-functor $\tilde{\pi}_*\Pi_*:C\to \mD^b(Y)$ is a localization, the category $C$ has a semi-orthogonal decomposition \eqref{eq:SOD_smooth_compactification}, and the kernel of $\tilde{\pi}_*\Pi_*$ is generated by a single object. So, this DG quasi-functor is a smooth categorical compactification of $\mD^b(Y)$ with required properties. This proves the statement of induction for $n+1.$
The part 1) of the theorem is proved.

The part 2) is proved in a completely analogous way. We first reduce to the case when the morphism $W:Y\to\A^1$ is proper, and then proceed as in the proof of part 1).
\end{proof}

\begin{remark}\label{rem:choice_of_parameters}Note that in the proof of Theorem \ref{th:smooth_compactification} we choose a parameter $k>0$ such that not only $Z_k\subset Y$ is a nonrational locus of $Y$ with respect to $f:X\to Y,$ but also all the infinitesimal neighborhoods $Z_l$ with $l\geq k.$ In the construction of a categorical resolution of \cite{KL}, it is only needed that $Z_k$ is a nonrational locus.\end{remark}


\appendix

\section{Matrix factorizations}
\label{app:MF}

Let $(\cC,W)$ be an abelian $\Z_+$-category. Recall that in Section \ref{sec:exotic_derived} we defined the DG category $MF_{dg}(\cC,W),$ the homotopy category $K(\cC,W)$ and the absolute derived category $D^{abs}(\cC,W).$ We denote by $MF_{ab}(\cC,W):=Z^0(MF_{dg}(\cC,W))$ the abelian category of matrix factorizations.

When $\cC$ is small, we have a totalization functor
$$\Tot:D^b(MF_{ab}(\cC,W))\to D^{abs}(\cC,W).$$
it is given by \begin{equation}\label{eq:Tot_functor}\Tot(F^{\bullet})^{ev}=\bigoplus\limits_{n\even}(F^n)^{ev}\oplus \bigoplus\limits_{n\odd}(F^n)^{odd},\quad \Tot(F^{\bullet})^{odd}=\bigoplus\limits_{n\even}(F^n)^{odd}\oplus \bigoplus\limits_{n\odd}(F^n)^{ev}.\end{equation}
The non-zero components of the "differential" on $\Tot(F^{\bullet})$ are given by $(-1)^n\delta_{F^n}:F^n\to F^n,$ and $d:F^n\to F^{n+1}.$ Recall (from Definition \ref{defi:MF}) that for a matrix factorization $F$ we denote by $F^{ev}$ (resp. $F^{odd}$) its even (resp. odd) component.

In the case when $\cC$ has small coproducts and exact direct limits (AB5), we have a direct sum totalization functor
$$\Tot^{\oplus}:D^{co}(MF_{ab}(\cC,W))\to D^{co}(\cC,W).$$
The components of $Tot^{\oplus}(F^{\bullet})$ are defined by the same formulas as in \eqref{eq:Tot_functor} (but now the direct sums are infinite), and the "differential" is defined in the same way.

We recall the notion of a Frobenius category and its stable triangulated category \cite{Hap}.

\begin{defi}\label{defi:Frobenius_category}1) An exact category $\cA$ is called a Frobenius category if it has enough injective objects, enough projective objects, and these two classes of objects coincide.

2) If $\cA$ is a Frobenius exact category, its stable triangulated category $\un{\cA}$ is defined as follows. First, $Ob(\un{\cA})=Ob(\cA).$ Further, for $X,Y\in\cA$ the abelian group of morphisms $\Hom_{\un{\cA}}(X,Y)$ is defined as the quotient of $\Hom_{\cA}(X,Y)$ by the subgroup of morphisms which are the compositions $X\to I\to Y$ for some projective-injective object $I\in\cA.$ The composition is induced by that in $\cA.$ The shift functor is defined by choosing a conflation
$$0\to X\to I\to X[1]\to 0,$$
for each $X\in Ob(\un{\cA}).$ Further, each triangle is isomorphic to a standard one $(\overline{i},\overline{p},\overline{e}),$ given by a morphism of conflations
$$
\begin{CD}
0 @>>> X @> i >> Y @> p >> Z @>>> 0\\
@.      @V\id VV    @VVV   @V e VV @.\\
0 @>>> X @>>> I @>>> X[1] @>>> 0.
\end{CD}
$$\end{defi}

We also recall other versions of the definition of stable category of a Frobenius category. Denote by $K_{ac}(\Inj \cA)$ the homotopy category of acyclic complexes of injective objects in $\cA.$ This is a triangulated category and the functor
$$Z^0:K_{ac}(\Inj \cA)\to \un{\cA}$$
is an equivalence of triangulated categories.

Also, take the bounded derived category $D^b(\cA),$ and its full subcategory
$K^b(\Inj \cA)\subset D^b(\cA).$ We have a functor
\begin{equation}\label{eq:equiv_stable_cat}F:\un{\cA}\to D^b(\cA)/K^b(\Inj \cA),\quad F(X)=X[0].\end{equation}
The functor $F$ is also an equivalence of triangulated categories.

Let now $(\cC,W)$ be an additive $\Z_+$-category. Denote by $MF_{ex}(\cC,W)$ the exact category of matrix factorizations, with the following exact structure: a sequence $0\to F_1\to F_2\to F_3\to 0$ is a conflation in $MF_{ex}(\cC,W)$ if it is componentwise split exact in $\cC.$ In particular, if $\cC$ is abelian, then conflations in $MF_{ex}(\cC,W)$ are also short exact sequences in $MF_{ab}(\cC,W).$

\begin{prop}\label{prop:MF_frob_cat}Let $(\cC,W)$ be an additive $\Z_+$-category. Then the exact category
$MF_{ex}(\cC,W)$ is a Frobenius category. Moreover, projective-injective objects are precisely null-homotopic matrix factorizations, and the stable category of $MF_{ex}(\cC,W)$ is identified with $K(\cC,W).$\end{prop}
\begin{proof}Note that for any matrix factorization $F$ the inclusion $F\to Cone(\id_F)$ is an inflation, and the projection $Fiber(\id_F)=Cone(\id_{F[-1]})\to F$ is a deflation. The objects $Cone(id_F)$ are projective-injective since
$$\Hom_{MF_{ex}(\cC,W)}(Cone(\id_F),G)\cong\Hom_{\cC}(F^{ev},G^{odd})\oplus \Hom_{\cC}(F^{odd},G^{ev}),$$
$$\Hom_{MF_{ex}(\cC,W)}(G,Cone(\id_F))\cong\Hom_{\cC}(G^{ev},F^{ev})\oplus \Hom_{\cC}(G^{odd},F^{odd}).$$
It follows that both projective and injective objects of $MF_{ex}(\cC,W)$ are exactly the direct summands of objects of the form $Cone(\id_F),$ and the category $MF_{ex}(\cC,W)$ has enough projectives and enough injectives. Thus, $MF_{ex}(\cC,W)$ is a Frobenius category.

Since all the objects $Cone(\id_F)$ are null-homotopic, we have a {\it full} functor $\pr:\un{MF_{ex}(\cC,W)}\to K(\cC,W),$ which is identity on objects. Further, any null-homotopic morphism $\varphi:F\to F'$ in $MF_{ex}(\cC,W)$ can be factored through $Cone(\id_F).$ It follows that the functor $\pr$ is also faithful, hence it is an equivalence. 

Finally, the projective-injective objects of $MF_{ex}(\cC,W)$ are exactly the objects which are isomorphic to zero in $K(\cC,W),$ i.e. null-homotopic matrix factorizations.\end{proof}

We always have an exact totalization functor
$$\Tot:D^b(MF_{ex}(\cC,W))\to K(\cC,W).$$
If $\cC$ has small coproducts, then we have a direct sum totalization functor
$$\Tot^{\oplus}:D^{co}(MF_{ex}(\cC,W))\to K(\cC,W).$$

\begin{prop}\label{prop:Tot_localizations} Let $(\cC,W)$ be an additive $\Z_+$-category.

1) If $\cC$ is small, then the  totalization functor
$$\Tot:D^b(MF_{ex}(\cC,W))\to K(\cC,W)$$ is a localization,
and its kernel is generated by objects of the form $Cone(\id_F)\in MF_{ex}(\cC,W),$ $F\in MF_{ex}(\cC,W).$

2) If $\cC$ is small and abelian, then the  totalization functor
$$\Tot:D^b(MF_{ab}(\cC,W))\to D^{abs}(\cC,W)$$ is a localization,
and its kernel is generated by objects of the form $Cone(\id_F)\in MF_{ab}(\cC,W),$ $F\in MF_{ab}(\cC,W).$

3) If $\cC$ has small coproducts, then the direct sum totalization functor
$$\Tot^{\oplus}:D^{co}(MF_{ex}(\cC,W))\to K(\cC,W)$$ is a localization, and its kernel is generated as a localizing subcategory by objects of the form $Cone(\id_F)\in MF_{ab}(\cC,W),$ $F\in MF_{ab}(\cC,W).$

4) If $\cC$ is abelain and has exact small coproducts, then the direct sum totalization functor
$$\Tot^{\oplus}:D^{co}(MF_{ab}(\cC,W))\to D^{co}(\cC,W)$$ is a localization, and its kernel is generated as a localizing subcategory by objects of the form $Cone(\id_F)\in MF_{ab}(\cC,W),$ $F\in MF_{ab}(\cC,W).$\end{prop}

\begin{proof}1) By Proposition \ref{prop:MF_frob_cat}, the category $MF_{ex}(\cC,W)$ is Frobenius, and its stable category is $K(\cC,W).$ Clearly, the functor
$\Tot$ vanishes on $K^b(\Inj MF_{ex}(\cC,W))\subset D^b(MF_{ex}(\cC,W)).$ The induced functor
$$\Tot:D^b(MF_{ex}(\cC,W))/K^b(\Inj MF_{ex}(\cC,W))\to K(\cC,W)$$
is a quasi-inverse to the functor \eqref{eq:equiv_stable_cat} (for $\cA=MF_{ex}(\cC,W))$), hence it is an equivalence. Since the class $\Inj MF_{ex}(\cC,W)$
is additively generated by objects $Cone(\id_F),$ the assertion is proved.

2) This follows directly from 1). Indeed, the natural functor
$$D^b(MF_{ex}(\cC,W))\to D^b(MF_{ab}(\cC,W))$$ is a localization, and its kernel is generated by short exact sequences
in $MF_{ab}(\cC,W).$ Similarly, by definition, the natural functor
$$K(\cC,W)\to D^{abs}(\cC,W)$$ is also a localization, and its kernel is generated by totalizations of short exact sequences
in $MF_{ab}(\cC,W).$ Altogether, we have a commutative square of exact functors
$$
\begin{CD}
D^b(MF_{ex}(\cC,W))@>q_1 >> D^b(MF_{ab}(\cC,W))\\
@V\Tot VV @V\Tot VV\\
K(\cC,W) @>q_2 >> D^{abs}(\cC,W).
\end{CD}
$$
Both horizontal arrows and the left vertical arrow are localizations, hence so is the right vertical arrow. Its kernel equals to $q_1(\ker \Tot)$ and is therefore generated by $Cone(\id_F),$ $F\in MF_{ab}(\cC,W).$ This proves 2).

3) We claim that the totalization functor has a right adjoint
$$R:K(\cC,W)\to D^{co}(MF_{ex}(\cC,W)),$$
$R(F)^n=Cone(\id_{F[n]}),$ and the differential is the composition
$$Cone(\id_{F[n]})\to F[n+1]\to Cone(\id_{F_[n+1]}).$$ In other words, $R$ is the composition
$$K(\cC,W)\cong \un{MF_{ex}(\cC,W)}\cong K_{ac}(\Inj MF_{ex}(\cC,W))\hookrightarrow D^{co}(MF_{ex}(\cC,W)).$$
To define the adjunction counit, let us note that $\Tot^{\oplus}(R(F))=Cone(\bigoplus\limits_{n\in\Z}F_{(n)}\xto{\phi} \bigoplus\limits_{n\in\Z}F_{(n)}),$ where $F_{(n)}$ denotes the $n$-th copy of $F,$ and $\phi$ has components $\id:F_{(n)}\to F_{(n)}$ and $-\id:F_{(n)}\to F_{(n-1)}.$ Note that $\phi$ is a split monomorphism, and $\coker(\phi)=F.$ The adjunction counit is given by the projection
$\Tot^{\oplus}R(F)\to \coker(\phi)=F.$
To define the adjunction unit, we take a complex $G^{\bullet}\in D^{co}(MF_{ex}(\cC,W)),$ and define the map $G^n\to R(\Tot^{\oplus}(G^{\bullet}))^n$ to be the inclusion of a direct summand.

To see that the counit $\Tot^{\oplus}R(F)\to F$ is an isomorphism in $K(\cC,W)$, we note that
$$Cone(Tot^{\oplus}R(F)\to F)=\Tot^{\oplus}(\bigoplus\limits_{n\in\Z}F_{(n)}\xto{\phi} \bigoplus\limits_{n\in\Z}F_{(n)}\to F),$$
and the RHS is a totalization of short exact sequence in $MF_{ex}(\cC,W),$ hence zero object in $K(\cC,W).$


To control the kernel of $\Tot^{\oplus},$
it suffices to show that for any $G\in D^{co}(MF_{ex}(\cC,W))$ the object
$$Cone(G\to R\Tot^{\oplus}(G))$$ is contained in the localizing subcategory generated by $Cone(\id_F).$ Since the category $D^{co}(MF_{ex}(\cC,W))$
is generated (as a localizing subcategory) by objects concentrated in degree zero, we may and will assume that $G=G[0]$ is actually an object of $MF_{ex}(\cC,W),$
Further, the unit morphism $G[0]\to R(G)$ is a component-wise inflation, so we have $$Cone(G[0]\to R(G))\cong R(G)/G[0]=R(G)^{\leq -1}\oplus \tau^{\geq 0}R(G).$$
Here $R(G)^{\leq -1}$ is a bounded above complex, and its terms are $Cone(\id_{G[-n]})$ for $n\geq 1.$ Hence, it is contained in the localizing
subcategory of $D^{co}(MF_{ex}(\cC,W)))$ generated by $Cone(\id_F).$ Further, the complex $\tau^{\geq 0}R(G)$ is acyclic and bounded below, hence it is zero in $D^{co}(MF_{ex}(\cC,W)).$ This proves 3).

4) This follows from 3) in the same way as 2) follows from 1).
\end{proof}


For an abelian category $\cC,$ we have exact functors $$\Phi_0,\Phi_1:MF_{ab}(\cC,W)\to \cC,$$ where $\Phi_0(F)=F^{ev}$ (resp. $\Phi_1(F)=F^{odd}$) is the even (resp. odd) component of the underlying $\Z/2$-graded object of $\cC$ (keeping the notation of Definition \ref{defi:MF}). They have both left and right adjoints. Namely, for $G\in\cC,$ put $$\Psi_0(G):=(\id_G,W_G)\in MF_{ab}(\cC,W),\quad \Psi_1(G):=(W_G,\id_G).$$ It is easy to check that the functor $\Psi_0$ is left adjoint to $\Phi_0$ and right adjoint to $\Phi_1.$ Similarly, the functor $\Psi_1$ is left adjoint to $\Phi_1$ and right adjoint to $\Phi_0.$

If $\cC$ is a locally noetherian abelian category, then the abelian category $MF_{ab}(\cC,W)$ is also locally noetherian, and $MF_{ab}(\cC,W)_f=MF_{ab}(\cC_f,W).$

Given abelian $\Z_+$-categories $(\cC_1,W_1), (\cC_2,W_2)$ with exact filtered colimits, and a left exact $\Z_+$-functor
$F:\cC_1\to\cC_2,$ the right derived functor $\bR F:D^{co}(\cC_1,W_1)\to D^{co}(\cC_2,W_2)$  (if it exists) is defined by the standard universal property. In the case when $\cC_1$ is locally noetherian, the functor $\bR F$ is given by the composition
$$D^{co}(\cC_1,W_1)\xto{\sim} K(\Inj \cC_1,W_1)\xto{F}D^{co}(\cC_2,W_2).$$
In general, one can use Proposition \ref{prop:Tot_localizations} to construct the (left and right) derived functors for matrix factorizations using the adapted classes of objects. A bit different but conceptually similar approach has been developed in \cite{BDFIK}.

Proposition \ref{prop:Tot_localizations} provides a very efficient tool for reducing the statements and constructions for exotic derived categories matrix factorizations to similar statements for conventional derived categories of abelian categories (see e.g. \cite[Remark 2.7]{EP}). Here we only mention several applications which we need in the present paper. 


\begin{prop}\label{prop:D^abs_to_D^abs}Let $(\cC_1,W_1),$ $(\cC_2,W_2)$ be locally noetherian abelian categories, and $F:\cC_1\to\cC_2$ a coproduct preserving left exact $\Z_+$-functor.
Assume that the derived functor
$$\bR F:D^{co}(\cC_1)\to D^{co}(\cC_2)$$ takes the subcategory $D^b_f(\cC_1)$ to $D^b_f(\cC_2).$ Then the induced functor
$$\bR F:D^{co}(\cC_1,W_1)\to D^{co}(\cC_2,W_2)$$ takes the subcategory $D^{abs}(\cC_{1 f},W_1)$ to $D^{abs}(\cC_{2 f},W_2).$\end{prop}

\begin{proof}We have an induced functor $F_{MF}:MF_{ab}(\cC_1,W_1)\to MF_{ab}(\cC_2,W_2).$ It is left exact and coproduct preserving. It has a derived functor $$\bR F_{MF}:D^{co}(MF_{ab}(\cC_1,W_1))\to D^{co}(MF_{ab}(\cC_2,W_2)).$$ We claim that this functor takes $D^b_f(MF_{ab}(\cC_1,W_1))$ to
$D^b_f(MF_{ab}(\cC_2,W_2)).$ Indeed, first we may replace $D^{co}$ by $D^{+}.$ Further, an object $G\in D^{+}(MF_{ab}(\cC_2,W_2))$ belongs to
$D^b_f(MF_{ab}(\cC_2,W_2))$ iff $\Phi_0(G)$ and $\Phi_1(G)$ belong to $D^b_f(\cC_2).$ It remains to note that we have commutative diagrams for $i=0,1:$
\begin{equation}\label{eq:commutation_with_Phi_i}
\begin{CD}
D^+(MF_{ab}(\cC_1,W_1)) @>\bR F_{MF} >> D^+(MF_{ab}(\cC_1,W_1))\\
@V \Phi_i VV @V \Phi_i VV\\
D^+(\cC_1) @> \bR F >>   D^+(\cC_2).
\end{CD}
\end{equation}

Now we apply the totalization. Namely, we have a commutative diagram
$$
\begin{CD}
D^{co}(MF_{ab}(\cC_1,W_1)) @>\bR F_{MF} >> D^{co}(MF_{ab}(\cC_1,W_1))\\
@V \Tot^{\oplus} VV @V \Tot^{\oplus} VV\\
D^{co}(\cC_1,W_1) @> \bR F >>   D^{co}(\cC_2,W_2).
\end{CD}
$$
The functors $\Tot^{\oplus}$ take $D^b_f(MF_{ab}(\cC_i,W_i))$ to $D^{abs}(\cC_i,W_i).$ This proves the proposition.\end{proof}



For a separated noetherian $\mk$-scheme $X$ and $W\in\cO(X),$ we have a bifunctor
$$\un{\Hom}_{ab}:MF_{ab}(\Coh X,W)^{op}\times MF_{ab}(\QCoh X,W)\to \Mod^{\Z/2}\mhyphen \mk:=MF_{ab}(\Mod\mhyphen \mk,0),$$
which is left exact in both arguments. We also have its derived functor
$$\bR\un{\Hom}_{ab}:D^b(MF_{ab}(\Coh X,W))^{op}\times D^+(MF_{ab}(\QCoh X,W))\to D^+(\Mod^{\Z/2}\mhyphen \mk).$$
We have a natural isomorphism
\begin{equation}\label{eq:Hom_MF_total}
\Tot^{\oplus}(\bR\un{\Hom}_{ab}(\cF,\cG))\cong\bR\Hom(\Tot(\cF),\Tot^{\oplus}(\cG)),\end{equation} 
for $\cF\in D^b(MF_{ab}(\Coh X,W)),$ $\cG\in D^+(MF_{ab}(\QCoh X,W)).$

Also, note that the bifunctor $$-\otimes-:\Mod^{\Z/2}\mhyphen \mk\times \Mod^{\Z/2}\mhyphen \mk\to \Mod^{\Z/2}\mhyphen \mk$$ is biexact, hence it induces a tensor product functor on $D^+(\Mod^{\Z/2}\mhyphen \mk).$

\begin{prop}\label{prop:box_product_Morita}
Given two schemes $X_1$ and $X_2$ of finite type over a perfect field, and regular functions $W_i$ on $X_i,$ the box product quasi-functor provides a Morita equivalence
\begin{equation}\label{eq:boxtimes_functor}-\boxtimes-:\mD^{abs}(X_1,W_1)\otimes \mD^{abs}(X_2,W_2)\to \mD^{abs}_{X_1^0\times X_2^0}(X_1\times X_2,W_1\boxplus W_2).\end{equation}\end{prop}

\begin{proof}We first show that the functor \eqref{eq:boxtimes_functor} is quasi-fully-faithful. By \cite[Proposition 6.20 a)]{L}, for any $F_1,G_1\in D^b_{coh}(X_1)$ and $F_2,G_2\in D^b_{coh}(X_2),$ we have \begin{equation}\label{eq:Lunts_Kunneth}\bR\Hom(F_1,G_1)\otimes\bR\Hom(F_2,G_2)\cong\bR\Hom(F_1\boxtimes F_2,G_1\boxtimes G_2)\text{ in }D^+(k).\end{equation} Now let us take any $\cF_1,\cG_1\in MF_{ab}(\Coh X_1,W_1),$ $\cF_2,\cG_2\in MF_{ab}(\Coh X_2,W_2).$ It follows from \eqref{eq:Lunts_Kunneth} and \eqref{eq:commutation_with_Phi_i} that we have
$$\bR\un{\Hom}_{ab}(\cF_1,\cG_1)\otimes \bR\un{\Hom}_{ab}(\cF_2,\cG_2)\cong \bR\un{\Hom}_{ab}(\cF_1\boxtimes \cF_2,\cG_1\boxtimes \cG_2)\text{ in }D^+(\Mod^{\Z/2}\mhyphen \mk).$$ Applying $\Tot^{\oplus},$ we obtain an isomorphism
$$\bR\Hom(\bar{\cF_1},\bar{\cG_1})\otimes\bR\Hom(\bar{\cF_2},\bar{\cG_2})\cong \bR\Hom(\bar{\cF_1}\boxtimes\bar{\cF_2},\bar{\cG_1}\boxtimes\bar{\cG_2}).$$
Therefore, the DG quasi-functor \eqref{eq:boxtimes_functor} is quasi-fully-faithful.

It remains to show that the objects $\cF_1\boxtimes\cF_2$ generate $D^{abs}_{coh,X_1^0\times X_2^0}(X_1\times X_2,W_1\boxplus W_2)$ as triangulated category. By \cite[Proposition 6.8, Proposition 6.14]{L}, the image of the bifunctor $D^b_{coh}(X_1^0)\times D^b_{coh}(X_2^0)\to D^b_{coh}(X_1^0\times X_2^0)$ generates the target as a triangulated category. To finish the proof, it suffices to note that the image of the pushforward functor $D^b_{coh}(X_1^0\times X_2^0)\to D^{abs}_{coh,X_1^0\times X_2^0}(X_1\times X_2,W_1\boxplus W_2)$ also generates the target as a triangulated category. 
\end{proof}

We finish this section with the following general statement on localization for categories of matrix factorizations

\begin{prop}\label{prop:Serre_quotient_MF}Let $(\cA,W)$ be an abelian $\Z_+$-category, and $\cB\subset\cA$ a Serre subcategory. Put $\cC:=\cA/\cB.$ Then $W$ induces a natural transformation $\bar{W}:\id_{\cC}\Rightarrow\id_{\cC},$ and the induced functor $D^{abs}(\cA,W)\to D^{abs}(\cC,\bar{W})$ is a localization. Its kernel is generated by the image of $D^{abs}(\cB,W).$\end{prop}

\begin{proof}The first assertion (about $\bar{W}$) is evident. It is easy to see that $MF_{ab}(\cB,W)\subset MF_{ab}(\cA,W)$ is a Serre subcategory. Let us check that the natural functor  $\pr:MF_{ab}(\cA,W)/MF_{ab}(\cB,W)\to MF_{ab}(\cC,\bar{W})$ is an equivalence of abelian categories.

{\it Faithfulness.} Suppose that $\pr(\bar{\varphi})=0$ for some morphism $\bar{\varphi}:\bar{F}\to\bar{F'}$ in $MF_{ab}(\cA,W)/MF_{ab}(\cB,W).$ We may and will assume that $\bar{\varphi}$ is the image of a morphism $\varphi:F\to F'$ in $MF_{ab}(\cA,W).$ Then $\im(\varphi^{ev}), \im(\varphi^{odd})\in\cB,$ hence $\im(\varphi)\in MF_{ab}(\cB,W).$ Thus, $\bar{\varphi}=0.$

{\it Fullness.} Suppose that we have a morphism $\psi:\pr(\bar{F})\to \pr(\bar{G})$ in $MF_{ab}(\cC,\bar{W}),$ where $F,G\in MF_{ab}(\cA,W).$ Then the morphisms $\psi^{ev},\psi^{odd}$ can be represented by some morphisms $(\varphi^{ev})':(F^{ev})'\to G^{ev}/(G^{ev})',$ $(\varphi^{odd})':(F^{odd})'\to G^{odd}/(G^{odd})',$ where $(F^{ev})'\subset F^{ev}$ (resp. $(G^{ev})'\subset G^{ev}$) is a subobject s.t. $F^{ev}/(F^{ev})'$ (resp. $(G^{ev})'$) is in $\cB,$ and similarly for $(F^{odd})'$ (resp. $(G^{odd})'$). Let us replace $(F^{ev})'$ by the intersection $(F^{ev})'\cap (\delta_F^{0})^{-1}((F^{odd})')=:(F^{ev})''.$ Then the quotient $F^{ev}/(F^{ev})''$ is still in $\cB$ (since $\cB$ is a Serre subcategory), and $\delta^{0}((F^{ev})'')\subseteq (F^{odd})'.$ Further, we replace $(F^{odd})'$ by the intersection $(F^{odd})'\cap (\delta_F^{1})^{-1}((F^{ev})'').$ Again, the quotient  $F^{odd}/(F^{odd})''$ is still in $\cB,$ and $\delta^{1}((F^{odd})'')\subseteq (F^{ev})''.$ But we also have the inclusion $\delta^0((F^{ev})'')\subseteq (F^{odd})'',$ since $\delta^1\delta^0((F^{ev})'')=W_{F^{ev}}((F^{ev})'')\subseteq (F^{ev})''.$ Therefore, $((F^{ev})'',(F^{odd})'')\subset F$ is a subfactorization. Analogously, we construct subobjects $(G^{ev})''\subseteq G^{ev},$ $(G^{odd})''\subseteq G^{odd},$ containing $(G^{ev})'$ (resp. $(G^{odd})'$) such that $(G^{ev})'',(G^{odd})''\in\cB,$ and $((G^{ev})'',(G^{odd})'')\subset G$ is a subfactorization. We denote by $(\varphi^{ev})'':(F^{ev})''\to G^{ev}/(G^{ev})''$ (resp. $(\varphi^{odd})'':(F^{odd})''\to G^{odd}/(G^{odd})''$) the morphism induced by $(\varphi^{ev})'$ (resp. $(\varphi^{odd})'$).

We denote by $\varphi'':F''\to G/G''$ the constructed morphism of $\Z/2$-graded objects in $\cA.$ By construction, the commutator $[\delta,\varphi'']$ vanishes in $\cC.$ Therefore, $\im([\delta,\varphi''])$ is a ($\Z/2$-graded) object of $\cB.$ Moreover, $\im([\delta,\varphi''])\subseteq G/G''$ is a subfactorization (since $[\delta,\varphi'']:F''[-1]\to G/G''$ is a morphism in $MF_{ab}(\cA,W)$). Therefore, we have an actual morphism $\varphi:F''\to \coker([\delta,\varphi''])$ in $MF(\cA,W),$ and $\pr(\bar{\varphi})=\psi.$ This proves fullness of $\pr.$

{\it Essential surjectivity.} Let us take some matrix factorization $\bar{F}=(\overline{F^{ev}},\overline{F^{odd}},\bar{\delta^0},\bar{\delta^1})$ in $MF(\cC,\bar{W}).$ We may and will assume that $\bar{\delta^1}$ is represented by an actual morphism $F^{odd}\xto{\delta^1} F^{ev}$ in $\cA.$ The morphism $(\bar{\delta^0})$ is represented by some $(\delta^0)':(F^{ev})'\to F^{odd}/(F^{odd})',$ where $F^{ev}/(F^{ev})',(F^{odd})'\in\cB.$ We can replace $F^{odd}$ by $F^{odd}/(F^{odd})',$ and $F^{ev}$ by $F^{ev}/\delta^1((F^{odd})'),$ and $(F^{ev})'$ by $(F^{ev})'/(F^{ev})'\cap\delta^1((F^{odd})')).$ Thus, we may and will assume that $(F^{odd})'=0.$ Similarly, we may and will assume that $(F^{ev})'=F^{ev}.$ So we have a data $(F^{ev},F^{odd},\delta^0,\delta^1)$ such that $\im(\delta^1\delta^0-W_{F^{ev}}),\im(\delta^0\delta^1-W_{F^{odd}})$ are in $\cB.$ Replacing $F$ by $\tilde{F}:=F/\im((\delta)^2-W_F),$ we obtain an object of $MF(\cA,W),$ and $\pr(\tilde{F})=\bar{F}.$ 

It follows that we have an equivalence of triangulated categories
$$D^b(MF_{ab}(\cA,W))/D^b_{MF_{ab}(\cB,W)}(MF_{ab}(\cA,W))\xto{\sim}D^b(MF_{ab}(\cC,\bar{W})).$$
In the commutative diagram
$$
\begin{CD}
D^b(MF_{ab}(\cA,W)) @>>> D^b(MF_{ab}(\cC,\bar{W}))\\
@V\Tot_{\cA}VV @V\Tot_{\cC}VV\\
D^{abs}(\cA,W) @>>> D^{abs}(\cC,\bar{W})
\end{CD}
$$
both vertical arrows and the upper horizontal arrow are localizations, hence so is the lower horizontal arrow. Its kernel is generated by $\Tot_{\cA}(D^b_{MF_{ab}(\cB,W)}(MF_{ab}(\cA,W)))=\im(D^{abs}(\cB,W)\to D^{abs}(\cA,W)).$ This proves the proposition. 
\end{proof}

\section{Grothendieck duality}
\label{app:duality}

Here we prove relative Grothendieck duality for coherent sheaves and matrix factorizations on nice ringed spaces.

First we need one general auxiliary result.

\begin{prop}\label{prop:D^+toD^+}Let $\cC_1,$ $\cC_2$ be locally noetherian abelian categories, and $\Phi:\cC_1\to\cC_2$ a coproduct preserving left exact functor. Suppose that its right derived functor $\bR^n \Phi$ vanishes for some $n>0$ (hence also $R^m \Phi=0$ for $m>n$).

Take the derived functor $\Phi:D^{co}(\cC_1)\to D^{co}(\cC_2).$ Then its right adjoint functor
$\Phi^!:D^{co}(\cC_2)\to D^{co}(\cC_1)$ takes $D^{+}(\cC_2)$ to $D^+(\cC_1).$\end{prop}

\begin{proof}First, if $\cC$ is a locally noetherian abelian category, then an object $F\in D^{co}(\cC)$ belongs to $D^{\geq d}(\cC)$ iff for any
$G\in\cC_f$ one has $\Hom^i(G,F)=0$ for $i\leq d-1.$

Now, suppose that $E\in D^{\geq l}(\cC_2).$ For any $F$ in $\cC_1$ we have (by assumption) $\bR \Phi(F)\in D^{b,\leq n}(\cC_2).$ Hence, for $i\leq l-n-1$ we have $$\Hom^i(F,\Phi^! E)=\Hom^i(\bR\Phi(F),E)=0.$$ Thus, $\Phi^!F\in D^{\geq (l-n)}(\cC_1).$ This proves Proposition.\end{proof}

Let $(X,\cA_X)$ be a nice ringed space. Suppose that $\cB_X$ is another coherent sheaf of $\cO_X$-algebras, so that $(X,\cB_X)$ is a nice ringed space. We have a natural bi-functor
$$\bR\cHom_{\cA_X}:D^b(\Coh\mhyphen \cA_X\tens{\cO_X}\cB_X^{op})\times D^{co}(\QCoh\mhyphen \cA_X)\to D^{co}(\QCoh\mhyphen \cB_X).$$

The proof of the following theorem follows the lines of \cite[Section 6]{N}.

\begin{theo}\label{th:Groth_dual_coh}Let $(f,\phi):(X,\cA_X)\to (Y,\cA_Y),$ $(f,\psi):(X,\cB_X)\to (Y,\cB_Y)$ be two morphisms of nice ringed spaces with the same underlying proper morphism $f:X\to Y.$ We have also the morphism $(f,\phi\otimes\psi^{op}):(X,\cA_X\tens{\cO_X}\cB_X^{op})\to (Y,\cA_Y\tens{\cO_Y}\cB_Y^{op}).$ Then for $\cF\in D^b(\Coh\mhyphen \cA_X\tens{\cO_X}\cB_X^{op}),$ $\cG\in D^{co}(\QCoh\mhyphen \cA_Y)$ we have a bi-functorial isomorphism
$$f_*\bR\cHom_{\cA_X}(\cF,f^!\cG)\stackrel{\sim}{\to} \bR\cHom_{\cA_Y}(f_*\cF,\cG)$$
in $D^{co}(\QCoh\mhyphen \cB_Y).$\end{theo}
\begin{proof}First, the required morphism of bi-functors is defined to be the composition
\begin{equation}\label{eq:appendixf_*f^!}f_*\bR\cHom_{\cA_X}(\cF,f^!\cG)\to \bR\cHom_{\cA_Y}(f_*\cF,f_*f^!\cG)\to \bR\cHom_{\cA_Y}(f_*\cF,\cG).\end{equation}
Since all the bifunctors in \eqref{eq:appendixf_*f^!} are continuous in $\cG,$ we may assume that $\cG$ is a compact object in $D^{co}(\QCoh\mhyphen \cA_Y),$
i.e. $\cG\in D^b(\Coh\mhyphen \cA_Y).$

We will assume that more generally $\cG\in D^+(\QCoh\mhyphen \cA_Y)\subset D^{co}(\QCoh\mhyphen \cA_Y).$ Then by Proposition \ref{prop:D^+toD^+} all the objects in \eqref{eq:appendixf_*f^!} are in $D^+(\QCoh\mhyphen \cB_Y)\subset D^{co}(\QCoh\mhyphen \cB_Y).$ Hence, it suffices to show that \eqref{eq:appendixf_*f^!} becomes an isomorphism after applying $\bR\Gamma(U,-)$ for any open $U\subset Y.$ Clearly, this holds for $U=Y$ (since $\bR\Gamma(\bR\cHom)\cong\bR\Hom$).

Let us take some open $U\subset Y,$ and denote by $j_U:U\to Y$ the corresponding open immersion, and similarly $j_{f^{-1}(U)}:f^{-1}(U)\to X.$ Also, we put $f^\prime:=f_{\mid f^{-1}(U)}:f^{-1}(U)\to U.$ It follows from Proposition \ref{prop:enough_inj} 3) that we have a natural isomorphism
$$f^\prime_*j_{f^{-1}(U)}^*\cong j_U^*f_*:D^{co}(\QCoh\mhyphen \cA_X)\to D^{co}(\QCoh\mhyphen \cA_U).$$ Passing to the right adjoints, we get an isomorphism
\begin{equation}\label{eq:app_commuting}j_{f^{-1}(U)*}(f^\prime)^!\cong f^!j_{U*}:D^{co}(\QCoh\mhyphen \cA_U)\to D^{co}(\QCoh\mhyphen \cA_X).\end{equation}
Now, we have a morphism of functors
\begin{equation}\label{eq:restriction_f^!}j_{f^{-1}(U)}^*f^!\to (f^\prime)^!j_U^*,\end{equation} which is a composition
$$j_{f^{-1}(U)}^*f^!\to j_{f^{-1}(U)}^*f^!j_{U*}j_U^*\cong j_{f^{-1}(U)}^*j_{f^{-1}(U)*}(f^\prime)^!j_U^*\cong (f^\prime)^!j_U^*.$$

We are left to show that \eqref{eq:restriction_f^!} is an isomorphism. This in turn equivalent to the following statement.

{\noindent{\bf Claim.}} {\it We have the following vanishing: $(j_{f^{-1}(U)}^*f^!)_{\mid D^{co}_{Y\setminus U}(\QCoh\mhyphen \cA_Y)}=0.$}

\begin{proof} This vanishing is local on $Y$ for the following reason. Let $Y=V_1\cup V_2$ be an open cover. Then for any $\cH\in D^{+}_{Y\setminus U}(\QCoh\mhyphen \cA_Y)$ we have the Mayer-Vietoris triangle
$$\cH\to j_{V_1 *}j_{V_1}^*\cH\to j_{V_2 *}j_{V_2}^*\cH\to j_{V_1\cap V_2 *}j_{V_1\cap V_2}^*\cH\to\cH[1].$$
Combining this with the isomorphism \eqref{eq:app_commuting} we reduce the problem to the case when $Y$ is affine.

Put $Z:=Y\setminus U.$ Clearly, if $Z=Z_1\cap\dots\cap Z_n,$ then it suffices to prove Claim for each $Z_i.$ Hence, we may and will assume that $Y$ is affine and $Z\subset Y$ is a divisor defined by the function $\gamma$ on $Y.$
In this case the category $D^{co}_{Y\setminus U}(\QCoh\mhyphen \cA_Y)$ is compactly generated by the objects $Cone(\cG\stackrel{\gamma}{\to}\cG),$ where $\cG\in\Coh\mhyphen \cA_Y.$ But we have
$$j_{f^{-1}(U)}^*f^!Cone(\cG\stackrel{\gamma}{\to}\cG)=j_{f^{-1}(U)}^*Cone(f^!\cG\stackrel{f^*(\gamma)}{\to}f^!\cG)=0.$$
This proves Claim.
\end{proof}
Claim implies that \eqref{eq:restriction_f^!} is an isomorphism. Finally, we have
\begin{multline*}\bR\Gamma(U,f_*\bR\cHom_{\cA_X}(\cF,f^!\cG))\cong \bR\Hom_{\cA_{f^{-1}(U)}}(j_{f^{-1}(U)}^*\cF,j_{f^{-1}(U)}^*f^!\cG)\\ \cong \bR\Hom_{\cA_{f^{-1}(U)}}(j_{f^{-1}(U)}^*\cF,f{'!}j_{U}^*\cG)\cong \bR\Hom_{\cA_U}(f'_*j_{f^{-1}(U)}^*\cF,j_{U}^*\cG)\\ 
\cong \bR\Hom_{\cA_U}(j_{U}^*f_*\cF,j_{U}^*\cG)\cong \bR\Gamma(U,\bR\cHom_{\cA_X}(f_*\cF,\cG)).\end{multline*}
This proves the theorem.
\end{proof}

\begin{remark}In the proof of Theorem \ref{th:Groth_dual_coh} we use implicitly that for a nice ringed space $(S,\cA_S)$ the natural functor
$D^{co}(\QCoh\mhyphen \cA_S)\to D^{co}(\Mod\mhyphen \cA_S)$ is fully faithful. For example, if $\cE\in D^+(\QCoh\mhyphen \cA_Y\tens{\cO_Y}\cB_Y^{op}),$ $\cH\in D^{co}(\QCoh\mhyphen \cA_Y)$ then $$\bR\cHom(\cE,\cH)\in D^{co}(\Mod\mhyphen \cB_Y),$$ but for $\cE\in D^b_{coh}(\QCoh\mhyphen \cA_Y\tens{\cO_Y}\cB_Y^{op})$ we have
that $\bR\cHom(\cE,\cH)$ is contained in the essential image of $D^{co}(\QCoh\mhyphen \cB_Y).$\end{remark}

\begin{cor}\label{cor:projection_formula_coh}Let $f:(X,\cA_X)\to (Y,\cA_Y)$ be a proper morphism of nice ringed spaces. Then there is a natural isomorphism of functors
$$f_*f^!\cong \bR\cHom_{\cA_Y}(\bR f_*\cA_X,-):D^{co}(\QCoh\mhyphen \cA_Y)\to D^{co}(\QCoh\mhyphen \cA_Y).$$
Here $\cA_X$ is considered as a coherent $(\cA_X\tens{\cO_X}\cA_X^{op})$-module.\end{cor}

\begin{proof}Indeed, it suffices to apply Theorem \ref{th:Groth_dual_coh} to $\cB_X=\cA_X,$ $\cB_Y=\cA_Y$ and $\cF=\cA_X.$\end{proof}

Now we turn to the case of matrix factorizations. Again, suppose that $(X,\cA_X)$ and $(X,\cB_X)$ are nice ringed spaces with the same underlying scheme $X.$ Let $W_1, W_2$ be regular functions on $X.$ Then we have a natural bi-functor
\begin{multline*}\bR\cHom_{\cA_X}:D^b(\Coh\mhyphen (\cA_X\tens{\cO_X}\cB_X^{op},W_1))\times D^{co}(\QCoh\mhyphen (\cA_X,W_2))\\\to D^{co}(\QCoh\mhyphen (\cB_X,W_2-W_1)).\end{multline*}

\begin{theo}\label{th:Groth_dual_MF}Keep the notation of Theorem \ref{th:Groth_dual_coh}. Let $W_1, W_2$ be regular functions on $Y.$ Then for $\cF\in D^{abs}_{coh}(\cA_X\tens{\cO_X}\cB_X^{op},f^*W_1),$ $\cG\in D^{co}(\QCoh\mhyphen (\cA_Y,W_2))$ we have a bi-functorial isomorphism
$$f_*\bR\cHom_{\cA_X}(\cF,f^!\cG)\stackrel{\sim}{\to} \bR\cHom_{\cA_Y}(f_*\cF,\cG)$$
in $D^{co}(\QCoh\mhyphen (\cB_Y,W_2-W_1)).$\end{theo}

\begin{proof}
For $\cF\in D^{abs}_{coh}(\cA_X\tens{\cO_X}\cB_X^{op},f^*W_1),$ $\cG\in D^{co}(\QCoh\mhyphen (\cA_Y,W_2))$
we have a natural composition
\begin{equation}\label{eq:appendixMFf_*f^!}f_*\bR\cHom_{\cA_X}(\cF,f^!\cG)\to \bR\cHom_{\cA_Y}(f_*\cF,f_*f^!\cG)\to \bR\cHom_{\cA_Y}(f_*\cF,\cG).\end{equation}
We need to show that the composition \eqref{eq:appendixMFf_*f^!} is an isomorphism in $D^{co}(\QCoh\mhyphen (\cB_Y,W_2-W_1)).$

We first replace the categories $D^{abs}_{coh}(\cA_X\tens{\cO_X}\cB_X^{op},f^*W_1)$ and $D^{co}(\QCoh\mhyphen (\cA_Y,W_2))$ by the categories
$D^b(MF_{ab}(\Coh\mhyphen \cA_X\tens{\cO_X}\cB_X^{op},f^*W_1))$ and $D^{co}(MF_{ab}(\QCoh\mhyphen \cA_Y,W_2)).$ We have the direct image functors on (co)derived categories:
$$f_*:D^b(MF_{ab}(\Coh\mhyphen \cA_X\tens{\cO_X}\cB_X^{op},f^*W_1))\to D^b(MF_{ab}(\Coh\mhyphen \cA_Y\tens{\cO_Y}\cA_Y^{op},W_1)),$$
$$f_*:D^{co}(MF_{ab}(\QCoh\mhyphen \cA_X,f^*W_2))\to D^{co}(MF_{ab}(\QCoh\mhyphen \cA_Y,W_2)).$$
We claim that for any $\tilde{\cF}\in D^b(MF_{ab}(\Coh\mhyphen (\cA_X\otimes_{\cO_X}\cB_X^{op}),f^*W_1)),$ $\tilde{\cG}\in D^{co}(MF_{ab}(\QCoh\mhyphen \cA_Y,W_2))$ the natural composition
\begin{equation}\label{eq:appendixMFabf_*f^!}f_*\bR\cHom_{\cA_X}(\tilde{\cF},f^!\cG)\to \bR\cHom_{\cA_Y}(f_*\tilde{\cF},f_*f^!\tilde{\cG})\to \bR\cHom_{\cA_Y}(f_*\tilde{\cF},\tilde{\cG})\end{equation}
is an isomorphism in $D^{co}(MF_{ab}(\QCoh\mhyphen \cB_Y,W_2-W_1)).$

As in  the proof of Theorem \ref{th:Groth_dual_coh} we easily reduce to the case when $\tilde{\cG}\in D^+(MF_{ab}(\QCoh\mhyphen \cA_Y,W_2)).$ So, we will assume that this is the case. Then, by Proposition \ref{prop:D^+toD^+} we also have that $f^!\tilde{\cG}\in D^+(MF_{ab}(\QCoh\mhyphen \cA_X,f^*W_2)).$

Recall that in Appendix \ref{app:MF} we introduced the functors $\Phi_0,$ $\Phi_1,$ $\Psi_0,$ $\Psi_1.$ It is clear that the functor $f_*$ in each case commutes with $\Phi_0,$ $\Phi_1,$ $\Psi_0,$ $\Psi_1.$ Passing to right adjoints, we see that the functor $f^!$ also commutes with $\Phi_0,$ $\Phi_1,$ $\Psi_0,$ $\Psi_1.$ Further, we have
\begin{equation}\label{eq:Phi_0_of_RHom}\Phi_0(\bR\cHom(-,?))\cong \bR\cHom(\Phi_0(-),\Phi_0(?))\oplus \bR\cHom(\Phi_1(-),\Phi_1(?)),\end{equation}
\begin{equation}\label{eq:Phi_1_of_RHom}\Phi_1(\bR\cHom(-,?))\cong \bR\cHom(\Phi_0(-),\Phi_1(?))\oplus \bR\cHom(\Phi_1(-),\Phi_0(?)).\end{equation}
Combining \eqref{eq:Phi_0_of_RHom}-\eqref{eq:Phi_1_of_RHom} with Theorem \ref{th:Groth_dual_coh}, for any $\tilde{\cF}\in D^b(MF_{ab}(\Coh\mhyphen \cA_X\tens{\cO_X}\cA_X^{op},f^*W_1)),$ $\tilde{\cG}\in D^{+}(MF_{ab}(\QCoh\mhyphen \cA_Y,W_2))$ we obtain the isomorphisms
\begin{equation}\label{eq:isomorphism_for_Phi_0}\Phi_0(f_*\bR\cHom_{\cA_X}(\tilde{\cF},f^!\tilde{\cG}))\cong \Phi_0(\bR\cHom_{\cA_Y}(f_*\tilde{\cF},\tilde{\cG})),\end{equation}
\begin{equation}\label{eq:isomorphism_for_Phi_1}\Phi_1(f_*\bR\cHom_{\cA_X}(\tilde{\cF},f^!\tilde{\cG}))\cong \Phi_1(\bR\cHom_{\cA_Y}(f_*\tilde{\cF},\tilde{\cG}))\end{equation}
in $D^{+}(\QCoh\mhyphen \cB_Y).$ The isomorphism \eqref{eq:isomorphism_for_Phi_0} (resp. \eqref{eq:isomorphism_for_Phi_1}) is exactly the image of \eqref{eq:appendixMFabf_*f^!} under the functor $\Phi_0$ (resp. $\Phi_1$).

Summarising, we see that the morphism \eqref{eq:appendixMFabf_*f^!} 
becomes an isomorphism after applying $\Phi_0$ and $\Phi_1.$ Hence, \eqref{eq:appendixMFabf_*f^!} is itself an isomorphism.

Now, we prove that \eqref{eq:appendixMFf_*f^!} is an isomorphism. Choose any $\tilde{\cF},$ $\tilde{\cG}$ such that $\Tot(\tilde{\cF})\cong\cF,$ $\Tot^{\oplus}(\tilde{\cG})=\cG.$ Clearly, the functors $f_*$ commute with $\Tot$ and $\Tot^{\oplus}.$ In particular, they preserve the kernels of $\Tot$ and $\Tot^{\oplus}.$

We claim that $f^!$ also commutes with $\Tot^{\oplus}.$ Indeed, since $\Tot^{\oplus}$ is a localization, it suffices to show
that $f^!$ preserves the kernel of $\Tot^{\oplus}.$ By Proposition \ref{prop:Tot_localizations}, this kernel is generated by $Cone(\id_{\cH})$
 for all $\cH\in MF_{ab}(\QCoh\mhyphen \cA_Y,W_2).$ It is easy to see that $Cone(\id_{\cH})=\Psi_0(\cH^{odd})\oplus\Psi_1(\cH^{ev}).$ But $f^!$ commutes with
$\Psi_i$ and also commutes with coproducts. Hence, $f^!$ preserves the kernel of $\Tot^{\oplus}.$

Therefore, we have a chain of isomorphisms
\begin{multline*}f_*\bR\cHom_{\cA_X}(\Tot(\tilde{\cF}),f^!\Tot^{\oplus}(\tilde{\cG}))\cong
f_*\bR\cHom_{\cA_X}(\Tot(\tilde{\cF}),\Tot^{\oplus}(f^!\tilde{\cG}))\cong\\ f_*\Tot^{\oplus}(\bR\cHom_{\cA_X}(\tilde{\cF},f^!\tilde{\cG}))\cong \Tot^{\oplus}(f_*\bR\cHom_{\cA_X}(\tilde{\cF},f^!\tilde{\cG}))\cong\Tot^{\oplus}(\bR\cHom_{\cA_Y}(f_*\tilde{\cF},\tilde{\cG}))\cong\\
\bR\cHom_{\cA_Y}(\Tot(f_*\tilde{\cF}),\Tot^{\oplus}(\tilde{\cG}))\cong \bR\cHom_{\cA_Y}(f_*\Tot(\tilde{\cF}),\Tot^{\oplus}(\tilde{\cG}))
\end{multline*}
in $D^{co}(\QCoh\mhyphen (\cB_Y,W_2-W_1)),$
and the composition coincides with the morphism \eqref{eq:appendixMFf_*f^!}. This proves the theorem.\end{proof}

\begin{remark}In the proof of Theorem \ref{th:Groth_dual_MF} we use implicitly that for a nice ringed space $(S,\cA_S)$ and a function $W$ on $S$ the natural functor
$D^{co}(\QCoh\mhyphen (\cA_S,W))\to D^{co}(\Mod\mhyphen (\cA_S,W))$ is fully faithful. For example, if $\cE\in D^{co}(\QCoh\mhyphen (\cA_Y\tens{\cO_Y}\cB_Y^{op},W_1)),$ $\cH\in D^{co}(\QCoh\mhyphen (\cA_Y,W_2))$ then $$\bR\cHom(\cE,\cH)\in D^{co}(\Mod\mhyphen (\cB_Y,W_2-W_1)),$$ but for $\cE\in D^{abs}_{coh}(\cA_Y\tens{\cO_Y}\cB_Y^{op},W_1)$ we have
that $\bR\cHom(\cE,\cH)$ is contained in the essential image of $D^{co}(\QCoh\mhyphen (\cB_Y,W_2-W_1)).$\end{remark}

\begin{cor}\label{cor:projection_formula_MF}Let $f:(X,\cA_X)\to (Y,\cA_Y)$ be a proper morphism of nice ringed spaces, and $W$ a regular functions on $Y.$ Then there is a natural isomorphism of functors
$$f_*f^!\cong \bR\cHom(\bR f_*\cA_X,-):D^{co}(\QCoh\mhyphen (\cA_Y,W))\to D^{co}(\QCoh\mhyphen (\cA_Y,W)).$$
Here $\cA_X$ is considered as an object of $D^{abs}_{coh}(\cA_X\tens{\cO_X}\cA_X^{op},0).$\end{cor}

\begin{proof}Indeed, it suffices to apply Theorem \ref{th:Groth_dual_MF} to $\cB_X=\cA_X,$ $\cB_Y=\cA_Y,$ $W_1=0,$ $W_2=W$ and $\cF=\cA_X.$\end{proof}


\end{document}